\documentclass[12pt,reqno]{amsart}

\usepackage{a4wide}
\usepackage[unicode,final]{hyperref}
\usepackage[T1]{fontenc}    
\usepackage{enumitem}
\usepackage{graphicx}

\usepackage{amsmath}
\usepackage{amssymb}
\usepackage{amsxtra}
\usepackage{amsthm}

\usepackage[silent]{fixme}
\fxsetup{
  status=draft,
  layout={marginclue,noinline}
}
\fxsetface{inline}{\itshape}

\usepackage[mathscr]{eucal} 
\usepackage{mathrsfs}       
\usepackage{slashed}        

\usepackage{mathtools}
\usepackage{extpfeil}
\usepackage{proof}
\usepackage{ifnextok}
\usepackage{xstring}

\usepackage[all]{xy}
\SelectTips{cm}{}
\SilentMatrices
\ifx\pdfoutput\undefined
  \xyoption{dvips}
\fi
\usepackage{tikz-cd}

\newcommand{\ifundef}[1]{\expandafter\ifx\csname#1\endcsname\relax}

\DeclareMathAlphabet{\mathbbe}{U}{bbold}{m}{n}

\makeatletter

\def\re@DeclareMathSymbol#1#2#3#4{%
    \let#1=\undefined
    \DeclareMathSymbol{#1}{#2}{#3}{#4}}

\ifundef{Top}
  \DeclareSymbolFont{tcSyC}{U}{txsyc}{m}{n}
  \SetSymbolFont{tcSyC}{bold}{U}{txsyc}{bx}{n}
  \DeclareFontSubstitution{U}{txsyc}{m}{n}

  \re@DeclareMathSymbol{\Top}{\mathord}{tcSyC}{120}
  \re@DeclareMathSymbol{\Bot}{\mathord}{tcSyC}{121}
\fi

\ifundef{righthalfcup}
  \DeclareFontFamily{U}{MnSymbolC}{}
  \DeclareSymbolFont{mnSyC}{U}{MnSymbolC}{m}{n}
  \SetSymbolFont{mnSyC}{bold}{U}{MnSymbolC}{b}{n}
  \DeclareFontShape{U}{MnSymbolC}{m}{n}{
      <-6>  MnSymbolC5
     <6-7>  MnSymbolC6
     <7-8>  MnSymbolC7
     <8-9>  MnSymbolC8
     <9-10> MnSymbolC9
    <10-12> MnSymbolC10
    <12->   MnSymbolC12}{}
  \DeclareFontShape{U}{MnSymbolC}{b}{n}{
      <-6>  MnSymbolC-Bold5
     <6-7>  MnSymbolC-Bold6
     <7-8>  MnSymbolC-Bold7
     <8-9>  MnSymbolC-Bold8
     <9-10> MnSymbolC-Bold9
    <10-12> MnSymbolC-Bold10
    <12->   MnSymbolC-Bold12}{}

  \re@DeclareMathSymbol{\righthalfcup}{\mathord}{mnSyC}{184}
  \re@DeclareMathSymbol{\lefthalfcap}{\mathord}{mnSyC}{185}
\fi

\DeclareFontFamily{U}{MnSymbolA}{}
\DeclareSymbolFont{mnSyA}{U}{MnSymbolA}{m}{n}
\SetSymbolFont{mnSyA}{bold}{U}{MnSymbolA}{b}{n}
\DeclareFontShape{U}{MnSymbolA}{m}{n}{
    <-6>  MnSymbolA5
   <6-7>  MnSymbolA6
   <7-8>  MnSymbolA7
   <8-9>  MnSymbolA8
   <9-10> MnSymbolA9
  <10-12> MnSymbolA10
  <12->   MnSymbolA12}{}
\DeclareFontShape{U}{MnSymbolA}{b}{n}{
    <-6>  MnSymbolA-Bold5
   <6-7>  MnSymbolA-Bold6
   <7-8>  MnSymbolA-Bold7
   <8-9>  MnSymbolA-Bold8
   <9-10> MnSymbolA-Bold9
  <10-12> MnSymbolA-Bold10
  <12->   MnSymbolA-Bold12}{}

\re@DeclareMathSymbol{\twoheadedswarrow}{\mathord}{mnSyA}{30}

\makeatother

\newcommand{\mlaux}[3]{\setbox0=\hbox{$\mathsurround=0pt #2{#3}$}%
  \dimen0=\dp0\advance\dimen0 by \ht0\lower#1\dimen0\box0}

\newcommand{\makellapm}[2]{\hbox to 0pt{\hss$\mathsurround=0pt #1{#2}$}}

\newcommand{\makerlapm}[2]{\hbox to 0pt{$\mathsurround=0pt #1{#2}$\hss}}

\newcommand{\makelapm}[2]{\hbox to 0pt{\hss$\mathsurround=0pt #1{#2}$\hss}}

\newcommand{\makeushort}[3]{%
  \setbox0=\hbox{$\mathsurround=0pt #2{#3}$}%
  \hbox to 1\wd0{\hss\underbar{\hbox to #1\wd0{\hss\box0\hss}}\hss}}

\def\makebigger#1#2#3{\scalebox{#1}{$\mathsurround=0pt #2{#3}$}}
\def\bigger#1#2{{\relax\mathpalette{\makebigger{#1}}{#2}}}

\def\scaleuphalf{1.0954}
\def\scaleupone{1.2}

\newcommand{\op}{^{\mathord{\text{\rm op}}}}
\newcommand{\co}{^{\mathord{\text{\rm co}}}}

\newcommand{\defeq}{\mathrel{:=}}

\makeatletter

\def\newmop{\@ifstar{\@newmop m}{\@newmop o}}
\def\@newmop#1{\@ifnextchar[{\@@newmop #1}{\@@@newmop #1}}
\def\@@newmop#1[#2]{\@declmathop #1#2}
\def\@@@newmop#1#2{\expandafter\@declmathop\expandafter #1\csname #2\endcsname{#2}}

\makeatother

\newdir{ |}{{}*!/-5pt/@{|}}

\newmop{im}
\newmop{coim}
\newmop{dom}
\newmop{cod}
\newmop{id}
\newmop{Map}

\newmop{obj}
\newmop{arr}
\newmop{sq}
\newmop{norm}

\newmop{el}

\newmop{ev}

\newmop{fun}

\newmop{Ext}
\newmop{icon}
\newmop{pbk}
\newmop{icom}

\newmop{cell}
\newmop{cof}
\newmop{fib}

\newmop{diag}

\newmop*{colim}
\newmop*{holim}

\newmop[\lan]{lan}
\newmop[\ran]{ran}

\newmop*{cone}
\newmop*{cocone}
\newmop*{base}

\newcommand{\comma}{\mathbin{\downarrow}}

\newcommand{\unit}{\eta}
\newcommand{\counit}{\varepsilon}

\makeatletter
\newcommand{\rotatemath}[2]{\rotatebox[origin=c]{180}{$\m@th #1{#2}$}}
\makeatother

\newcommand{\yoneda}{\mathscr{Y}\!}

\newmop[\pr]{pr}
\newmop[\dm]{d}

\newmop[\cpr]{in}

\newcommand{\pocorner}{\hbox to 8pt{{\vrule height8pt depth0pt width0.5pt}%
    \vbox to 8pt{{\hrule height0.5pt width7.5pt depth0pt}\vfill}}}
\newcommand{\poexcursion}{\save[]-<15pt,-15pt>*{\pocorner}\restore}
\newcommand{\pbcorner}{\vbox to 0pt{\kern 4pt\hbox to 0pt{\kern 4pt%
      \vbox{{\hrule height0.5pt width7.5pt depth0pt}}%
      {\vrule height8pt depth0pt width0.5pt}\hss}\vss}}
\newcommand{\pbexcursion}{\save[]+<5pt,-5pt>*{\pbcorner}\restore}

\newcommand{\pwr}{\mathbin\pitchfork}

\newcommand{\wlim}[2]{\{#1,#2\}}

\newmop{cls}

\newcommand{\leib}[1]{\mathbin{\widehat{#1}}}

\newcommand{\category}[1]{\underline{\smash[b]{\text{\rm{#1}}}}}

\newcommand{\cattwo}{{\bigger{1.12}{\mathbbe{2}}}}

\newcommand{\iso}{{\mathbb{I}}}

\makeatletter

\def\Del@Sym{{\bigger\scaleuphalf{\mathbbe{\Delta}}}}

\def\del@fn{\futurelet\del@next}
\def\del@dn{\def\del@next}

\def\parsedel@{%
  \ifx +\del@next \del@dn+{\Del@Sym_{\mathord{+}}}%
  \else \del@dn {\del@fn\parsedel@@}%
  \fi\del@next}

\def\parsedel@@{%
  \ifx\space@\del@next \expandafter\del@dn\space{\del@fn\parsedel@@}%
  \else\ifx [\del@next \del@dn[{\del@fn\parsedel@@@}%
  \else\ifx _\del@next \del@dn{\Delta}%
  \else\ifx ^\del@next \del@dn{\Delta}%
  \else \del@dn{\Del@Sym}%
  \fi\fi\fi\fi\del@next}

\def\parsedel@@@{%
  \ifx\space@\del@next \expandafter\del@dn\space{\del@fn\parsedel@@@}%
  \else\ifx t\del@next \del@dn t{\Del@Sym_\infty\del@fn\parsedel@@@@}%
  \else\ifx b\del@next \del@dn b{\Del@Sym_{-\infty}\del@fn\parsedel@@@@}%
  \else \del@dn{\errmessage{unexpected modifier}}%
  \fi\fi\fi\del@next}

\def\parsedel@@@@{%
  \ifx\space@\del@next \expandafter\del@dn\space{\del@fn\parsedel@@@@}%
  \else\ifx ]\del@next \del@dn]{}%
  \else \del@dn{\errmessage{expecting close of option block}}%
  \fi\fi\del@next}

\def\Del{\del@fn\parsedel@}

\makeatother

\newcommand{\Horn}{\Lambda}

\newcommand{\Cat}{\category{Cat}}

\newcommand{\sCat}{\sSet\text{-}\Cat}
\newcommand{\sSet}{\category{sSet}}
\newcommand{\qCat}{\category{qCat}}

\newcommand{\coCart}{\category{coCart}}

\newcommand{\qMod}[2]{\mathrm{mod}(#1,#2)}
\newcommand{\dmod}[3]{\xymatrix@=1.25em{{#2} \ar[r]|\mid^{ {#1}} & {#3}}}

\newcommand{\pbshape}{{\mathord{\bigger\scaleupone\righthalfcup}}}

\newcommand{\vertex}{\nu}
\newcommand{\degen}{\sigma}

\newcommand{\fbv}[1]{\{{#1}\}}

\newcommand{\join}{\mathbin\star}

\newmop{dec}
\newmop{fatdec}
\newmop{slc}
\newmop{fatslc}

\newmop{ir} 
\newmop{incl}

\newcommand{\nrv}{N}
\newcommand{\ho}{h}

\newcommand{\hN}{\nrv}
\newcommand{\gC}{\mathfrak{C}}

\newcommand{\boundary}{\partial}

\def\reedyfilt#1_#2{#1_{\leq #2}}
\newmop{sk}
\newmop{cosk}
\newmop{res}

\newmop{map}
\newmop{Ho}

\newmop[\pth]{path}
\newmop{cyl}

\newmop[\const]{c}

\newcommand{\Kan}{\category{Kan}}

\newmop{coll}
\newmop{wgt}

\newdir{ >}{{}*!/-7pt/@{>}}
\newdir{u(}{{}*!/-4pt/@^{(}}
\newdir{d(}{{}*!/-4pt/@_{(}}
\newdir{|>}{%
  !/4.5pt/@{|}*:(1,-.2)@^{>}*:(1,+.2)@_{>}*+@{}}

\makeatletter
\def\makeslashed#1#2#3#4#5{#1{\mathpalette{\sla@{#2}{#3}{#4}}{#5}}}

\def\@mathlower#1#2#3{\setbox0=\hbox{$\m@th#2#3$}\lower#1\ht0\box0}
\def\mathlower#1#2{\mathpalette{\@mathlower{#1}}{#2}}
\makeatother

\newcommand{\inc}{\hookrightarrow}

\newcommand{\tfib}{\twoheadrightarrow}

\newcommand{\longtwoheadrightarrow}{\mathrel{\mathord{-}\mkern-3mu\mathord\twoheadrightarrow}}

\newcommand{\we}{\xrightarrow{\mkern10mu{\smash{\mathlower{0.6}{\sim}}}\mkern10mu}}

\newcommand{\trvfib}{\stackrel{\smash{\mkern-2mu\mathlower{1.5}{\sim}}}\longtwoheadrightarrow}

\newcommand{\vsim}{\mathrel{\rotatebox{270}{$\sim$}}}

\newcommand{\To}{\Rightarrow}

\newcommand{\iTo}{\mathbin{\stackrel{\smash{\mkern-5mu\mathlower{.25}{\sim}}}\Rightarrow}}

\makeatletter

\def\tens@fn{\futurelet\tens@next}
\def\tens@dn{\def\tens@nextcont}
\newtoks\tens@toks
\def\addtotens@toks#1{\tens@toks=\expandafter{\the\tens@toks#1}}

\def\parsetens@@{%
    \ifx\space@\tens@next \expandafter\tens@dn\space{\tens@fn\parsetens@@}%
    \else\ifx ^\tens@next \tens@dn ^##1{\parsetens@procsep^\addtotens@toks{##1}%
      \tens@fn\parsetens@@}%
    \else\ifx _\tens@next \tens@dn _##1{\parsetens@procsep_\addtotens@toks{##1}%
      \tens@fn\parsetens@@}%
    \else\tens@dn{\ifx *\tens@last \else\addtotens@toks\egroup\fi\the\tens@toks}%
    \fi\fi\fi\tens@nextcont}

\def\parsetens@procsep#1{%
  \ifx *\tens@last \addtotens@toks{#1}\addtotens@toks\bgroup%
  \else\ifx \tens@last\tens@next \addtotens@toks,%
  \else \addtotens@toks\egroup\addtotens@toks\bgroup%
    \addtotens@toks\egroup\addtotens@toks{#1}\addtotens@toks\bgroup%
  \fi\fi\let\tens@last\tens@next}

\newcommand{\tn}[1]{\let\tens@last=*\tens@toks={#1}\tens@fn\parsetens@@}

\makeatother

\def\adjdisplay#1-|#2:#3->#4.{{%
    \xymatrix@R=0em@!C=2.5em{%
      *+[l]{#3} \ar@/_0.55pc/[rr]_-{#2} & {\bot} &
      *+[r]{#4}\ar@/_0.55pc/[ll]_-{#1}}}}

\def\adjdisplaytwo#1-|#2:#3->#4.{{%
\xymatrix@=1.2em{
      {#3}\ar@/_1.5ex/[rr]_-{#2}^-{}="one"
      & & {#4}
      \ar@/_1.5ex/[ll]_-{#1}^-{}="two"
      \ar@{}"one";"two"|{\bot}
    }}}

\def\tripleadjdisplay#1-|#2-|#3:#4->#5.{{%
\xymatrix@=2.4em{
{#4}\ar[r]|{#2} &
{#5} \ar@/_3ex/[l]_{#1}^{\bot} \ar@/^3ex/[l]_{\bot}^{#3}}
}}

\def\adjinline#1-|#2:#3->#4.{{#1}\dashv{#2}:#3\to #4}

\newcommand{\pent}[1]{
  \xybox{
    \POS (0,-15)*+{\a}="0",
         (-14,-5)*+{\b}="1",
         (-9,12)*+{\c}="2",
         (9,12)*+{\d}="3",
         (14,-5)*+{\e}="4"
    \POS"0" \ar "1"^{\labelstyle \ab}|{}="01"
    \POS"1" \ar "2"^{\labelstyle \bc}|{}="12"
    \POS"2" \ar "3"^{\labelstyle \cd}|{}="23"
    \POS"3" \ar "4"^{\labelstyle \de}|{}="34"
    \POS"0" \ar "4"_{\labelstyle \ae}|{}="04"
    \ifcase #1
    \POS"0" \ar "2"|{\labelstyle \ac}="02"
    \POS"0" \ar "3"|{\labelstyle \ad}="03"
    \POS"02";"1"**{}, ?(0.3) \ar@{=>} ?(0.7)^{\labelstyle \abc}
    \POS"03";"2"**{}, ?(0.25) \ar@{=>} ?(0.5)_{\labelstyle \acd}
    \POS"04";"3"**{}, ?(0.2) \ar@{=>} ?(0.4)_{\labelstyle \ade}
    \or
    \POS"1" \ar "3"|{\labelstyle \bd}="13"
    \POS"1" \ar "4"|{\labelstyle \be}="14"
    \POS"13";"2"**{}, ?(0.3) \ar@{=>} ?(0.7)_{\labelstyle \bcd}
    \POS"14";"3"**{}, ?(0.25) \ar@{=>} ?(0.5)_{\labelstyle \bde}
    \POS"04";"1"**{}, ?(0.25) \ar@{=>} ?(0.5)_{\labelstyle \abe}
    \or
    \POS"2" \ar "4"|{\labelstyle \ce}="24"
    \POS"0" \ar "2"|{\labelstyle \ac}="02"
    \POS"02";"1"**{}, ?(0.3) \ar@{=>} ?(0.7)^{\labelstyle \abc}
    \POS"04";"2"**{}, ?(0.2) \ar@{=>} ?(0.35)_{\labelstyle \ace}
    \POS"24";"3"**{}, ?(0.2) \ar@{=>} ?(0.6)^{\labelstyle \cde}
    \or
    \POS"1" \ar "3"|{\labelstyle \bd}="13"
    \POS"0" \ar "3"|{\labelstyle \ad}="03"
    \POS"04";"3"**{}, ?(0.2) \ar@{=>} ?(0.4)_{\labelstyle \ade}
    \POS"13";"2"**{}, ?(0.3) \ar@{=>} ?(0.7)_{\labelstyle \bcd}
    \POS"03";"1"**{}, ?(0.25) \ar@{=>} ?(0.5)^{\labelstyle \abd}
    \or
    \POS"2" \ar "4"|{\labelstyle \ce}="24"
    \POS"1" \ar "4"|{\labelstyle \be}="14"
    \POS"24";"3"**{}, ?(0.2) \ar@{=>} ?(0.6)^{\labelstyle \cde}
    \POS"04";"1"**{}, ?(0.25) \ar@{=>} ?(0.5)_{\labelstyle \abe}
    \POS"14";"2"**{}, ?(0.25) \ar@{=>} ?(0.5)^{\labelstyle \bce}
    \else\fi
  }
}

\newcommand{\pentofpent}[1]{
  \def\baselen{#1}
  \begin{xy}
    0;<\baselen,0mm>:
    *{\xybox{
        \POS(0,-4)*[o]{\pent 0}="zero"
        \POS(16,40)*[o]{\pent 3}="three"
        \POS(72,40)*[o]{\pent 1}="one"
        \POS(88,-4)*[o]{\pent 4}="four"
        \POS(44,-36)*[o]{\pent 2}="two"
        \ar@<1ex>"zero";"three"^-{\objectstyle\abcd}
        \ar@<1ex>"three";"one"^-{\objectstyle\abde}
        \ar@<1ex>"one";"four"^-{\objectstyle\bcde}
        \ar@<-1ex>"zero";"two"_-{\objectstyle\acde}
        \ar@<-1ex>"two";"four"_-{\objectstyle\abce}
        \ar@{=>}(44,-5);(44,+15)^{\objectstyle\abcde}
     }}
  \end{xy}
}

\makeatletter
\newcommand{\qc}[1]{\mathord{\text{\normalfont{\textsf{#1}}}}}
\newcommand{\qop}[1]{\mathord{\qc{#1}}}
\def\ec@#1#2<.>{{\mathcal{#1}\mkern-2mu\text{\normalfont{%
      \textsf{\slshape #2}}}\mkern2mu}}
\newcommand{\ec}[1]{\mathord{\ec@#1<.>}}
\newcommand{\eop}[1]{\mathord{\ec{#1}}}
\makeatother

\newcommand{\SSet}{\eop{SSet}}

\newcommand{\qqCat}{\qop{qCat}}

\newcommand{\gr}{^{\mathrm{gr}}}

\newcommand{\qA}{\qc{A}}
\newcommand{\qB}{\qc{B}}
\newcommand{\qC}{\qc{C}}
\newcommand{\qD}{\qc{D}}
\newcommand{\qE}{\qc{E}}
\newcommand{\qF}{\qc{F}}
\newcommand{\qG}{\qc{G}}

\newcommand{\qJ}{\qc{J}}
\newcommand{\qK}{\qc{K}}
\newcommand{\qL}{\qc{L}}
\newcommand{\qM}{\qc{M}}
\newcommand{\qP}{\qc{P}}
\newcommand{\qS}{\qc{S}}

\newcommand{\qX}{\qc{X}}

\newcommand{\eK}{\ec{K}}
\newcommand{\eL}{\ec{L}}
\newcommand{\eM}{\ec{M}}
\newcommand{\eA}{\ec{A}}

\newcommand{\eC}{\ec{C}}
\newcommand{\eD}{\ec{D}}

\newcommand{\Hom}{\qop{Hom}}
\newcommand{\Fun}{\qop{Fun}}

\renewcommand{\Map}{\qop{Map}}

\renewcommand{\qMod}{\qop{Mod}}
\renewcommand{\coCart}{\qop{coCart}}
\newcommand{\Cart}{\qop{Cart}}
\newcommand{\Radj}{\qop{Radj}} 
\newcommand{\Ladj}{\qop{Ladj}}

\renewcommand{\qCat}{\eop{QCat}}
\renewcommand{\Kan}{\eop{Kan}}

\newcommand{\psh}{\mathcal{P}}

\newmop{core}
\newmop{Fam}
\newmop{last}
\newmop{rv}
\newmop{free}

\usepackage{xr}

\newcommand{\extRef}[3]{%
  {\protect\IfBeginWith{#3}{itm:}{}{#2.}}\ref*{#1:#3}}

 \newcommand{\refI}{\extRef{found}{I}}
 
 \newcommand{\refIII}{\extRef{complete}{III}}
\newcommand{\refIV}{\extRef{yoneda}{IV}}
 \newcommand{\refV}{\extRef{equipment}{V}}
  \newcommand{\refVI}{\extRef{comprehend}{VI}}
    \newcommand{\refVII}{\extRef{holimits}{VII}}

\setlist{}
\setenumerate{leftmargin=*,labelindent=\parindent}
\setitemize{leftmargin=*,labelindent=0.5\parindent}
\setdescription{leftmargin=1em}

\swapnumbers

\theoremstyle{plain}

\newtheorem{thm}{Theorem}[subsection]
\newtheorem{lem}[thm]{Lemma}
\newtheorem{cor}[thm]{Corollary}

\newtheorem{prop}[thm]{Proposition}

\theoremstyle{definition}
\newtheorem{defn}[thm]{Definition}
\newtheorem{ex}[thm]{Example}
\newtheorem{ntn}[thm]{Notation}

\theoremstyle{remark}
\newtheorem{obs}[thm]{Observation}

\newtheorem{rmk}[thm]{Remark}

\makeatletter
\let\c@equation\c@thm
\makeatother
\numberwithin{equation}{subsection}

\title{On the construction of limits and colimits in $\infty$-categories}

\author[Riehl]{Emily Riehl}
\address{
  Department of Mathematics \\
Johns Hopkins University \\
Baltimore, MD 21218\\
  USA
}
\email{eriehl@math.jhu.edu}

\author[Verity]{Dominic Verity}
\address{
  Centre of Australian Category Theory \\
  Macquarie University \\
  NSW 2109 \\
  Australia
}
\email{dominic.verity@mq.edu.au}

\date{\today}


\begin{document}

  \ifpdf
  \DeclareGraphicsExtensions{.pdf, .jpg, .tif}
  \else
  \DeclareGraphicsExtensions{.eps, .jpg}
  \fi

  \begin{abstract}
In previous work, we introduce an axiomatic framework within which to prove theorems about many varieties of infinite-dimensional categories simultaneously. In this paper, we establish criteria implying that an $\infty$-\emph{category} --- for instance, a quasi-category, a complete Segal space, or a Segal category --- is complete and cocomplete, admitting limits and colimits indexed by any small simplicial set. Our strategy is to build (co)limits of diagrams indexed by a simplicial set inductively from (co)limits of restricted diagrams indexed by the pieces of its skeletal filtration. We show directly that the \emph{modules} that express the universal properties of (co)limits of diagrams of these shapes are reconstructable as limits of the modules that express the universal properties of (co)limits of the restricted diagrams. We also prove that the Yoneda embedding preserves and reflects limits in a suitable sense, and deduce our main theorems as a consequence.
  \end{abstract}

  \maketitle
  \tableofcontents

\section{Introduction}

This paper is a continuation of previous work \cite{RiehlVerity:2012tt, RiehlVerity:2012hc, RiehlVerity:2013cp,RiehlVerity:2015fy,RiehlVerity:2015ke,RiehlVerity:2017cc,RiehlVerity:2018rq} to lay the foundations for the formal theory of $\infty$-\emph{categories}, which model weak higher categories. In contrast with the pioneering work of Joyal \cite{Joyal:2008tq} and Lurie \cite{Lurie:2009fk,Lurie:2012uq}, our approach is ``synthetic'' in the sense that our proofs do not depend on what precisely these $\infty$-categories \emph{are}, but rather rely upon an axiomatisation of the universe in which they \emph{live}. To describe an appropriate ``universe,'' we introduce the notion of an $\infty$-\emph{cosmos}, a (large) simplicially enriched category $\eK$ satisfying certain axioms. The objects of an $\infty$-cosmos are called $\infty$-\emph{categories}. A theorem, e.g. \cite[4.1.10]{RiehlVerity:2015fy} reproduced as Definition \ref{defn:cocart-fibration}, that characterises a \emph{cartesian fibration} of $\infty$-categories in terms of the presence of an adjunction between \emph{comma $\infty$-categories}, is a result about the objects of any $\infty$-cosmos, and thus applies of course to every $\infty$-cosmos. 

The prototypical example is the $\infty$-cosmos whose objects are \emph{quasi-categories}, a model of $(\infty,1)$-categories as simplicial sets satisfying the weak Kan condition, and whose function complexes are the quasi-categories of functors between them.  But there are other $\infty$-cosmoi whose objects are complete Segal spaces or Segal categories, each of these being models of $(\infty,1)$-\emph{categories}; and of $\theta_n$-spaces, or iterated complete Segal spaces, or $n$-trivial saturated complicial sets, each modelling $(\infty,n)$-\emph{categories}. For any $\infty$-cosmos $\eK$ containing an $\infty$-category $B$, the slice category $\eK_{/B}$ is again an $\infty$-cosmos. Thus each of these objects are $\infty$-categories in our sense and our theorems apply to all of them.\footnote{This may seem like sorcery but in some sense it is really just the Yoneda lemma. To a close approximation, an $\infty$-cosmos is a ``category of fibrant objects enriched over quasi-categories.''  When the Joyal--Lurie theory of quasi-categories is expressed in a sufficiently categorical way, it extends to encompass analogous results for the corresponding ``representably defined'' notions in a general $\infty$-cosmos.}  Along the road to our main theorems here, we prove that cartesian and cocartesian fibrations over fixed or varying bases define $\infty$-cosmoi (see Propositions \ref{prop:radj-cosmos} and \ref{prop:cartesian-cosmoi}). While we only require a minor consequence of these results here, they lay the foundations for a complementary approach to parametrised $\infty$-category that is very much in the spirit of  \cite{BDGNS:2016ph, Shah:2016ph}.

Our first paper in this series \cite{RiehlVerity:2012tt} develops the basic theory of limits or colimits of diagrams indexed by a simplicial set $X$ and valued in an $\infty$-category $A$. In the case where $X$ is the nerve of an ordinary 1-category, this data is traditionally thought of as defining a ``homotopy coherent'' diagram of that shape in $A$. In this paper, we shall  explain how to construct such limits inductively using the canonical skeletal decomposition of the simplicial set $X$, in which the cells attached at stage $n$ are indexed by the set $L_nX \subset X_n$ of non-degenerate $n$-simplices:
\begin{equation}\label{eq:skeletal-decomposition} \hbox to 0.8\textwidth{\hss\hspace{3em}$\xymatrix@R=2em@C=1.8em{ &  & & & \coprod_{L_nX} \partial\Delta^n \ar@{^(->}[r] \ar[d] & \coprod_{L_nX}  \Delta^n \ar[d] \\ \emptyset \ar@{^(->}[r] & \sk_0 X \ar@{^(->}[r] & \sk_1X \ar@{^(->}[r] & \sk_2X \ar@{..}[r] & \sk_{n-1} X \ar@{^(->}[r] & \sk_n X \poexcursion \ar@{..}[r] & \colim_n \sk_n X \cong X}$\hss}\end{equation}
The skeletal description gives rise to a presentation of the \emph{diagram $\infty$-category} $A^X$ as the limit in the $\infty$-cosmos of a countable tower of restriction functors, each of which is a pullback of a product of maps of the form $A^{\Delta^n} \tfib A^{\partial\Delta^n}$. We will argue that limits of $X$-indexed diagrams in $A$ can be defined inductively provided that $A$ admits products, pullbacks, and sequential inverse limits---though since sequential inverse limits may be built from countable products and pullbacks we are only required to postulate the existence of the first two of these.\footnote{We thank Tim Campion for pointing out this retrospectively obvious fact to us.} This allows us to provide criteria for ascertaining that an $\infty$-category $A$ is complete (or, dually, cocomplete), obtaining an generalisation of a result \cite[4.4.2.6]{Lurie:2009fk} that Lurie has proven for quasi-categories via a similar decomposition of the indexing simplicial set (see \cite[\S 4.2.3]{Lurie:2009fk}):

{
\renewcommand{\thethm}{\ref{thm:limit-construction}}
\begin{thm}
  Suppose that $\kappa$ is a regular cardinal and that $A$ is an $\infty$-category that admits  products of cardinality $<\kappa$ and pullbacks. If $X$ is a $\kappa$-presentable simplicial set then $A$ admits all limits of diagrams of shape $X$.
\end{thm}
\addtocounter{thm}{-1}
}

To explain the proof strategy, consider a pushout diagram of simplicial sets 
\[ \xymatrix{ X \ar@{^(->}[r] \ar[d] & Y \ar[d] \\ Z \ar@{^(->}[r] & P \poexcursion}\] and suppose that an $\infty$-category $A$ admits limits of shape $X$, $Y$, and $Z$ and also pullbacks, which are limits of shape $\pbshape \defeq \Horn^{2,2}$. A diagram $d \in A^P$ restricts to sub-diagrams $d_X \in A^X$, $d_Y \in A^Y$, and $d_Z \in A^X$. By hypothesis, these each have limits $\ell_X$, $\ell_Y$, and $\ell_Z$ which can be seen to assemble into an internal diagram $d_\pbshape \defeq \ell_Y \to \ell_X \leftarrow \ell_Z$  in $A^\pbshape$, which by hypothesis also has a limit $\ell_\pbshape$. In Proposition \ref{prop:colim-of-diags}, we argue that $\ell_\pbshape$ defines a limit for the original $P$-shaped diagram $d$.

To explain why this is the case, we appeal to one of many equivalent definitions of a limit of a diagram valued in an $\infty$-category. In general, $\ell \in A$ \emph{defines a limit for a diagram} $d \in A^P$ if $\ell$ represents the $\infty$-\emph{category of cones} $\Delta \comma d$ over $d$; see \S\ref{ssec:limits} for precise definitions. This representability is encoded by an equivalence $A \comma \ell \simeq \Delta \comma d$ of \emph{modules} from $1$ to $A$, these modules being the $\infty$-categories defined by pullbacks in the $\infty$-cosmos:
\[ \xymatrix{ A \comma \ell \pbexcursion \ar@{->>}[d] \ar[r] & A^\cattwo \ar@{->>}[d]^{(p_1,p_0)} & & \Delta \comma d\pbexcursion \ar[r] \ar@{->>}[d] & (A^P)^\cattwo \ar@{->>}[d]^{(p_1,p_0)}  \\ 1 \times A \ar[r]^-{\ell \times \id} & A \times A & &  1 \times A \ar[r]^-{d \times \Delta} & A^P \times A^P}\]
So our hypothesised limits for the  sub-diagrams of $d$ provides equivalences $A \downarrow \ell_X \simeq \Delta \comma d_X$, $A \downarrow \ell_Y \simeq \Delta \comma d_Y$, and $A \downarrow \ell_Z \simeq \Delta \comma d_Z$, and similarly, the universal property of $\ell_\pbshape$ as the limit of the $\pbshape$-shaped diagram $d_\pbshape$ is encoded by an equivalence $A \downarrow \ell_\pbshape \simeq \Delta \comma d_\pbshape$ of modules.

We must show that $\ell_\pbshape$ has the stronger universal property of representing cones over the diagram $d$, i.e., that $A \comma \ell_\pbshape$ is equivalent to $\Delta \comma d$. Since the diagram $P$ is a pushout, it follows easily that the $\infty$-category of $P$-shaped cones $\Delta \comma d$ is isomorphic to the pullback:
\[
\xymatrix{ \Delta \comma d \ar@{->>}[d] \pbexcursion \ar[r] & \Delta \comma d_Y \ar@{->>}[d] \\ \Delta \comma d_Z \ar[r] & \Delta \comma d_X}\]
This result appears as Lemma \ref{lem:colim-of-diags}. So we may demonstrate the desired equivalence by arguing that $A \comma \ell_\pbshape$ is the pullback of the equivalent cospan
\[ \xymatrix{ A \comma \ell_Y \ar[r]  \ar[d]_{\rotatebox{90}{$\simeq$}}& A \comma \ell_X \ar[d]^{\rotatebox{270}{$\simeq$}} & A \comma \ell_Z \ar[l] \ar[d]^{\rotatebox{270}{$\simeq$}} \\ \Delta \comma d_Y \ar[r] & \Delta \comma d_X & \Delta \comma d_Z \ar[l]}
\] 
in the large quasi-category of modules from $1$ to $A$. To demonstrate this, and similar results for products and inverse limits of sequences, we prove that

{
\renewcommand{\thethm}{\ref{prop:gen.yoneda.pres.lim}}
\begin{prop}
 For any $\infty$-category $A$, the covariant and contravariant Yoneda embeddings
\[ \Fun_{\eK}(1,A)\hookrightarrow {}_1\qMod_A \qquad \text{and}\qquad \Fun_{\eK}(1,A)\op \hookrightarrow {}_A{\qMod_1}\co\]  preserves any family of limits which is stable under precomposition in $\eK$.
\end{prop}
\addtocounter{thm}{-1}
}

The clause ``stable under precomposition in $\eK$''  has to do with a subtlety in the statement: the Yoneda embedding appearing there is \emph{external}, defined as a functor of quasi-categories, rather than internal to the $\infty$-cosmos $\eK$. The limits it preserves are those arising from the $\infty$-cosmos in which $A$ is defined, which are detectable as those limits in the underlying quasi-category $\Fun_{\eK}(1,A)$ of the $\infty$-category $A$ that are ``stable under precomposition.'' This condition disappears when $\eK$ is the $\infty$-cosmos of quasi-categories, in which case this result is first proven by Lurie in \cite[5.1.3.2]{Lurie:2009fk}.

To prove Proposition \ref{prop:gen.yoneda.pres.lim}, in turn, we must first analyse limits in (large) quasi-categories such as ${}_1\qMod_A$ that are defined as homotopy coherent nerves of Kan-complex-enriched categories. A companion paper \cite{RiehlVerity:2018rq} does exactly this. There, we consider a general notion of \emph{pseudo homotopy limit} of a homotopy coherent diagram, defined to be a particular \emph{flexible weighted limit} whose universal property is satisfied up to equivalence of quasi-categories. 
The main theorem, recalled as Theorem \ref{thm:nerve-completeness} below, demonstrates that limits in the quasi-category ${}_1\qMod_A$ transpose across the homotopy coherent realization--homotopy coherent nerve adjunction to pseudo homotopy limits in the Kan-complex-enriched category of modules from $1$ to $A$. In \S\ref{ssec:complete-quasi}, we provide explicit calculations of the pseudo homotopy limits needed to prove Theorem \ref{thm:limit-construction} and its dual.

While the steps in the proof of Theorem \ref{thm:limit-construction} certainly contain more subtleties than in the classical case, the construction given here of a general limit out of iterations of simpler limits, is entirely analogous to the proof of the classical 1-category theoretic result presented, for instance, in \cite[3.4.12]{Riehl:2016cc}.

This paper contains all of the background needed to fill in the details of this outline, with the proofs of these results appearing in  \S\ref{sec:construction}.  To concisely cite previous work in this program, we refer to the results of \cite{RiehlVerity:2012tt, RiehlVerity:2012hc, RiehlVerity:2013cp,RiehlVerity:2015fy,RiehlVerity:2015ke,RiehlVerity:2017cc, RiehlVerity:2018rq}  as I.x.x.x., II.x.x.x, III.x.x.x, IV.x.x.x, V.x.x.x, VI.x.x.x, or VII.x.x.x respectively, though the statements of the most important results are reproduced here for ease of reference. When an external reference accompanies a restated result, this generally indicates that more expository details can be found there.

In \S\ref{sec:background}, we introduce $\infty$-cosmoi and flexible weighted limits, and then recall the notion of an absolute lifting diagram as defined in the homotopy 2-category of an $\infty$-cosmos. In \S\ref{sec:cartesian}, we define cartesian fibrations, cocartesian fibrations, and the accompanying notion of modules between $\infty$-categories. We also prove a new general result in $\infty$-cosmology, demonstrating that the subcategories
\begin{equation}\label{eq:cart-cosmoi} \coCart(\eK)_{/B} \inc \eK_{/B} \qquad \mathrm{and} \qquad \Cart(\eK)_{/B}\inc \eK_{/B}\end{equation}
of (co)cartesian fibrations and cartesian functors between them define $\infty$-cosmoi, as subcategories of the sliced $\infty$-cosmos.

Perhaps the main technical challenge in extending the classical categorical theory of limits and colimits to the $\infty$-categorical context is in merely \emph{defining} the  Yoneda embedding that appears in Proposition \ref{prop:gen.yoneda.pres.lim}; a comparable amount of work is involved in our favorite construction of the quasi-categorical Yoneda embedding, exposed by Cisinski in \cite[\S 5]{Cisinski:2019hc}. In \cite{RiehlVerity:2017cc}, the Yoneda embedding is constructed as an instance of the versatile \emph{comprehension construction}. This material is reviewed in \S\ref{sec:comprehension-yoneda}. For any fixed cocartesian fibration $p \colon E \tfib B$ in an $\infty$-cosmos $\eK$ and $\infty$-category $A$, the comprehension construction produces a simplicial functor
\[\gC\Fun_{\eK}(A,B) \xrightarrow{c_{p,A}} \coCart(\eK)_{/A} \subset \eK_{/A}\] defined on a vertex $a \colon A \to B$ by the pullback:
\[
\xymatrix{ E_a \ar@{->>}[d]_{p_a} \ar[r]^-{\ell_a} \pbexcursion & E \ar@{->>}[d]^p \\ A \ar[r]_a & B}
\]
whose codomain is the Kan-complex-enriched core of the subcategory of $\eK_{/A}$ spanned by the cocartesian fibrations and cartesian functors. The functor $c_{p,A}$ transposes to define a functor from the quasi-category $\Fun_{\eK}(A,B)$ of functors from $A$ to $B$ to the large quasi-category of cocartesian fibrations over $A$. The Yoneda embedding is defined as the restriction of a particular instance of this, obtained by applying this result to the \emph{arrow $\infty$-category}, which defines a cocartesian fibration $(p_1,p_0) \colon A^\cattwo \tfib A \times A$ in the sliced $\infty$-cosmos $\eK_{/A}$.

 In \S\ref{sec:formal},  we provide an brief introduction to synthetic $\infty$-category theory developed in an $\infty$-cosmos, focusing on the theory of limits and colimits and the functors that preserve them. The new material in this section develops the theory of fully faithful and strongly generating functors, isolating the formal properties of the Yoneda embedding that will allow us to prove Proposition \ref{prop:gen.yoneda.pres.lim}. Finally, in \S\ref{sec:construction}, we define the pseudo homotopy limits appearing in the statement of Theorem \ref{thm:nerve-completeness} and compute some explicit examples.   We then apply the material of \S\ref{sec:cartesian} and \S\ref{sec:formal} to the Yoneda embedding of \S\ref{sec:comprehension-yoneda} to prove Proposition \ref{prop:gen.yoneda.pres.lim}, and then use this to prove our main theorem.

While many of the results herein will be familiar to the $\infty$-categorically well informed reader, the context in which they are applied and the approach we take to their proofs is likely to be more novel. When specialized to an $\infty$-cosmos of $(\infty,1)$-categories, Theorem \ref{thm:limit-construction} appears as \cite[4.4.2.6]{Lurie:2009fk}, Proposition \ref{prop:gen.yoneda.pres.lim} appears as \cite[5.1.3.2]{Lurie:2009fk}, and Theorem \ref{thm:nerve-completeness} appears as  \cite[4.2.4.1]{Lurie:2009fk}
  and will thus be familiar to the quasi-categorical cognoscenti, but our extension of these results to an arbitrary $\infty$-cosmos allows us to press them  into service to explicate certain aspects of the \emph{meta-theory} of other species of $\infty$-category. Our guiding light in developing these works has been the \emph{pro-arrow equipment\/}~\cite{wood:proI,wood:proII} and \emph{Yoneda structure\/}~\cite{street.walters:yoneda} based accounts of classical 1-category theory. For example, the preservation result developed in Proposition~\ref{prop:gen.yoneda.pres.lim} is a direct analogue of an important component of the pro-arrow axiomatics. The arguments given here also lead, in subsequent work, to independent proofs of the exponentiability of (co)cartesian fibrations of quasi-categories (and generalisations to certain higher contexts) and of the density of the point in spaces (in certain $\infty$-cosmoi of fibred $\infty$-categories). Furthermore, our foundations will (eventually) make substantial use of the fact that the flexible weighted limit-creating inclusions \eqref{prop:cartesian-cosmoi} of Proposition \ref{prop:cocartesian-completeness} define monadic functors between the corresponding large quasi-categories.

We might also note, in passing, that the proofs leading to Theorem~\ref{thm:limit-construction} may be generalised to deliver other important results of that kind. For example, when our ambient $\infty$-cosmos $\eK$ is cartesian closed then it is natural to study limits of diagrams indexed by $\infty$-categories in $\eK$, rather than by simplicial sets external to it. In that situation, our approach to these results leads to an analogue of Proposition~\ref{prop:colim-of-diags} which applies to pseudo homotopy colimits of diagram shapes in $\eK$. Indeed, similar comments apply to an endeavour close to our hearts, that of generalising results of this kind to the $(\infty,\infty)$-categorical theory of \emph{complicial sets\/}~\cite{Verity:2007:wcs1}. The extension of many of the methods we present here to that, much more general, context is largely a matter of taking a little more care to push \emph{markings\/} (or \emph{stratifications}) around our homotopy coherent structures; this, however, is a topic for another work.

\subsection{Size conventions}

The quasi-categories defined as homotopy coherent nerves are typically large. All other quasi-categories or simplicial sets, particularly those used to index homotopy coherent diagrams, are assumed to be small. In particular,    when discussing the existence of limits and colimits we shall implicitly assume that these are indexed by small categories, and correspondingly, completeness and cocompleteness properties will implicitly reference the existence of small limits and small colimits.  Here, as is typical, ``small'' sets will usually refer to those members of a Grothendieck universe defined relative to a fixed inaccessible cardinal.

Our intent is to provide a size classification which allows us state and prove results that require such a distinction for non-triviality, principally those of the form ``such and such a \emph{large} category admits all \emph{small} limits''. Our arguments mostly comprise elementary constructions, so in applications this size distinction need not invoke the full force of a Grothendieck universe, indeed it might be as simple as that between the finite and the infinite. At the other extreme it might involve the choice of two Grothendieck universes to prove results about large categories. 

We use a common typeface --- e.g.~$\qA$, $\qK$, ${}_1\qMod(\qK)_A$ --- to differentiate small and large quasi-categories from generic $\infty$-categories $A$; see \ref{ntn:qcat-ntn}.
 
\subsection{Acknowledgements}

The authors are grateful for support from the National Science Foundation (DMS-1551129 and DMS-1652600) and from the Australian Research Council (DP160101519). This work was commenced when the second-named author was visiting the first at Harvard and then at Johns Hopkins, continued while the first-named author was visiting the second at Macquarie, and completed after everyone finally made their way home. We thank all three institutions for their assistance in procuring the necessary visas as well as for their hospitality.

We owe an additional debt of gratitude to the referee who pointed out a simplification of the proof of the converse half of Theorem \ref{thm:nerve-completeness}, which had originally appeared here, allowing us to move the proofs of both directions of the biconditional to \cite{RiehlVerity:2018rq}. In addition to a number of other cogent mathematical and expository suggestions which precipitated a cascading reorganization of the present manuscript, the referee directed us to results in the literature we had overlooked. We apologize to those authors and  encourage them to write to us directly  if in the future we again fail to do our due diligence. 


\section{\texorpdfstring{$\infty$}{infinity}-cosmoi and flexible weighted limits}\label{sec:background}

In this section we review the axiomatic framework for the formal theory of $\infty$-categories, introducing the notion of an $\infty$-cosmos in an abbreviated \S\ref{ssec:cosmoi-background}. In \S\ref{ssec:flexible}, we review some of the more exotic \emph{flexible weighted limits} that exist in any $\infty$-cosmos. This will be used to establish the new $\infty$-cosmoi of \S\ref{sec:cartesian}.

In \S\ref{ssec:absolute}, we recall the homotopy 2-category of an $\infty$-cosmos and the notion of  \emph{absolute lifting diagrams}, which encode universal properties of $\infty$-categories using the structure of a strict 2-category of $\infty$-categories, $\infty$-functors, and $\infty$-natural transformations constructed as a quotient of an $\infty$-cosmos. The universal properties expressed by absolute lifting diagrams can also be encoded internally to the $\infty$-cosmos as a fibred equivalence between \emph{comma $\infty$-categories}, which are the subject of \S\ref{ssec:comma}. 

\subsection{\texorpdfstring{$\infty$}{infinity}-cosmoi and the comma construction}\label{ssec:cosmoi-background}

An $\infty$-cosmos is a category $\eK$ whose objects $A, B$ we call $\infty$-\emph{categories} and whose function complexes $\Fun_{\eK}(A,B)$ are quasi-categories of \emph{functors} between them. The handful of axioms imposed on the ambient quasi-categorically enriched category $\eK$ permit the development of a general theory of $\infty$-categories ``synthetically,'' i.e., only in reference to this axiomatic framework, as we shall discover in \S\ref{sec:formal}.

\begin{defn}[$\infty$-cosmos]\label{defn:cosmos}
An $\infty$-\emph{cosmos} is a simplicially enriched category $\eK$ whose 
\begin{itemize}
\item objects we refer to as the \emph{$\infty$-categories} in the $\infty$-cosmos, whose
\item hom simplicial sets $\Fun_{\eK}(A,B)$ are  quasi-categories, 
\end{itemize} and that is equipped with a specified subcategory of \emph{isofibrations}, denoted by ``$\tfib$'',
satisfying the following axioms:
 \begin{enumerate}[label=(\alph*),series=defn:cosmos]
    \item\label{defn:cosmos:a} (completeness) As a simplicially enriched category,  $\eK$ possesses a terminal object $1$, small products, cotensors $A^U$ of  objects $A$ by all small simplicial sets $U$, inverse limits of countable sequences of isofibrations, and pullbacks of isofibrations along any functor.
    \item\label{defn:cosmos:b} (isofibrations) The class of isofibrations contains the isomorphisms and all of the functors $!\colon A \tfib 1$ with codomain $1$; is stable under pullback along all functors; is closed under inverse limit of countable sequences; and if $p\colon E\tfib B$ is an isofibration in $\eK$ and $i\colon U\inc V$ is an inclusion of  simplicial sets then the Leibniz cotensor $i\leib\pwr p\colon E^V\tfib E^U\times_{B^U} B^V$ is an isofibration. Moreover, for any object $X$ and isofibration $p \colon E \tfib B$, $\Fun_{\eK}(X,p) \colon \Fun_{\eK}(X,E) \tfib \Fun_{\eK}(X,B)$ is an isofibration of quasi-categories.
\end{enumerate}
\end{defn}

For ease of reference, we refer to the limit types listed in axiom \ref{defn:cosmos:a} as the \emph{cosmological limit types}, these referring to diagrams of a particular shape with certain maps given by isofibrations.

The underlying category of an $\infty$-cosmos $\eK$ has a canonical subcategory of representably-defined equivalences, denoted by ``$\we$'', satisfying the 2-of-6 property: a functor $f \colon A \to B$ is an \emph{equivalence} just when the induced functor $\Fun_{\eK}(X,f) \colon \Fun_{\eK}(X,A) \to \Fun_{\eK}(X,B)$ is an equivalence of quasi-categories for all objects $X \in \eK$.  The  \emph{trivial fibrations}, denoted by ``$\trvfib$'', are those functors that are both equivalences and isofibrations. These axioms imply that the underlying 1-category of an $\infty$-cosmos is a category of fibrant objects in the sense of Brown. Consequently, many familiar homotopical properties follow from Definition \ref{defn:cosmos}.  For instance, the axioms of an $\infty$-cosmos permit us to construct \emph{arrow} and \emph{comma} $\infty$-categories as particular simplicially enriched limits.

\begin{defn}[comma $\infty$-categories]\label{defn:comma} 
For any $\infty$-category $A$, the simplicial cotensor 
\[ \xymatrix@C=30pt{ A^\cattwo \defeq A^{\Del^1} \ar@{->>}[r]^-{(p_1,p_0)} & {A^{\boundary\Delta^1}} \cong A \times A}\] defines the \emph{arrow $\infty$-category} $A^\cattwo$, equipped with an isofibration $(p_1,p_0)\colon A^\cattwo \tfib A \times A$, where $p_1 \colon A^\cattwo \tfib A$ and $p_0 \colon A^\cattwo \tfib A$ denote the codomain and domain projections.

More generally, any pair of functors  $f\colon B\to A$ and $g\colon C\to A$ in an $\infty$-cosmos $\eK$ has an associated \emph{comma $\infty$-category}, constructed by the following pullback in $\eK$:
\[
    \xymatrix@=2.5em{
      {f\comma g}\pbexcursion \ar[r]\ar@{->>}[d]_{(p_1,p_0)} &
      {A^\cattwo} \ar@{->>}[d]^{(p_1,p_0)} \\
      {C\times B} \ar[r]_-{g\times f} & {A\times A}
    }
\]
\end{defn}

\begin{prop}[{maps between commas, \refVII{prop:trans-comma}}]\label{prop:trans-comma}
A natural transformation of co-spans on the left of the following display gives rise to the diagram of pullbacks on the right
\begin{equation*}
  \vcenter{
    \xymatrix@R=1.5em@C=2.5em{
      {C}\ar[r]^{g}\ar[d]_{c} & {A}\ar[d]_{a} & {B}\ar[l]_{f}\ar[d]^{b} \\
      {C'}\ar[r]_{g'} & {A'} & {B'}\ar[l]^{f'}
    }}\mkern30mu\rightsquigarrow\mkern30mu
  \vcenter{
    \xymatrix@=1em{
      {f\comma g}\pbexcursion\ar[rr]\ar@{->>}[dd] &&
      {A^{\cattwo}}\ar@{->>}[dd]|!{[rd];[ld]}\hole & \\
      & {f'\comma g'}\pbexcursion\ar[rr]\ar@{->>}[dd] &&
      {(A')^{\cattwo}}\ar@{->>}[dd]\\
      {C\times B}\ar[rr]|!{[ru];[rd]}\hole_(0.65){g\times f} && {A\times A} & \\
      & {C'\times B'}\ar[rr]_{g'\times f'} &&
      {A'\times A'}
      \ar@{-->} "1,1";"2,2"
      \ar "3,1";"4,2"_-{c\times b}
      \ar "1,3";"2,4" ^{a^{\cattwo}}
      \ar "3,3";"4,4" ^{a\times a}
    }}
\end{equation*}
in which the uniquely induced dashed map completing the commutative cube is denoted \[\comma(b,a,c) \colon f \comma g \to f' \comma g'.\] Moreover,  $\comma(b,a,c)$ is an isofibration (resp.~trivial fibration, equivalence) whenever the components $a$, $b$ and $c$ are all maps of that kind.
\end{prop}

The axioms defining an $\infty$-cosmos are intentionally quite sparse so that there will be many examples, such as:

  \begin{prop}[{$\infty$-cosmoi of isofibrations, \refVII{prop:isofib-cosmoi}}]\label{prop:isofib-cosmoi} For any $\infty$-cosmos $\eK$, there is an $\infty$-cosmos $\eK^\cattwo$ which has:
  \begin{itemize}
  \item objects all isofibrations $p \colon E \tfib A$ in $\eK$;
  \item functor space from $p \colon E \tfib A$ to $q \colon F \tfib B$ defined by taking the pullback
    \[
      \xymatrix@R=1.5em@C=4em{
        {\Fun_{\eK^{\cattwo}}(p,q)}\pbexcursion\ar[r]\ar@{->>}[d] &
        {\Fun_{\eK}(E,F)}\ar@{->>}[d]^{\Fun_{\eK}(E,q)} \\
        {\Fun_{\eK}(A,B)}\ar[r]_-{\Fun_{\eK}(p,B)} & {\Fun_{\eK}(E,B)}
      }
    \]
    in simplicial sets so, in particular, the $0$-arrows from $p$ to $q$ are commutative squares
    \begin{equation}\label{eq:K-cattwo-squares}
      \xymatrix@=2em{
        {E}\ar@{->>}[d]_{p}\ar[r]^{g} & {F}\ar@{->>}[d]^{q} \\
        {A}\ar[r]_{f} & {B}
      }
    \end{equation}
    in $\eK$;
  \item equivalences those squares~\eqref{eq:K-cattwo-squares} whose components $f$ and $g$ are equivalences in $\eK$ and isofibrations (resp.\ trivial fibrations) those squares for which the map $f$ and the induced map $E\dashrightarrow A\times_B F$ (and thus also $g$) are isofibrations (resp.\ trivial fibrations) in $\eK$.
  \end{itemize}
  The cosmological limits are defined object-wise in $\eK$, or in other words are jointly created by the domain and codomain projections $\dom,\cod\colon\eK^{\cattwo}\to\eK$.
\end{prop}

A \emph{cosmological functor} is a simplicial functor $\eK \to \eL$ between $\infty$-cosmoi that preserves the classes of isofibrations and all the cosmological limits. For instance, the  domain and codomain projections $\dom,\cod\colon\eK^{\cattwo}\to\eK$ are both cosmological. The cosmological functor $\cod \colon \eK^\cattwo \to \eK$ has a special property: namely its fibres define $\infty$-cosmoi: the sliced $\infty$-cosmoi $\eK_{/B}$  of Example \refIV{ex:sliced.contexts}.

\subsection{Flexible weighted limits in an \texorpdfstring{$\infty$}{infinity}-cosmos}\label{ssec:flexible}

The basic simplicially-enriched limit notions enumerated in axiom \ref{defn:cosmos}\ref{defn:cosmos:a} imply that an $\infty$-cosmos $\eK$ possesses a much larger class of simplicially enriched limits. Before describing them recall the general notion of a \emph{weighted limit} of a simplicially enriched functor valued in a simplicial category $\eC$ with hom-spaces dented by $\Map_{\eC}(X,Y)$.

\begin{defn}\label{defn:simp-weight}
A \emph{weight} for a diagram indexed by a small simplicial category $\eA$ is a simplicial functor $W \colon \eA \to \SSet$. A $W$-\emph{cone} over a diagram $F \colon \eA \to \eC$ in a simplicial category $\eC$ is comprised of an object $L \in \eC$ together with a simplicial natural transformation $\lambda \colon W \to \Map_{\eC}(L,F-)$. Such a cone displays $L$ as a $W$-\emph{weighted limit of} $F$ if and only if for all $X \in \eC$ the simplicial map
  \begin{equation}\label{eq:weighted-UP}
    \xymatrix@R=0em@C=5em{
      \Map_{\eC}(X,L) \ar[r]^-{\cong} & \Map_{\SSet^{\eA}}(W,\Map_{\eC}(X,F-))
    }
  \end{equation}
  given by post-composition with $\lambda$ is an isomorphism, in which case the limit object $L$ is typically denoted by $\wlim{W}{F}_{\eA}$ or simply $\wlim{W}{F}$. In this notation, the universal property \eqref{eq:weighted-UP} of the weighted limit asserts an isomorphism
  \[ 
      \xymatrix@R=0em@C=5em{
      \Map_{\eC}(X,\wlim{W}{F}_{\eA}) \ar[r]^-{\cong} & \{W,\Map_{\eC}(X,F-)\}_{\eA}.
    }
    \]
\end{defn}

\begin{defn}[flexible weights]\label{defn:flexible-weight}
  For a small simplicial category $\eA$ and pair of objects $[n] \in \Del$ and
  $A \in \eA$, the \emph{projective $n$-cell} associated with $A$ is the
  simplicial natural transformation:
  \[
    \boundary\Del^n\times \Map_{\eA}(A,-)\inc
    \Del^n\times \Map_{\eA}(A,-).
  \]
A weight $W\colon\eA \to \SSet$ is  a \emph{flexible\/} if the inclusion $\emptyset\inc W$ may be expressed as a countable composite of pushouts of coproducts of projective cells. 
\end{defn}

\begin{prop}[{\refVII{prop:flexible-weights-are-htpical}}]\label{prop:flexible-weights-are-htpical} Let $\eK$ be an $\infty$-cosmos and let $\eA$ be a small simplicial category.
  \begin{enumerate}[label=(\roman*)]
  \item\label{itm:flexible-exist} For any diagram $F \colon \eA\to\eK$ and flexible weight $W \colon \eA \to \SSet$, the weighted limit $\wlim{W}{F}$ exists in $\eK$.
  \item\label{itm:flexible-htpical} If $\kappa \colon F \To G$ is a simplicial natural transformation between two such diagrams whose components are equivalences, isofibrations, or trivial fibrations in $\eK$ and $W$ is a flexible weight, then the induced map
    \begin{equation*}
      \xymatrix@R=0em@C=6em{
        {\wlim{W}{F}}\ar[r]^-{\wlim{W}{\kappa}} & {\wlim{W}{G}}
      }
    \end{equation*}
    is an equivalence, isofibration, or trivial fibration (respectively) in $\eK$.
  \end{enumerate}
\end{prop}

When working in a quasi-categorically enriched category $\eK$ it is often the case that we are only interested in weighted (co)limits that are defined up to \emph{equivalence\/} rather than \emph{isomorphism}. To that end we have the following definition:

\begin{defn}[flexible weighted homotopy limits]\label{defn:flexible-hty-limit}
  Suppose that $W\colon\eA\to\SSet$ is a flexible weight and that $F\colon\eA\to\eK$ is a diagram in a quasi-categorically enriched category $\eK$. We say that a $W$-cone $\lambda\colon W\to \Fun_{\eK}(L,-)$ displays an object $L\in \eK$ as a {\em flexible weighted homotopy limit\/} of $F$ weighted by $W$ if for all objects $X\in\eK$ the map
\begin{equation}\label{eq:flex-hty-lim-comp}
  \xymatrix@R=0em@C=5em{
    \Fun_{\eK}(X,L) \ar[r]^-\simeq &
    \{W,\Fun_{\eK}(X,F-)\}_{\eA}.
  }
\end{equation}
induced by post-composition with $\lambda$ is an equivalence of quasi-categories,\footnote{Here we show that the codomain of the comparison map in~\eqref{eq:flex-hty-lim-comp} is a quasi-category by applying Proposition~\ref{prop:flexible-weights-are-htpical} in the $\infty$-cosmos of quasi-categories.} in which case we denote the limit object by $\wlim{W}{F}^\simeq_{\eA}$.
\end{defn}

\subsection{Absolute lifting diagrams and (relative) adjunctions}\label{ssec:absolute}

Recall the quotient homotopy 2-category of an $\infty$-cosmos: 

\begin{defn}[the homotopy 2-category of an $\infty$-cosmos]\label{defn:hty-2-cat} The \emph{homotopy 2-category} of an $\infty$-cosmos $\eK$ is defined by applying the homotopy category functor $\ho \colon \qCat \to \eop{Cat}$ to the functor spaces of the $\infty$-cosmos: 
\begin{itemize}
\item The objects of $\ho_*\eK$ are the objects of $\eK$, i.e., the $\infty$-categories.
\item The 1-cells $f \colon A \to B$ of $\ho_*\eK$ are the vertices $f \in \Fun_{\eK}(A,B)$ in the functor spaces of $\eK$, i.e., the $\infty$-functors.
\item A 2-cell  $\xymatrix{ A \ar@/^2ex/[r]^f \ar@/_2ex/[r]_g \ar@{}[r]|{\Downarrow\alpha}& B}$ in $\ho_*\eK$ is represented by a 1-arrow $\alpha \colon f \to g \in \Fun_{\eK}(A,B)$, where a parallel pair of 1-arrows in $\Fun_{\eK}(A,B)$ represent the same 2-cell if and only if they bound a 2-arrow whose remaining outer face is degenerate.
\end{itemize}
\end{defn}

\begin{defn}[dual $\infty$-cosmoi]\label{defn:dual-cosmoi} For any $\infty$-cosmos $\eK$, write $\eK\co$ for
 the $\infty$-cosmos with the same objects but with the opposite functor spaces
\[ \Fun_{\eK\co}(A,B) \defeq\Fun_{\eK}(A,B)\op.\] 
 The homotopy 2-category of $\eK\co$ is the ``co'' dual of the homotopy 2-category of $\eK$, reversing the 2-cells but not the 1-cells. 
 \end{defn}

Pleasingly, some of the formal theory of $\infty$-categories can be developed in these strict 2-categories of $\infty$-categories, $\infty$-functors, and $\infty$-natural transformations.

\begin{defn} An \emph{adjunction} between $\infty$-categories $A, B \in \eK$ is simply an adjunction in the homotopy 2-category $\ho_*\eK$: i.e., is comprised of a pair functors $f \colon B \to A$ and $u \colon A \to B$, together with a pair of 2-cells $\unit \colon \id_B \To uf$ and $\counit \colon fu \To \id_A$ satisfying the triangle identities.
\end{defn}

We shall also make use of the following ``partial adjunction'' notion:

\begin{defn}[absolute right lifting]\label{defn:absolute-right-lifting} Given a cospan $C \xrightarrow{g} A \xleftarrow{f} B$, a functor $\ell \colon C \to B$ and a 2-cell
  \begin{equation}\label{eq:abs-right-lifting}
    \xymatrix{
      \ar@{}[dr]|(.7){\Downarrow\lambda} & B \ar[d]^f \\
      C \ar[ur]^\ell \ar[r]_g & A}
  \end{equation}
  define an \emph{absolute right lifting of $g$ through $f$} if any 2-cell as displayed below-left factors uniquely through $\lambda$ as displayed below-right
  \begin{equation}\label{eq:abs-rl-univ}
    \vcenter{\xymatrix{
        X \ar[d]_c \ar[r]^b \ar@{}[dr]|{\Downarrow\alpha} &
        B \ar[d]^f \\
        C \ar[r]_g & A}}
    \mkern20mu = \mkern20mu
    \vcenter{\xymatrix{
        X \ar[d]_c \ar[r]^b \ar@{}[dr]|(.3){\exists !\Downarrow\bar\alpha}
        |(.7){\Downarrow\lambda} & B \ar[d]^f \\
        C \ar[ur]|(.4)*+<2pt>{\scriptstyle \ell} \ar[r]_g & A}}
\end{equation}
When this property holds, we say that the triangle displayed in \eqref{eq:abs-right-lifting} as an \emph{absolute right lifting diagram}.
\end{defn}

\begin{rmk}\label{rmk:abs-lift-stab-precomp}
  In category theory, the term ``absolute'' typically means ``preserved by all functors.'' In that spirit, an absolute right lifting diagram is a right lifting diagram\footnote{A \emph{right lifting diagram} is a pair $(\ell,\lambda)$ as in \eqref{eq:abs-right-lifting} satisfying the universal property of \eqref{eq:abs-rl-univ} only in the case where the functor $c$ is $\id_C$.} $\lambda \colon f\ell \To g$ with the property that the restriction of $\lambda$ along any generalised element $c \colon X \to C$ again defines a right lifting diagram.
\end{rmk}

\begin{ex}\label{ex:radj-lifting}
Importantly, $f\dashv u$ is an adjunction with counit $\counit\colon fu\To \id_A$ if and only if the triangle
\begin{equation*}
  \xymatrix{
    \ar@{}[dr]|(.7){\Downarrow\counit} & B \ar[d]^f \\
    A \ar[ur]^{u} \ar@{=}[r] & A}
\end{equation*}
is an absolute right lifting diagram.
\end{ex}

Directly from the universal property of absolute right lifting diagrams observe:

\begin{lem}[composition and cancellation of absolute right lifting diagrams]\label{lem:comp-canc-abs-lift}
  In any 2-category, suppose we are given a diagram
  \begin{equation*}
    \xymatrix@R=1.5em@C=3em{
      \ar@{}[dr]|(0.7){\Downarrow\lambda} &
      {A}\ar[d]^{f} \\
      {D}\ar[ur]^{h}\ar[r]|{k}\ar[dr]_{l} &
      {B}\ar[d]^{g} \\
      \ar@{}[ur]|(0.7){\Downarrow\mu} & {C}
    }
  \end{equation*}
and assume that the lower triangle is an absolute right lifting diagram. Then the upper triangle is an absolute right lifting diagram if and only if the composite triangle displays $h$ as an absolute right lifting of $l$ along $gf$. \qed
\end{lem}

\begin{defn}\label{defn:abs-lift-trans-right-exactness}
Transformations 
\begin{equation}\label{eq:induced-mate}
    \xymatrix@=1.5em{
      {} & {B}\ar[d]_{f}\ar[dr]^{v} & {} \\
      {C}\ar[r]^{g}\ar[dr]_{w} & {A}\ar[dr]|*+<1pt>{\scriptstyle u} & 
      {B'}\ar[d]^{f'} \\
      {} & {C'}\ar[r]_{g'} & {A'}
    }
    \end{equation}
   between diagrams which admit absolute right liftings give rise to the following diagram
\[
    \vcenter{
      \xymatrix@=1.5em{
        {} \ar@{}[dr]|(.7){\Downarrow\lambda}
        & {B}\ar[d]^{f}\ar[dr]^{v} & {} \\
        {C}\ar[ur]^{\ell}\ar[r]_{g}\ar[dr]_{w} & 
        {A}\ar[dr]|*+<1pt>{\scriptstyle u} & 
        {B'}\ar[d]^{f'} \\
        {} & {C'}\ar[r]_{g'} & {A'}
      }
    } 
    \mkern40mu = \mkern40mu
    \vcenter{
      \xymatrix@=1.5em{
        {} & {B}\ar[dr]^{v} & {} \\
        {C}\ar[ur]^{\ell}\ar[dr]_{w} & 
        {\scriptstyle\Downarrow\tau}\ar@{}[dr]|(.7){\Downarrow\lambda'} & 
        {B'}\ar[d]^{f'} \\
        {} & {C'}\ar[ur]^{\ell'}\ar[r]_{g'} & {A'}
      }
    } 
\]
  in which the triangles are absolute right liftings and the 2-cell $\tau$ is induced by the universal property of the triangle on the right. We say that the transformation \eqref{eq:induced-mate} is {\em right exact\/} if and only if the induced 2-cell $\tau$ is an isomorphism. This right exactness condition holds if and only if, in the diagram on the left, the whiskered 2-cell $u\lambda$ displays $v\ell$ as the absolute right lifting of $g'w$ through $f'$.
\end{defn}

\begin{obs}\label{obs:2-cat-rep-abs-lifting} Unpacking the definitions in any 2-category, $\lambda \colon f\ell \To g$ defines an absolute right lifting diagram if and only if the induced functor 
\begin{equation}\label{eq:rep-lifting-comma-iso}\hom(X,B)  \comma \hom(X,\ell) \xrightarrow{\cong} \hom(X,f) \comma \hom(X,g) \end{equation} defines an isomorphism of comma categories, natural in $X$. From the 2-categorical universal property of these comma categories, it is now clear that $\lambda \colon f\ell \To g $ defines an absolute right lifting diagram if and only if
\begin{enumerate}[label=(\roman*)]
\item\label{itm:rep-liftings} for each object $X$, the diagram
\begin{equation}\label{eq:rep.lifting.triangle}
    \xymatrix@R=2em@C=3em{ \ar@{}[dr]|(.7){\Downarrow\hom(X,\lambda)} & \hom(X,B) \ar[d]^{\hom(X,f)} \\ \hom(X,C) \ar[r]_{\hom(X,g)} \ar[ur]^{\hom(X,\ell)} & \hom(X,A)}
    \end{equation}
    defines an absolute right lifting diagram in $\eop{Cat}$, and 
    \item moreover, each morphism $e \colon Y \to X$ induces a right exact transformation 
    \[    \xymatrix@C=3em@R=2em{
      {} & {\hom(X,B)}\ar[d]_{\hom(X,f)}\ar[dr]^{\hom(e,X)} & {} \\
      {\hom(X,C)}\ar[r]^{\hom(X,g)}\ar[dr]_{\hom(e,X)} & {\hom(X,A)}\ar[dr]|*+<1pt>{\scriptstyle \hom(e,X)} & 
      {\hom(Y,B)}\ar[d]^{\hom(Y,f)} \\
      {} & {\hom(Y,C)}\ar[r]_{\hom(Y,g)} & {\hom(X,A)}
    }
\] 
    \end{enumerate}
    In fact, by the Yoneda lemma, it suffices to assume in \ref{itm:rep-liftings} that $\hom(X,g)$ admits any absolute right lifting along $\hom(X,f)$ for which the transformation induced from any $e \colon Y \to X$ is right exact. Specialising to the case $X=C$ reveals that these absolute left liftings are represented by a morphism $\ell \colon C \to B$ and 2-cell $\lambda \colon f\ell \To g$.
    \end{obs}

\subsection{Fibred equivalences between comma \texorpdfstring{$\infty$}{infinity}-categories}\label{ssec:comma}

In general the homotopy 2-category of an $\infty$-cosmos  will admit few 2-dimensional limit notions. Nonetheless, the 2-cell representing the horizontal functor in the defining pullback of the comma $\infty$-category of a cospan:
  \begin{equation}\label{eq:comma-arrow}
  \vcenter{
    \xymatrix@=2.5em{
      {f\comma g}\pbexcursion \ar[r]\ar@{->>}[d]_{(p_1,p_0)} &
      {A^\cattwo} \ar@{->>}[d]^{(p_1,p_0)} \\
      {C\times B} \ar[r]_-{g\times f} & {A\times A}
    }}
 \qquad\qquad
\vcenter{
    \xymatrix@=0.6em{
      & f \downarrow g \ar@{->>}[dl]_{p_1} \ar@{->>}[dr]^{p_0} & \\ 
      C \ar[dr]_g \ar@{}[rr] \ar@{=>} ?(0.6);?(0.4) _{\phi_{f,g}}
      & & B \ar[dl]^f \\ 
      & A & }}
\end{equation}
enjoys a weak universal property:

\begin{prop}[{\refIV{obs:ess.unique.1-cell.ind}}]\label{prop:1-univ-comma}
2-cells in the homotopy 2-category $\ho_*\eK$ of the form depicted on the left in the following diagram
  \begin{equation*}
    \vcenter{\xymatrix@R=2em@C=2em{
        {D}\ar[r]^{b}\ar[d]_{c} & {B}\ar[d]^{f} \\
        {C}\ar[r]_{g} & {A}
        \ar@{} "1,2";"2,1"
        \ar@{=>} ?(0.35);?(0.65) ^{\alpha}}}
    \mkern40mu \leftrightsquigarrow \mkern40mu
    \vcenter{\xymatrix@R=2em@C=0.5em{
        {D}\ar[dr]_{(c,b)}\ar[rr]^-{\bar\alpha} &&
        {f\comma g}\ar@{->>}[dl]^{(p_1,p_0)} \\
        & {C\times B} &
      }}
  \end{equation*}
  stand in bijective correspondence to isomorphism classes of 1-cells in the slice 2-category ${(\ho_*\eK)}_{/C\times B}$ as shown on the right, with the action of this bijection, from right to left, is given by composition with the comma square depicted in~\eqref{eq:comma-arrow}. 
\end{prop}

We refer the curious reader interested in more details to \S\refI{subsec:weak-2-limits} or \S\refIV{ssec:abstract}.

\begin{obs}\label{obs:trans.induce.comma}
As an application of the weak universal property, we may generalise the construction of Proposition~\ref{prop:trans-comma}
  to diagrams in $\ho_*\eK$ of the following form:
  \begin{equation}\label{eq:trans.induce.diag}
    \xymatrix@R=1.5em@C=5em{
      {C}\ar[r]^{g}\ar[d]_{c} & {A}\ar[d]_{a}
      \ar@{}[dr]\ar@{<=}?(0.4);?(0.6)^{\alpha}
      & {B}\ar[l]_{f}\ar[d]^{b} \\
      {C'}\ar[r]_{g'}\ar@{}[ur]\ar@{<=}?(0.4);?(0.6)^{\beta} &
      {A'} & {B'}\ar[l]^{f'}
    }
  \end{equation}
  This may be glued onto the square that displays the comma $f\comma g$ \eqref{eq:comma-arrow} to give the pasted square on the left of the following diagram
  \begin{equation}\label{eq:trans.induce.diag2}
    \vcenter{\xymatrix@R=1.5em@C=1.5em{
        & {f\comma g}\ar[dl]_{p_1}\ar[dr]^{p_0} & \\
        {C}\ar[d]_-{c}\ar[dr]^{g}
        \ar@{}[rr]\ar@{<=}?(0.4);?(0.6)^{\phi_{f,g}} &&
        {B}\ar[d]^-{b}\ar[dl]_{f} \\
        {C'}\ar[dr]_{g'} & {A}\ar[d]^-{a} & {B'}\ar[dl]^{f'} \\
        & {A'} &
        \ar@{}"3,1";"3,2"\ar@{<=}?(0.35);?(0.65)^{\beta}
        \ar@{}"3,2";"3,3"\ar@{<=}?(0.35);?(0.65)^{\alpha}
      }}
    \mkern20mu = \mkern20mu
    \vcenter{\xymatrix@R=1.5em@C=1.5em{
        & {f\comma g}\ar@{.>}[d]_{\exists}^{\alpha\downarrow\beta} & \\
        & {f'\comma g'}\ar[dl]_{p'_1}\ar[dr]^{p'_0} & \\
        {C'}\ar[dr]_{g'}\ar@{}[rr]\ar@{<=}?(0.4);?(0.6)^{\phi_{f',g'}} &&
        {B'}\ar[dl]^{f'} \\
        & {A'} &
      }}
  \end{equation}
  which induces a functor $\alpha\comma\beta$ as shown on the right, by the weak 2-universal property of $f'\comma g'$. Indeed, the Proposition \ref{prop:1-univ-comma} tells us that $\alpha\comma\beta$ is a representative of a uniquely determined isomorphism class of such functors in $\eK_{/C'\times B'}$.

  This construction is functorial in the following sense, suppose that we are given a second diagram of the form given in~\eqref{eq:trans.induce.diag}:
  \begin{equation*}
    \xymatrix@R=1.5em@C=5em{
      {C'}\ar[r]^{g'}\ar[d]_{c'} & {A'}\ar[d]_{a'}
      \ar@{}[dr]\ar@{<=}?(0.4);?(0.6)^{\alpha'}
      & {B'}\ar[l]_{f'}\ar[d]^{b'} \\
      {C''}\ar[r]_{g''}\ar@{}[ur]\ar@{<=}?(0.4);?(0.6)^{\beta'} &
      {A''} & {B''}\ar[l]^{f''}
    }
  \end{equation*}
  We may juxtapose these two diagrams vertically to give the following diagram
  \begin{equation*}
    \vcenter{\xymatrix@R=1.5em@C=1.5em{
        & {f\comma g}\ar[dl]_{p_1}\ar[dr]^{p_0} & \\
        {C}\ar[d]_-{c}\ar[dr]^{g}
        \ar@{}[rr]\ar@{<=}?(0.4);?(0.6)^{\phi_{f,g}} &&
        {B}\ar[d]^-{b}\ar[dl]_{f} \\
        {C'}\ar[dr]^*-{\scriptstyle g'}\ar[d]_-{c'} & {A}\ar[d]^-{a} &
        {B'}\ar[dl]_*-{\scriptstyle f'}\ar[d]^-{b'} \\
        {C''}\ar[dr]_{g''} & {A'}\ar[d]^-{a'} & {B''}\ar[dl]^{f''} \\
        & {A''} &
        \ar@{}"3,1";"3,2"\ar@{<=}?(0.35);?(0.65)^{\beta}
        \ar@{}"3,2";"3,3"\ar@{<=}?(0.35);?(0.65)^{\alpha}
        \ar@{}"4,1";"4,2"\ar@{<=}?(0.35);?(0.65)^{\beta'}
        \ar@{}"4,2";"4,3"\ar@{<=}?(0.35);?(0.65)^{\alpha'}
      }}
    \mkern20mu = \mkern20mu
    \vcenter{\xymatrix@R=1.5em@C=1.5em{
        & {f\comma g}\ar@{.>}[d]^{\alpha\downarrow\beta} & \\
        & {f'\comma g'}\ar[dl]_{p'_1}\ar[dr]^{p'_0} & \\
        {C'}\ar[dr]^*-{\scriptstyle g'}\ar[d]_-{c'} &  &
        {B'}\ar[dl]_*-{\scriptstyle f'}\ar[d]^-{b'} \\
        {C''}\ar[dr]_{g''} & {A'}\ar[d]^-{a'} & {B''}\ar[dl]^{f''} \\
        & {A''} &
        \ar@{}"3,1";"3,3"\ar@{<=}?(0.4);?(0.6)^{\phi_{f',g'}}
        \ar@{}"4,1";"4,2"\ar@{<=}?(0.35);?(0.65)^{\beta'}
        \ar@{}"4,2";"4,3"\ar@{<=}?(0.35);?(0.65)^{\alpha'}
      }}
    \mkern20mu = \mkern20mu
    \vcenter{\xymatrix@R=1.5em@C=1.5em{
        & {f\comma g}\ar@{.>}[d]^{\alpha\downarrow\beta} & \\
        & {f'\comma g'}\ar@{.>}[d]^{\alpha'\downarrow\beta'} & \\
        & {f''\comma g''}\ar[dl]_{p''_1}\ar[dr]^{p''_0} & \\
        {C''}\ar[dr]_{g''}\ar@{}[rr]\ar@{<=}?(0.4);?(0.6)^{\phi_{f'',g''}} &&
        {B''}\ar[dl]^{f''} \\
        & {A''} &
      }}
  \end{equation*}
  in which we've applied the induction process depicted in~\eqref{eq:trans.induce.diag2} twice. Alternatively, take the same diagram on the left and start by forming the pasting composites of each column of squares
  \begin{equation*}
    \vcenter{\xymatrix@R=1.5em@C=1.5em{
        & {f\comma g}\ar[dl]_{p_1}\ar[dr]^{p_0} & \\
        {C}\ar[d]_-{c}\ar[dr]^{g}
        \ar@{}[rr]\ar@{<=}?(0.4);?(0.6)^{\phi_{f,g}} &&
        {B}\ar[d]^-{b}\ar[dl]_{f} \\
        {C'}\ar[dr]^*-{\scriptstyle g'}\ar[d]_-{c'} & {A}\ar[d]^-{a} &
        {B'}\ar[dl]_*-{\scriptstyle f'}\ar[d]^-{b'} \\
        {C''}\ar[dr]_{g''} & {A'}\ar[d]^-{a'} & {B''}\ar[dl]^{f''} \\
        & {A''} &
        \ar@{}"3,1";"3,2"\ar@{<=}?(0.35);?(0.65)^{\beta}
        \ar@{}"3,2";"3,3"\ar@{<=}?(0.35);?(0.65)^{\alpha}
        \ar@{}"4,1";"4,2"\ar@{<=}?(0.35);?(0.65)^{\beta'}
        \ar@{}"4,2";"4,3"\ar@{<=}?(0.35);?(0.65)^{\alpha'}
      }}
    \mkern20mu = \mkern20mu
    \vcenter{\xymatrix@R=1.5em@C=1.5em{
        & {f\comma g}\ar[dl]_{p_1}\ar[dr]^{p_0} & \\
        {C}\ar[dd]_-{c'c}\ar[dr]^{g}
        \ar@{}[rr]\ar@{<=}?(0.4);?(0.6)^{\phi_{f,g}} &&
        {B}\ar[dd]^-{b'b}\ar[dl]_{f} \\
        & {A}\ar[dd]^(0.7){a'a} & \\
        {C''}\ar[dr]_{g''} & & {B''}\ar[dl]^{f''} \\
        & {A''} &
        \ar@{}@<1.5ex>"4,1";"3,2"
        \ar@{<=}?(0.35);?(0.65)_(0.2){\beta'r\cdot p'\beta}
        \ar@{}@<-1.5ex>"3,2";"4,3"
        \ar@{<=}?(0.35);?(0.65)^(0.2){p'\alpha\cdot\alpha'q}
      }}
    \mkern20mu = \mkern20mu
    \vcenter{\xymatrix@R=1.5em@C=1.5em{
        & {f\comma g}
        \ar@{.>}[dd]^{(p'\alpha\cdot\alpha'q)\downarrow
          (\beta'r\cdot p'\beta)} & \\
        &  & \\
        & {f''\comma g''}\ar[dl]_{p''_1}\ar[dr]^{p''_0} & \\
        {C''}\ar[dr]_{g''}\ar@{}[rr]\ar@{<=}?(0.4);?(0.6)^{\phi_{f'',g''}} &&
        {B''}\ar[dl]^{f''} \\
        & {A''} &
      }}
  \end{equation*}
  then apply the induction process depicted in~\eqref{eq:trans.induce.diag2} only once. It follows that the functors $(\alpha'\comma\beta') (\alpha\comma\beta)$ and $(p'\alpha\cdot\alpha'q) \downarrow(\beta'r\cdot p'\beta)$ are both induced by the same diagram on the left under the weak 2-universal property of $f''\comma g''$, consequently there exists an isomorphism between them in $\eK_{/C''\times B''}$.
\end{obs}

For instance, suppose that we are given triangle as in~\eqref{eq:abs-right-lifting} of Definition~\ref{defn:absolute-right-lifting}, then we may apply the construction of Observation \ref{obs:trans.induce.comma} to the diagram
  \begin{equation*}
    \xymatrix@=1.5em{
      {C}\ar[r]^{\ell}\ar@{=}[d] & {B}\ar[d]^{f} &
      {B}\ar@{=}[l]\ar@{=}[d] \\
      {C}\ar[r]_{g} & {A} & {B}\ar[l]^{f}
      \ar@{}"1,2";"2,1"\ar@{=>}?(0.35);?(0.65)_{\lambda}
    }
  \end{equation*} 
  to give a functor $B\comma\lambda\colon B\comma \ell\to f\comma g$ fibred over $C\times B$, which detects whether the diagram \eqref{eq:abs-right-lifting} is absolute right lifting:

\begin{prop}[\refI{prop:absliftingtranslation}]\label{prop:absliftingtranslation}
  Given a cospan $C \xrightarrow{g} A \xleftarrow{f} B$ then there exists an equivalence $f\comma g \simeq B \comma \ell$ over $C \times B$
 if and only if there exists an absolute right lifting diagram of the following form
  \begin{equation*}
    \xymatrix{
      \ar@{}[dr]|(.7){\Downarrow\lambda} & B \ar[d]^f \\
      C \ar[ur]^\ell \ar[r]_g & A}
  \end{equation*}
\end{prop}

\begin{rmk}\label{rmk:trans-induct-rel}
  The construction in Observation~\ref{obs:trans.induce.comma} may clearly be regarded as being a generalisation of that discussed in Proposition~\ref{prop:trans-comma}. There is, however, a good reason for distinguishing them in the way we have. On the one hand, the construction in Proposition~\ref{prop:trans-comma}  relies only upon the strict universal property of the comma construction, and so it delivers a uniquely determined map. On the other hand, the construction in Observation~\ref{obs:trans.induce.comma} depends upon a choice of functors representing a 2-cell and so is only defined up to isomorphism. It is, nevertheless, the case that when both constructions apply the first provides a specific choice which is certainly a member of the isomorphism class determined by the second.

  This distinction becomes particularly important in situations where it is important to infer that the map induced by a transformation of cospans is an isofibration. As we have seen such results hold for the construction of Proposition~\ref{prop:trans-comma}, but they cannot reasonably be expected to hold for that of Observation~\ref{obs:trans.induce.comma}, simply because the class of isofibrations is not closed under isomorphisms at the 2-cell level.
\end{rmk}


\section{Cartesian fibrations and modules}\label{sec:cartesian}

In \S\ref{ssec:cartesian} we review the notions of \emph{cartesian} and \emph{cocartesian fibrations} between $\infty$-cat\-e\-gor\-ies, which combine to yield a two-sided structure we refer to as a \emph{module} between $\infty$-cat\-e\-gor\-ies; elsewhere this notion is called a \emph{profunctor} or \emph{correspondence}. Again, we neglect to describe the full theory developed in \cite{RiehlVerity:2015fy}, only recalling those aspects necessary to understand the Yoneda embeddings of \S\ref{sec:comprehension-yoneda}, which represent an element of an $\infty$-category $A$ as a module from $A$ to $1$ or $1$ to $A$.

In \S\ref{sec:limits-colimits-fibrations}, we prove that cartesian or cocartesian fibrations over a fixed base assemble into an $\infty$-cosmos, with cosmological structure inherited from the sliced $\infty$-cosmos $\eK_{/B}$. This gives a new context for the results of \cite{RiehlVerity:2017cc} and allows to immediately apply the main theorem of \cite{RiehlVerity:2018rq} to the large quasi-categories of cocartesian fibrations, cartesian fibrations, or modules with a fixed base; see Examples \ref{ex:comp-fibs} and \ref{ex:comp-mods}.

\subsection{Cartesian fibrations and modules}\label{ssec:cartesian}

Cocartesian fibrations are isofibrations $p \colon E \tfib B$ in an $\infty$-cosmos $\eK$ whose fibres depend functorially on the base, in a sense described by a lifting property for certain 2-cells.  \emph{Cartesian fibrations} are cocartesian fibrations in the dual $\infty$-cosmos $\eK\co$ of Definition \ref{defn:dual-cosmoi}.

 Cocartesian fibrations can be characterised internally to the $\infty$-cosmos via adjoint functors involving comma $\infty$-categories, for which we now establish notation.
 
 \begin{ntn}\label{ntn:comma-adjoints}
 For any isofibration $p \colon E \tfib B$, there exist canonically defined functors 
 \[
\xymatrix{ E \ar@{-->}[dr]^i  \ar@{=}@/_2ex/[ddr] \ar[r]^\Delta & E^\cattwo \ar@{->>}[dr]^{p^\cattwo} & & E^\cattwo  \ar@{-->}[dr]^k \ar@{->>}@/_2ex/[ddr]_{p_0} \ar@{->>}@/^2ex/[drr]^{p^\cattwo}\\ & p \comma B \ar@{->>}[r] \pbexcursion \ar@{->>}[d]_{p_0} & B^\cattwo \ar@{->>}[d]^{p_0} & & p \comma B \ar@{->>}[r] \pbexcursion \ar@{->>}[d]_{p_0} & B^\cattwo \ar@{->>}[d]^{p_0} \\ & E \ar@{->>}[r]_p & B & & E \ar@{->>}[r]_p & B}
\]
Note, $k$ is the map $\comma(E,p,p)\colon E^\cattwo\to p\comma B$, so by Proposition~\ref{prop:trans-comma} it is an isofibration.
  \end{ntn}
    
   \begin{defn}[{\refIV{thm:cart.fib.chars}}]\label{defn:cocart-fibration}
An isofibration $p \colon E \tfib B$ is a \emph{cocartesian fibration} if either of the following equivalent conditions hold:
  \begin{enumerate}[label=(\roman*)]
        \item\label{itm:cocart.fib.chars.ii} The functor $i\colon E\to p\comma B$ admits a left adjoint in the slice $\infty$-cosmos $\eK_{/B}$:
    \begin{equation}\label{eq:cocartesian.fib.adj}
      \xymatrix@R=2em@C=3em{
        {E}\ar@{->>}[dr]_{p} \ar@/_0.8pc/[]!R(0.5);[rr]_{i}^{}="a" & &
        {p \comma B}\ar@{->>}[dl]^{p_1} \ar@{-->}@/_0.8pc/[ll]!R(0.5)_{\ell}^{}="b" 
        \ar@{}"a";"b"|{\bot} \\
        & B &
      }
    \end{equation}
  \item\label{itm:cocart.fib.chars.iii} The functor $k\colon E^\cattwo\to p\comma B$ admits a left adjoint right inverse in $\eK$:
    \begin{equation}\label{eq:cartesian.isosect.adj}
    \xymatrix@C=6em{
      {E^\cattwo}\ar@/_0.8pc/[]!R(0.6);[r]!L(0.45)_{k}^{}="u" &
      {p \comma B}\ar@{-->}@/_0.8pc/[]!L(0.45);[l]!R(0.6)_{\bar{\ell}}^{}="t"
      \ar@{}"u";"t"|(0.6){\bot}
    }
    \end{equation}
  \end{enumerate}
\end{defn}

Recall from Lemma \refVII{lem:groupoidal-object}, that an object $E \in \eK$ in an $\infty$-cosmos is \emph{groupoidal} if every functor space $\Fun_{\eK}(X,E)$ with codomain $E$ is a Kan complex. Intuitively, a groupoidal cocartesian fibration is a cocartesian fibration whose fibres are groupoidal $\infty$-categories, though for exotic $\infty$-cosmoi the following definition is somewhat stronger:

\begin{defn}[{\refIV{prop:iso-groupoidal-char}}]
  An isofibration $p\colon E\tfib B$  is a {\em groupoidal cocartesian fibration\/} if and only if either of the following equivalent conditions hold:
\begin{enumerate}[label=(\roman*)]
  \item It  is a cartesian fibration and it is groupoidal as an object of the slice $\eK_{/B}$.
  \item The functor $k\colon E^\cattwo\to p \comma B$ is an equivalence.
  \end{enumerate}
\end{defn}

\begin{defn}[{\refIV{thm:cart.fun.chars}}]\label{defn:cartesian-functor}
  Given two cocartesian fibrations $p\colon E\tfib B$ and
  $q\colon F\tfib A$ in $\eK$, then a pair of functors $(g,f)$ in the
  following commutative square
  \begin{equation}\label{eq:cart.fun}
    \xymatrix{{F}\ar[r]^{g}\ar@{->>}[d]_-{q} & {E} \ar@{->>}[d]^-{p} \\ {A}\ar[r]_{f} & {B}}
  \end{equation}
  comprise a {\em cartesian functor\/} if and only if the mate of either (and
  thus both) of the commutative squares
  \begin{equation*}
    \xymatrix{{F}\ar[r]^{g}\ar[d]_{i} & {E} \ar[d]^{i} \\
      {q\comma A}\ar[r]_{\comma(g,f, f)} & {p\comma B}}
    \mkern40mu
    \xymatrix{{F^{\cattwo}}\ar[r]^{g^{\cattwo}}\ar@{->>}[d]_{k} &
      {E^{\cattwo}} \ar@{->>}[d]^{k} \\ {q\comma A}\ar[r]_{\comma (g,f,f)} & {p\comma B}}
  \end{equation*}
  under the adjunctions of Definition \ref{defn:cocart-fibration} is an isomorphism. 
\end{defn}


Pullbacks provide an important source of cartesian functors.

\begin{prop}[{\refIV{prop:cart-fib-pullback} and \refIV{cor:groupoidal-pullback}}]\label{prop:cart-fib-pullback}
Consider a simplicial pullback
\[ \xymatrix{ F \pbexcursion \ar@{->>}[d]_q \ar[r]^g & E \ar@{->>}[d]^p \\ A \ar[r]_f & B}\] in $\eK$. If $p \colon E \tfib B$ is a (groupoidal) cocartesian fibration, then $q \colon F \tfib A$ is a (groupoidal) cocartesian fibration and the pullback square defines a cartesian functor. 
\end{prop}

A \emph{module} from $A$ to $B$ is an $\infty$-category upon which $A$ acts covariantly and $B$ acts contravariantly. The calculus of modules, developed in \S\refV{sec:virtual}, is not needed here.

\begin{defn}\label{defn:module} In an $\infty$-cosmos $\eK$, a \emph{module $E$
    from $A$ to $B$} is given by an
  isofibration $(q,p) \colon E \tfib A \times B$ such that
  \begin{enumerate}[label=(\roman*)]
\item\label{itm:cartesian-on-the-right} $\vcenter{\xymatrix@=1em{ E \ar@{->>}[dr]_q \ar[rr]^-{(q,p)} & & A \times B \ar@{->>}[dl]^{\pi_1} \\ & A}}$ is a \emph{cartesian fibration} in  $\eK_{/A}$; informally, ``$B$ acts on the right of $E$, over $A$.'' 
\item\label{itm:cocartesian-on-the-left}  $\vcenter{\xymatrix@=1em{ E \ar@{->>}[dr]_p \ar[rr]^-{(q,p)} & & A \times B \ar@{->>}[dl]^{\pi_0} \\ & B}}$  is a \emph{cocartesian fibration} in $\eK_{/B}$; informally,  ``$A$ acts on the left of $E$, over $B$.'' 
\item\label{itm:groupoidal-fibers} $(q,p) \colon E \tfib A \times B$ is groupoidal as an object in $\eK_{/A \times B}$.
\end{enumerate}
Note that \ref{itm:groupoidal-fibers} implies in particular that  $(q,p) \colon E \tfib A \times B$ has groupoidal fibres and that the sliced (co)cartesian fibrations in \ref{itm:cartesian-on-the-right} and \ref{itm:cocartesian-on-the-left} are both groupoidal.
\end{defn}

\begin{ex}[groupoidal (co)cartesian fibrations as
  modules]\label{ex:modules-over-1} A module from $1$ to $A$ is exactly a
  groupoidal cartesian fibration over $A$, while a module from $A$ to $1$ is
  exactly a groupoidal cocartesian fibration over $A$.
\end{ex}

\begin{ex}[{\refV{prop:hom-is-a-module}}]\label{ex:hom-is-a-module} For any
  $\infty$-category $A$ the arrow $\infty$-category $A^\cattwo$ and its associated projections $(p_1,p_0) \colon A^\cattwo \tfib A \times A$ give a module $A^\cattwo$ from $A$ to $A$ called the \emph{unit} module on $A$.
\end{ex}

\begin{defn}[bi-fibres of modules]
  Suppose that $E$ is a module from $A$ to $B$ and that $g\colon C\to A$ and
  $f\colon D\to B$ are functors in $\eK$. We may form the following
  pullback
  \begin{equation*}
    \xymatrix@=2em{
      {E(f,g)}\pbexcursion\ar[r]\ar@{->>}[d]_{(s,r)} &
      {E}\ar@{->>}[d]^{(q,p)} \\
      {C\times D}\ar[r]_-{(g,f)} & {A\times B}
    }
  \end{equation*}
  and recall from~\refV{prop:two-sided-pullback} that the left-hand vertical
  again defines a module $E(f,g)$ from $C$ to $D$. We call this the
  \emph{bi-fibre\/} of $E$ over the pair $(f,g)$.
\end{defn}

\begin{ex}[hom-space modules]\label{ex:hom.spaces}
  Given an $\infty$-category $A$ it is sometimes illuminating to write $\Hom_A$ for the module $A^{\cattwo}$ and refer to it as the \emph{hom-space module}. Functors $a\colon X\to A$ and $a'\colon X'\to A$ define a pair of generalised elements of $A$, so we may regard the  bi-fibre $\Hom_{A}(a,a')$ of $\Hom_A$ as being the hom-space of arrows from $a$ to $a'$ in $A$; this is, of course, isomorphic to the comma object $a\comma a'$. In the particular case of a quasi-category $\qA$ and objects $a,a'\colon \Del^0\to \qA$, the hom-space $\Hom_{\qA}(a',a)$ is simply the usual hom-space Kan complex of $\qA$, as for instance discussed in~\refVI{defn:qcat-comma-hom} or \cite[\S 1.2.2]{Lurie:2009fk}.

  Given a functor $f\colon A\to B$ of $\infty$-categories in $\eK$ the commutative diagram
  \begin{equation*}
    \xymatrix@R=1.5em@C=2em{
      {A}\ar@{=}[r]\ar@{=}[d] & {A}\ar[d]^{f} & {A}\ar@{=}[l]\ar@{=}[d] \\
      {A}\ar[r]_{f} & {B} & {A}\ar[l]^{f}
    }
  \end{equation*}
  induces a functor $\bar{f}\colon\Hom_A\to\Hom_B(f,f)$ between modules in the slice $\eK_{/A\times A}$, as discussed in Proposition~\ref{prop:trans-comma}; this is called the action of $f$ on the hom-spaces of $A$. Pulling this map back to $\eK_{/X\times X'}$ we obtain a map $f_{a,a'}\colon\Hom_A(a,a')\to\Hom_B(fa,fa')$ which we regard as being the action of $f$ on the hom-space between $a$ and $a'$; it is otherwise describable as the functor induced as in Proposition~\ref{prop:trans-comma} by the commutative diagram:
  \begin{equation*}
    \xymatrix@R=1.5em@C=2em{
      {X'}\ar[r]^{a'}\ar@{=}[d] & {A}\ar[d]^{f} & {X}\ar[l]_{a}\ar@{=}[d] \\
      {X'}\ar[r]_{fa'} & {B} & {X}\ar[l]^{fa}
    }
  \end{equation*} 
  In the case of a quasi-category $\qA$ and objects $a,a'\colon\Del^0\to\qA$ this construction simply gives the usual action of $f$ on the hom-space Kan complex between $a$ and $a'$.

  Observe also that if we apply the construction of Observation~\ref{obs:trans.induce.comma} to a diagram of the following form
  \begin{equation*}
    \xymatrix@R=1.5em@C=2em{
      {Y}\ar[r]^{b}\ar@{=}[d] & {A}\ar@{=}[d] & {X}\ar[l]_{a}\ar@{=}[d] \\
      {Y}\ar[r]_{b'} & {A} & {X}\ar[l]^{a'}
      \ar@{}"2,1";"1,2"\ar@{<=}?(0.4);?(0.6)^{\beta}
      \ar@{}"1,2";"2,3"\ar@{<=}?(0.4);?(0.6)^{\alpha}
    }
  \end{equation*}
  then we obtain a functor $\Hom_A(\alpha,\beta)\colon\Hom_A(a,b)\to\Hom_A(a',b')$ in $\eK_{/X\times Y}$, this being the action on hom-spaces given by pre-composition with $\alpha$ and post-composition with $\beta$. When $\qA$ is a quasi-category and $\alpha$ and $\beta$ are 2-cells between vertices $\Del^0\to\qA$, then they correspond to isomorphism classes of objects in the Kan complexes $\Hom_{\qA}(a',a)$ and $\Hom_{\qA}(b,b')$ (or equivalently of arrows in $\qA$). Then the functor $\Hom_A(\alpha, \beta)$ is a member of the isomorphism class of maps of hom-space Kan complexes induced by choices of pre-/post-composition in $\qA$ by representatives of $\alpha$ and $\beta$.
\end{ex}

We shall restrict our use of the notation $\Hom_{\qA}(a,a')$ to situations where $\qA$ is a quasi-category and $a$ and $a'$ is a pair of its objects. In all other cases, we shall continue to use our established comma notation $a\comma a'$ for this structure.

\begin{lem}[{\refV{lem:two-sided-groupoidal-cells}, \refV{lem:span-maps-cartesian}}]\label{lem:module.legs} Let $E$ be a module from $A$ to $B$ given by an isofibration $(q,p)\colon E\tfib A\times B$.
\begin{enumerate}[label=(\roman*)]
\item The right-hand leg $p\colon E\tfib B$ is a cartesian fibration, and the left-hand leg $q\colon E\tfib A$ is a cocartesian fibration.
\item  Given a second module determined  by an isofibration $(s,r)\colon F\tfib A\times B$, and a functor
  \begin{equation*}
    \xymatrix@C=1em@R=2em{
      {E}\ar[rr]^{g}\ar@{->>}[dr]_{(q,p)} && {F}\ar@{->>}[dl]^{(s,r)} \\
      & {A\times B} &
    }
  \end{equation*} then $g$ defines a cartesian functor between the right-hand cartesian fibrations and also between the left-hand cocartesian fibrations.
  \end{enumerate}
  \end{lem}

\begin{defn}\label{defn:rep-modules} Given a functor $f \colon A \to B$ between $\infty$-categories, we say that:
\begin{enumerate}[label=(\roman*)]
\item\label{itm:covar-mod} a module $E$ from $A$ to $B$ is \emph{covariantly represented\/} by $f$ if it is equivalent, over $A\times B$, to the module $B\comma f$, and 
\item\label{itm:contra-mod} a module $E$ from $B$ to $A$ is \emph{contravariantly represented\/} by $f$ if it is equivalent to the module $f\comma B$.
\end{enumerate}
\end{defn}

\subsection{\texorpdfstring{$\infty$}{infinity}-cosmoi of (co)cartesian fibrations}\label{sec:limits-colimits-fibrations}

Let $\pbshape$ denote the category indexing a cospan and write $\eK^\pbshape$ for the simplicially enriched category of cospans in an $\infty$-cosmos $\eK$. Thus, objects of $\eK^\pbshape$ are pairs of maps
  \begin{equation}\label{eq:obj-K-pbshape}
    \xymatrix@=1.5em{
      {} & {B}\ar[d]^f \\
      {C}\ar[r]_g & {A}
    }
  \end{equation}
  and 0-arrows are natural transformations
  \begin{equation}\label{eq:obj-K-pbshape-trans}
    \xymatrix@=1.5em{
      {} & {B}\ar[d]_{f}\ar[dr]^{v} & {} \\
      {C}\ar[r]^{g}\ar[dr]_{w} & {A}\ar[dr]|*+<1pt>{\scriptstyle u} & 
      {B'}\ar[d]^{f'} \\
      {} & {C'}\ar[r]_{g'} & {A'}
    }
  \end{equation}

\begin{obs}
  As with any enriched functor category, $\eK^\pbshape$ inherits limits pointwise from $\eK$. In particular, it  has limits weighted by flexible weights, which are then preserved by the simplicial projection functors $\eK^\pbshape\to\eK$  which evaluate at any object of $\pbshape$.
  \end{obs}

In \refIII{sec:absolute}, we introduced a subcategory of $\eK^\pbshape$ (considered here in the dual) whose objects are those diagrams~\eqref{eq:obj-K-pbshape} in which $g$ admits an absolute left lifting through $f$ in the homotopy 2-category of $\eK$, in the sense of Definition \ref{defn:absolute-right-lifting} but with direction of the 2-cells reversed. The 0-arrows in this simplicially enriched category are the \emph{left exact transformations} dual to Definition \ref{defn:abs-lift-trans-right-exactness}.

\begin{defn}
  Let $\eK_{\ell}^\pbshape$ denote the simplicial subcategory of $\eK^\pbshape$ with:
  \begin{itemize}
    \item {\bf objects} those diagrams~\eqref{eq:obj-K-pbshape} in which $g$ admits an absolute left lifting through $f$, and
    \item {\bf $n$-arrows} those $n$-arrows of $\eK^\pbshape$ whose vertices are left exact transformations.
  \end{itemize}
\end{defn}

\begin{lem}\label{lem:absolute-lifting-repleteness}
The inclusion $\eK^\pbshape_\ell \inc \eK^\pbshape$ is replete up to equivalence. That is, given any natural transformation \eqref{eq:obj-K-pbshape-trans} between cospans in $\eK$ that is a pointwise equivalence, then if either cospan admits an absolute left lifting, the other also admits an absolute left lifting defined in such a way that the transformation is left exact. Furthermore, any natural transformation that it pointwise equivalent to a left exact transformation is itself left exact.
\end{lem}
\begin{proof}
A pointwise equivalence \eqref{eq:obj-K-pbshape-trans} induces an isomorphism between the hom-categories appearing in the 2-categorical representable characterisation \eqref{eq:rep-lifting-comma-iso} of absolute left lifting diagrams: here if $g$ admits an absolute lifting $\ell \colon C \to B$ through $f$, then the absolute left lifting of $g'$ through $f'$ is defined to be the composite of $\ell$ with $v$ and the equivalence inverse to $w$. As the isomorphism of comma categories for $g'$ and $f'$ is defined by transporting the isomorphism for $g$ and $f$ along $u$, $v$, and $w$, this construction makes the transformation  \eqref{eq:obj-K-pbshape-trans}  exact. The argument for repleteness of 0-arrows is similar.
\end{proof}

In \cite{RiehlVerity:2013cp}, we defined these simplicial categories in order to prove the following proposition:

\begin{prop}[{\refIII{prop:abslifts-in-limits}}]\label{prop:abslifts-in-limits}
The simplicial subcategory $\qCat^\pbshape_{\ell} \inc \qCat^{\pbshape}$ is closed under flexible weighted limits.
\end{prop}

The proof of Proposition \refIII{prop:abslifts-in-limits} made use of Theorem \refI{thm:pointwise}, which provides a special pointwise characterisation of absolute left lifting diagrams between quasi-categories, but nevertheless this result generalises from $\qCat$ to any $\infty$-cosmos. The proof uses the Yoneda lemma to demonstrate that there is a pullback of simplicial categories:
\[ 
\xymatrix{ \eK^\pbshape_\ell \ar@{^(->}[d] \ar@{^(->}[r]^-y \pbexcursion & (\qCat^\pbshape_\ell)^{\eK\op} \ar@{^(->}[d] \\ \eK^\pbshape  \ar@{^(->}[r]^-y & (\qCat^\pbshape)^{\eK\op}}\]
which tells us that absolute left lifting diagrams in $\eK$ can be characterised representably in $\qCat$ in analogy with the 2-categorical representable characterisation of Observation \ref{obs:2-cat-rep-abs-lifting}. Here, however, we require only a corollary of Proposition \refIII{prop:abslifts-in-limits} which generalises more easily to a quasi-categorically enriched category, without any of the additional structures of an $\infty$-cosmos. So rather than presenting the generalisation of Proposition \refIII{prop:abslifts-in-limits}, we instead pursue a proof of Proposition \ref{prop:radj-limits}, which characterizes flexible weighted limits in the quasi-categorically enriched category of of right adjoint functors and exact squares.

\begin{defn}
A commutative square between two parallel 0-arrows $p$ and $q$ admitting left adjoints is \emph{exact} if and only if its mate is an isomorphism.
  \begin{equation}\label{eq:exact-square}
    \xymatrix{
      E \ar[d]^p \ar[r]^f & F \ar[d]_q  & &
      E \ar[r]^f & F  & \ar@{}[d]|{\displaystyle \defeq} &
      \ar@{}[dr]|(.7){\Uparrow\eta}  &
      E \ar[r]^f \ar[d]^p & F  \ar@{}[dr]|(.3){\Uparrow\epsilon}
      \ar[d]_q \ar@{=}[r] & F \\
      B \ar@{-->}@/^2ex/[u]^\ell_\dashv \ar[r]_g &
      A \ar@{-->}@/_2ex/[u]_k^\vdash & &
      B \ar@/^2ex/[u]^\ell \ar[r]_g & A \ar@/_2ex/[u]_k
      \ar@{}[ul]\ar@{=>}?(0.35);?(0.65)_{\cong}^{\epsilon_{f\ell}\cdot kg\eta} & &
      B \ar@{=}[r] \ar[ur]^\ell & B \ar[r]_g & A \ar[ur]_k & }
  \end{equation}
\end{defn}

\begin{defn} For a quasi-categorically enriched category $\eC$, let $\Radj(\eC) \inc \eC^\cattwo$ denote the simplicial subcategory of the simplicially-enriched category of arrows $\eC^\cattwo$ with:
  \begin{itemize}
    \item {\bf objects} those morphisms that admit a left adjoint in the homotopy 2-category of $\eC$, and
    \item {\bf $n$-arrows} those $n$-arrows of $\eC^\cattwo$ whose $0$-arrow vertices are exact squares.
  \end{itemize}
\end{defn}

Substituting the dual $\eC\co$ of Definition \ref{defn:dual-cosmoi} for $\eC$ exchanges right adjoints with left adjoints: $\Radj(\eC\co) \cong \Ladj(\eC)$. In this way, all of the results appearing below will have duals with ``left'' exchanged with ``right.'' 

\begin{prop}\label{prop:qcat-radj-limits}
The simplicial subcategory admits and the inclusion $\Radj(\qCat) \inc \qCat^\cattwo$ is closed under flexible weighted limits. Moreover, the inclusion $\Radj(\qCat) \inc \qCat^\cattwo$ is replete up to equivalence on objects and 0-arrows.
\end{prop}
\begin{proof}
By Example \ref{ex:radj-lifting}, $p \colon E \to B$ admits a left adjoint if and only if the cospan
\[
\xymatrix{ \ar@{}[dr]|(.7){\Uparrow\eta} & E \ar[d]^p \\ B \ar@{=}[r] \ar@{-->}[ur]^\ell & B}\] admits an absolute left lifting. A square \eqref{eq:exact-square} is exact if and only if the triple $(g,f,g)$ defines a left exact transformation from the cospan $(p,\id_B)$ to the cospan $(q, \id_A)$, the induced 2-cell \eqref{eq:induced-mate} being the desired mate. This proves that we have a pullback diagram of simplicial categories
\[
\xymatrix{ \Radj(\qCat) \ar[r] \ar[d] \pbexcursion & \qCat^\pbshape_\ell \ar@{^(->}[d] \\ \qCat^\cattwo \ar[r]_C & \qCat^\pbshape}\] where the functor $C$ is given by precomposing with the surjective functor $\pbshape \to \cattwo$ that sends one of the arrows to the identity on the codomain object. Proposition \ref{prop:abslifts-in-limits} proves that the right-hand vertical functor creates flexible weighted limits, and since these are  pointwise defined in $\qCat$, the functor $C$ preserves them. A stricter special case of Lemma \ref{lem:absolute-lifting-repleteness} tells us  that the right-hand vertical functor (and hence also the left-hand vertical functor) is an isofibration. It follows as in the proofs of Lemma \refVI{lem:computad-colimits} and Proposition \refVII{prop:gpdal-infty-cosmos} that $\Radj(\qCat)$ admits and $\Radj(\qCat) \to \qCat^\cattwo$ creates all flexible weighted limits. 

Now this pullback and Lemma \ref{lem:absolute-lifting-repleteness} tells us that the inclusion $\Radj(\qCat)\inc\qCat^\cattwo$ is similarly replete.
\end{proof}

The extension of Proposition \ref{prop:qcat-radj-limits} to any quasi-categorically enriched category $\eC$ makes use of the following representable characterisation of adjunctions. 

\begin{prop}\label{prop:representable-adjunctions} Let $\eC$ be a quasi-categorically enriched category. Then \[ 
\xymatrix{ \Radj(\eC) \ar@{^(->}[d] \ar@{^(->}[r]^-y \pbexcursion & \Radj(\qCat)^{\eC\op} \ar@{^(->}[d] \\ \eC^\cattwo  \ar@{^(->}[r]^-y & (\qCat^\cattwo)^{\eC\op}}\] is a pullback diagram of simplicial categories.
  \end{prop}

\begin{proof} On objects, the pullback asserts that 
\begin{enumerate}[label=(\roman*)]
\item\label{itm:rep-radj} A functor $u \colon A \to B$ admits a left adjoint if and only if for each $X \in \eC$, the functor $\Map_{\eC}(X,u) \colon \Map_{\eC}(X,A) \to \Map_{\eC}(X,B)$ admits a left adjoint and moreover for each $e \colon Y \to X$, the square
\[
  \xymatrix@C=3em{ \Map_{\eC}(X,A) \ar[d]^-{\Map_{\eC}(X,u)} \ar[r]^{\Map_{\eC}(e,A)} & \Map_{\eC}(Y,A) \ar[d]_{\Map_{\eC}(Y,u)}  \\ \Map_{\eC}(X,B) \ar@{-->}@/^2ex/[u]_\dashv^{L_X} \ar[r]_-{\Map_{\eC}(e,B)} & \Map_{\eC}(Y,B) \ar@{-->}@/_2ex/[u]^\vdash_{L_Y} }
  \]
  is exact.
  \end{enumerate}
  On 0-arrows, the pullback asserts that moreover:
    \begin{enumerate}[label=(\roman*), resume]
  \item\label{itm:rep-exact} a commutative square between functors admitting left adjoints is exact in $\eC$ if and only if each $\Map_{\eC}(X,-)$ carries this square to an exact square between quasi-categories.
   \end{enumerate}
The vertical inclusions are full on $n$-arrows for $n>0$, so we need only prove the statements \ref{itm:rep-radj} and \ref{itm:rep-exact}.

The preservation halves of each statement are clear, and the converse direction of \ref{itm:rep-exact} is easy: to test whether a square \eqref{eq:exact-square} is exact, we need to show that a single 2-cell is invertible, and this 2-cell is represented by a 1-arrow in $\Map_{\eC}(B,F)$. If the square in the image of $\Map_{\eC}(B,-)$ is exact, then this 1-arrow, as a component of an invertible natural transformation, is an isomorphism, which is what we wanted to show.

For the converse to \ref{itm:rep-radj}, we use the left adjoint $L_B \colon \Map_{\eC}(B,B) \to \Map_{\eC}(B,A)$   to define $\ell \colon B \to A$ to be $L_B(\id_B)$. Exactness tells us that for any $b \colon X \to B$, $L_X(b) \cong \ell b \colon X \to A$, from which we infer that the adjointness of $L_X \dashv \Map_{\eC}(X,u)$ implies that $\Map_{\eC}(X,\ell) \dashv \Map_{\eC}(X,u)$ for all $X$, with any $e \colon Y \to X$ defining an \emph{adjunction morphism}, by which we mean a strict exact transformation. Specialising to $X=A$ and $X=B$, we can extract 1-simplices representing the counit and unit of $\ell \dashv u$ as components of the counit and unit at $\id_A$ and $\id_B$. The adjunction morphisms corresponding to precomposition with $u \colon A \to B$ and $\ell \colon B \to A$ are used to verify the triangle identities.
\end{proof}

\begin{prop}\label{prop:radj-limits}
For any quasi-categorically enriched category $\eC$, $\Radj(\eC)\inc \eC^\cattwo$ is closed under any flexible weighted homotopy limits in $\eC^\cattwo$ that are preserved by the evaluation functors $\dom,\cod \colon \eC^\cattwo \to \eC$. Moreover, if the flexible weighted limit in $\eC^\cattwo$ is strict then so is the flexible weighted limit in $\Radj(\eC)$.
\end{prop}
\begin{proof} 
Consider a diagram
 \[ 
\xymatrix{\eA \ar[r]^-D &  \Radj(\eC) \ar@{^(->}[d] \ar@{^(->}[r]^-y \pbexcursion & \Radj(\qCat)^{\eC\op} \ar@{^(->}[d] \\ & \eC^\cattwo  \ar@{^(->}[r]^-y & (\qCat^\cattwo)^{\eC\op}}\] so that the diagram $D \colon \eA \to \eC^\cattwo$ admits a flexible $W$-weighted homotopy limit $u \colon L \to K$. The hypothesis that this limit is preserved by the domain and codomain functors tells us that the limit cone induces the horizontal natural equivalences
\begin{equation}\label{eq:pointwise-nat-equiv}
\xymatrix{ \Map_{\eC}(X,L) \ar[r]^-\sim \ar[d]_{\Map_{\eC}(X,u)} & \{W, \Map_{\eC}(X,\dom D-) \} \ar[d] \\
  \Map_{\eC}(X,K) \ar[r]^-\sim & \{W, \Map_{\eC}(X,\cod D-)\} \ar@{-->}@/_2ex/[u]^\vdash}
  \end{equation}
  
By Proposition \ref{prop:qcat-radj-limits}, the $W$-weighted limit of $yD \colon \eA \to \Radj(\qCat)^{\eC\op}$ exists and is created by the inclusion into $(\qCat^\cattwo)^{\eC\op}$, where it is preserved by the evaluation functors to $\qCat$. Thus, we conclude that $\Map_{\eC}(-,u) \in (\qCat^\cattwo)^{\eC\op}$ is naturally equivalent to a diagram whose components at each $X \in \eC$, displayed as the right-hand verticals of \eqref{eq:pointwise-nat-equiv}, admit right adjoints and for which the transformations induced by each $e \colon Y \to X$ are exact. The repleteness statement of Proposition \ref{prop:qcat-radj-limits} tells us that an object of $(\qCat^\cattwo)^{\eC\op}$ that is naturally equivalent to an object of $\Radj(\qCat)^{\eC\op}$ is also in $\Radj(\qCat)^{\eC\op}$. Hence,   each component $\Map_{\eC}(X,u) \colon \Map_{\eC}(X,L) \to \Map_{\eC}(X,K)$ admits a right adjoint and the transformations induced by each $e \colon Y \to X$ are exact. Now the pullback of Proposition \ref{prop:representable-adjunctions} tells us that the flexible weighted homotopy limit $u \colon L \to K$ is present in $\Radj(\eC)$. A similar analysis shows that the legs of the limit cone are exact.
\end{proof}

When $\eK$ is an $\infty$-cosmos the previous results restrict to the quasi-categorically enriched categories of arrows or right adjoints in $\eK$ that are isofibrations, for which we re-appropriate the previous notation. If $\eK$ is an $\infty$-cosmos, let $\eK^\cattwo$ denote the $\infty$-\emph{cosmos of isofibrations} of Proposition \ref{prop:isofib-cosmoi}.

\begin{prop}\label{prop:radj-cosmos} When $\eK$ is an $\infty$-cosmos, $\Radj(\eK)$ is an $\infty$-cosmos and $\Radj(\eK)\inc\eK^\cattwo$ is a cosmological functor, creating isofibrations and simplicially enriched limits.
\end{prop}
\begin{proof}
By Proposition \ref{prop:radj-limits}, isofibrations and strict simplicially enriched limits in $\Radj(\eK)$ are inherited from $\eK^\cattwo$. This proves that the inclusion is a cosmological functor. 
\end{proof}

\begin{ntn} For any  $\infty$-category $A$ in an $\infty$-cosmos $\eK$, let $\coCart(\eK)_{/A} \inc \eK_{/A}$ denote the quasi-categorically enriched subcategory with  
  \begin{itemize}
    \item {\bf objects} cocartesian fibrations in $\eK$ with codomain $A$, and 
    \item {\bf $n$-arrows} those $n$-arrows of $\eK_{/A}$ whose vertices are cartesian functors.
\end{itemize}

Let $\coCart\gr(\eK)_{/A} \inc \coCart(\eK)_{/A}$ denote the full subcategory spanned by the groupoidal cocartesian fibrations in $\eK$ with codomain $A$; since any functor between groupoidal cocartesian fibrations is cartesian, the inclusion $\coCart\gr(\eK)_{/A} \inc \eK_{/A}$ is also full.  Since groupoidal cocartesian fibrations define groupoidal objects in the $\infty$-cosmos $\eK_{/A}$, this full subcategory is in fact enriched over Kan complexes. The subcategories $\Cart\gr(\eK)_{/A} \inc\Cart(\eK)_{/A}\inc\eK_{/A}$ are defined similarly.

Finally, for any $A,B \in \eK$ let ${}_A\qMod(\eK)_{B} \subset \eK_{/A \times B}$ denote the full subcategory of modules from $A$ to $B$. As modules are groupoidal objects, this subcategory is also enriched in Kan complexes.
\end{ntn}

\begin{prop}\label{prop:cocartesian-completeness} For any $\infty$-cosmos $\eK$, the subcategories $\coCart(\eK)_{/B}$ and $\Cart(\eK)_{/B}$ admit and the inclusions \[\coCart(\eK)_{/B} \inc \eK_{/B} \qquad \mathrm{and} \qquad  \Cart(\eK)_{/B} \inc \eK_{/B}\] are closed under flexible weighted homotopy limits.
\end{prop}
\begin{proof}
We prove the statements for cocartesian fibrations, the results for cartesian fibrations being dual.   We first argue that there is a pullback of quasi-categorically enriched categories:
\begin{equation}\label{eq:cocart-pullback} \xymatrix{ \coCart(\eK)_{/B} \ar[r] \ar@{^(->}[d] \pbexcursion & \Radj(\eK_{/B}) \ar@{^(->}[d] \\ \eK_{/B} \ar[r]_-K & (\eK_{/B})^\cattwo}\end{equation}
where $K$ is the functor that carries an isofibration $p \colon E \tfib B$ to the isofibration $E^\iso \tfib p \comma B$ over $B$ defined by restricting the Leibniz cotensor $\langle p_0, p \rangle \colon E^\cattwo \tfib p \comma B$ along $E^\iso \tfib E^\cattwo$; recall our convention that the objects in each of these four categories are isofibrations.

We must argue that the claimed pullback relationship holds on objects and 0-arrows. On objects, this  follows from Definition \ref{defn:cocart-fibration}\ref{itm:cocart.fib.chars.ii}, which tells us that $p \colon E \to B$ is cocartesian if and only if the functor $E \to p\comma B$ admits a left adjoint over $B$. This map factors as an equivalence $E \we E^\iso$ followed by the map $E^\iso \tfib p\comma B$, so the composite admits a left adjoint in $\eK_{/B}$ if and only if the map $E^\iso \tfib p \comma B$ does. On 0-arrows, this follows from  Definition \ref{defn:cartesian-functor}, which tells us that a functor 
\[ \xymatrix{ E \ar@{->>}[dr]_p \ar[rr]^f & & F \ar@{->>}[dl]^q \\ & B}\] between cocartesian fibrations is cartesian if and only if the induced outer square
\[    \xymatrix@=1.5em{ {E}\ar[d]_{\rotatebox{90}{$\sim$}}\ar[r]^{f} & {F}\ar[d]^{\rotatebox{90}{$\sim$}} \\ E^\iso \ar[r]^{f^\iso} \ar@{->>}[d] & F^\iso \ar@{->>}[d] \\
   {p\comma B}\ar[r]_{\scriptstyle (f,\id)} & {q\comma B}
    }\]
is exact, and this is the case if and only if the lower square is exact. This tells us that a 0-arrow between cocartesian fibrations in $\eK_{/B}$ defines a cartesian functor if and only if its image under the functor $K$ defines an exact square. Since the pairs of subcategories are full on $n$-arrows for $n >0$, the claimed pullback relationship follows from this pair of results.

Now Proposition \ref{prop:radj-limits} tells us that $\Radj(\eK_{/B})\inc(\eK_{/B})^\cattwo$ creates flexible weighted homotopy limits. The functor  $K\colon \eK_{/B} \to (\eK_{/B})^\cattwo$ preserves them since its construction involves a weighted limit, that commutes with these flexible weighted limits. It follows as in the proof of that result that $\coCart(\eK)_{/B}$ admits flexible weighted limits and both legs of the pullback cone preserve them. In particular the inclusion  $\coCart(\eK)_{/B} \inc \eK_{/B}$ creates flexible weighted limits as claimed, proving the claim.
\end{proof}

Proposition \ref{prop:cocartesian-completeness} has an interesting and important corollary.

\begin{prop}\label{prop:cartesian-cosmoi} For any $\infty$-category $B$ in an $\infty$-cosmos $\eK$, the quasi-categorically enriched subcategories $\coCart(\eK)_{/B}, \Cart(\eK)_{/B} \subset \eK_{/B}$ spanned by the (co)cartesian fibrations and cartesian functors define $\infty$-cosmoi, with isofibrations and limits inherited from $\eK_{/B}$. The Kan-complex-enriched subcategories of groupoidal objects are, respectively, the subcategories
\[ \coCart(\eK)\gr_{/B}\inc\coCart(\eK)_{/B} \qquad \mathrm{and} \qquad \Cart(\eK)\gr_{/B}\inc\Cart(\eK)_{/B}\] of groupoidal (co)cartesian fibrations.
\end{prop}
\begin{proof}
Proposition \ref{prop:cocartesian-completeness} proves axiom \ref{defn:cosmos}\ref{defn:cosmos:a}. The closure properties of the isofibrations enumerated in axiom \ref{defn:cosmos:b} are inherited from the sliced $\infty$-cosmos. 
\end{proof}

\begin{rmk} The Kan-complex-enriched category ${}_A\qMod(\eK)_B$ of modules from $A$ to $B$ in an $\infty$-cosmos $\eK$ is also the category of groupoidal objects in an $\infty$-cosmos, namely the $\infty$-cosmos of \emph{two-sided fibrations}
\[
\Cart(\coCart(\eK)_{/A})_{/\pi \colon A \times B \tfib A} \cong \coCart(\Cart(\eK)_{/B})_{/\pi \colon A \times B \tfib B}\]
 that will be introduced in a forthcoming paper \cite{RiehlVerity:2017ts}.
\end{rmk}


\section{Comprehension and the Yoneda embedding}\label{sec:comprehension-yoneda}

The (external) Yoneda embedding carries an element $a \colon 1 \to A$ of an $\infty$-category $A$ to the module $p_0 \colon A \comma a \tfib A$ from $1$ to $A$, a groupoidal cartesian fibration over $A$. This is the assignment on objects of a functor from $\Fun(1,A)$, the \emph{underlying quasi-category} of the $\infty$-category $A$, to the quasi-category of modules from $1$ to $A$. Our aim in this section is to review the construction of this functor.

In \S\ref{sec:comprehension-construction}, we describe the general \emph{comprehension construction}, which is the subject of \cite{RiehlVerity:2017cc}. The comprehension construction associates to any $\infty$-category $A$ and any (co)car\-tes\-ian fibration $p \colon E \tfib B$ a functor from the quasi-category $\Fun(A,B)$ to the quasi-category of (co)cartesian fibrations and cartesian functors over $A$. On objects, the comprehension construction carries a functor $a \colon A \to B$ to the pullback of the fibration $p$:
\[\xymatrix{ E_a \ar@{->>}[d]_{p_a} \ar[r]^-{\ell_a} \pbexcursion & E \ar@{->>}[d]^p \\ A \ar[r]_a & B}\]

Importantly, the construction of the comprehension functor provided by Theorem \ref{thm:comprehension} may be used for $\infty$-categories in any $\infty$-cosmos. In \S\ref{sec:yoneda-case}, we exploit this versatility to define the co- and contravariant Yoneda embeddings as specialisation of the comprehension functor to an appropriate sliced $\infty$-cosmos. These functors  will be used to prove our completeness and cocompleteness results in \S\ref{sec:construction}.


\subsection{The comprehension functor}\label{sec:comprehension-construction}

In this section, we review the construction of the comprehension functor from \cite{RiehlVerity:2017cc} that will be specialised in \S\ref{sec:yoneda-case} to define the Yoneda embedding. 

\begin{ntn}\label{ntn:qcat-ntn} 
Any $\infty$-cosmos $\eK$ admits a maximal $(\infty,1)$-categorical core $g_*\eK$, the subcategory with the same objects and with functor spaces 
\[ \Fun_{g_*\eK}(A,B) \defeq g\Fun(A,B)\] defined to be the maximal groupoid cores of the functor quasi-categories; see Definition \refVII{defn:infinity,1-core}. By \cite[2.1]{Cordier:1986:HtyCoh} the homotopy coherent nerve of a Kan-complex-enriched category is a quasi-category, so we let:
  \begin{itemize}
  \item $\qK_{/A}$ denote the quasi-category $\hN(g_*(\eK_{/A}))$,
  \item $\coCart(\qK)_{/A}$ denote the quasi-category $\hN(g_*(\coCart(\eK)_{/A}))$,
  \item $\Cart(\qK)_{/A}$ denote the quasi-category $\hN(g_*(\Cart(\eK)_{/A}))$,
  \item $\coCart\gr(\qK)_{/A}$ denote the quasi-category $\hN(\coCart\gr(\eK)_{/A})$,
  \item $\Cart\gr(\qK)_{/A}$ denote the quasi-category $\hN(\Cart\gr(\eK)_{/A})$,  and
  \item ${}_A\qMod(\qK)_{B}$ denote the quasi-category $\hN({}_A\qMod(\eK)_B)$.
  \end{itemize}
\end{ntn}

With this notation in hand, we a may now introduce the \emph{comprehension functor}.

\begin{thm}[{\refVI{thm:general-comprehension}}]\label{thm:comprehension}
For any cocartesian fibration $p \colon E \tfib B$ in an $\infty$-cosmos $\eK$ and any $\infty$-category $A \in \eK$, there is a functor
\[ \Fun_{\eK}(A,B) \xrightarrow{c_{p,A}} \coCart(\qK)_{/A}\]
defined on 0-arrows by mapping a functor $a \colon A \to B$ to the pullback:
\[  \xymatrix{ E_a \ar@{->>}[d]_{p_a} \ar[r]^-{\ell_a} \pbexcursion & E \ar@{->>}[d]^p \\ A \ar[r]_a & B}\]
  Its action on 1-arrows $f\colon a\to b$ is defined by lifting $f$ to a $p$-cocartesian 1-arrow as displayed in the diagram
  \begin{equation*}
    \begin{xy}
      0;<1.4cm,0cm>:<0cm,0.75cm>::
      *{\xybox{
          \POS(1,0)*+{A}="one"
          \POS(0,1)*+{A}="two"
          \POS(3,0.5)*+{B}="three"
          \ar@{=} "one";"two"
          \ar@/_5pt/ "one";"three"_{b}^(0.1){}="otm"
          \ar@/^10pt/ "two";"three"^{a}_(0.5){}="ttm"|(0.325){\hole}
          \ar@{=>} "ttm"-<0pt,7pt> ; "otm"+<0pt,10pt> ^(0.3){f}
          \POS(1,2.5)*+{E_{b}}="one'"
          \POS(1,2.5)*{\pbcorner}
          \POS(0,3.5)*+{E_{a}}="two'"
          \POS(0,3.6)*{\pbcorner}
          \POS(3,3)*+{E}="three'"
          \ar@/_5pt/ "one'";"three'"_{\ell_{\hat{b}}}^(0.1){}="otm'"
          \ar@/^10pt/ "two'";"three'"^{\ell_{\hat{a}}}_(0.55){}="ttm'"
          \ar@{->>} "one'";"one"_(.3){p_b}
          \ar@{->>} "two'";"two"_{p_a}
          \ar@{->>} "three'";"three"^{p}
          \ar@{..>} "two'";"one'"_*!/^2pt/{\scriptstyle E_f}
          \ar@{=>} "ttm'"-<0pt,4pt> ; "otm'"+<0pt,4pt> ^(0.3){\ell_{\hat{f}}}
        }}
    \end{xy}
  \end{equation*}
  and then factoring its codomain to obtain the requisite $\infty$-functor $E_f \colon E_a \to E_b$ between the fibres over $a$ and $b$. 
  \end{thm}

A key advantage of Theorem \ref{thm:comprehension} is  that it may be interpreted in any $\infty$-cosmos and in particular applies to slices and duals. 

\begin{rmk}[dual case]
  A cartesian fibration $p\colon E \tfib B$ in $\eK$ defines a cocartesian fibration in $\eK\co$. The comprehension construction in $\eK\co$ defines a functor
  \[ \Fun_{\eK\co}(A,B) \to \coCart(\qK\co)_{/A} \subset
    \qK\co_{/A}
  \]
  Because the duality isomorphism $\qK\co_{/A} \cong (\qK_{/A})\co$ interchanges cocartesian and cartesian fibrations, it follows that $\coCart(\qK\co)_{/A}$ is isomorphic to $\Cart(\qK)_{/A}\co$. Consequently the comprehension construction can be rewritten as
  \begin{equation*}
    \Fun_{\eK}(A,B)\op \to \Cart(\qK)_{/A}\co
    \subset \qK_{/A}\co
  \end{equation*}
\end{rmk}
  
\begin{rmk}[groupoidal case]
If $p \colon E \tfib B$ is a groupoidal (co)cartesian fibration, then Proposition \ref{prop:cart-fib-pullback} demonstrates that its pullbacks are again groupoidal (co)cartesian fibrations. So in this instance, the comprehension functors land in the full sub quasi-categories $\coCart\gr(\qK)_{/A}$ or $\Cart\gr(\qK)_{/A}\co$ spanned by the groupoidal objects.
\end{rmk}

\subsection{The Yoneda embedding}\label{sec:yoneda-case}

\newcommand{\colarr}[3]{%
  \xybox{%
    0;<5.5ex,0cm>:
    \POS (0,0.5)*+{#1}="a",
         (0,-0.5)*+{#3}="b"
    \POS "a" \ar@{->>} "b" ^(0.4){#2}}}
\newcommand{\colarrow}[3]{\xy(0,0)*{\colarr{#1}{#2}{#3}}\endxy}

 In this section, we specialise the comprehension construction to define the covariant and contravariant Yoneda embeddings. This makes use of the hom module $(p_1,p_0) \colon A^\cattwo \tfib A \times A$ of Example \ref{ex:hom-is-a-module} associated to the $\infty$-category $A$. As recalled in Definition \ref{defn:module}, the domain-projection functor defines a cartesian fibration and the codomain-projection functor defines a cocartesian fibration in the sliced $\infty$-cosmos $\eK_{/A}$. Applying a special case of the comprehension construction in each of these instances defines the co- and contravariant Yoneda embeddings as full and faithful functors internal to the $\infty$-cosmos of (large) quasi-categories: the underlying quasi-category $\Fun_{\eK}(1,A)$ of $A$ is embedded covariantly into $\Cart(\qK)\gr_{/A} \cong {}_1\qMod(\qK)_A$ and contravariantly into $(\coCart(\qK)\gr_{/A})\co\cong{}_A\qMod(\qK)_1\co$. A generalisation of this construction, where the terminal $\infty$-category $1$ is replaced by a generic $\infty$-category, will be used in the \S\ref{sec:construction} to prove our general limit and colimit construction theorems.

 \begin{defn}[covariant Yoneda embedding]\label{defn:covariant-yoneda-embedding}
   For any object $A$ in an $\infty$-cosmos $\eK$, the cotensor $(p_1,p_0) \colon A^\cattwo \tfib A \times A$ defines a (groupoidal) cocartesian fibration
   \begin{equation}\label{eq:pre.yoneda.2}
     \xymatrix{ A^\cattwo \ar@{->>}[rr]^-{(p_1,p_0)} \ar@{->>}[dr]_{p_0}
       & & A \times A \ar@{->>}[dl]^{\pi_0} \\ & A}
   \end{equation}
   in the slice $\infty$-cosmos $\eK_{/A}$; see Lemma~\refV{lem:disc-cart-on-right}. The comprehension construction defines a functor:
   \begin{equation}\label{eq:pre.yoneda}
     \Fun_{\eK_{/A}}\left(
       \colarrow{A}{\id_ A}{A},
       \colarrow{A\times A}{\pi_0}{A}
     \right)
     \longrightarrow \coCart(\qK_{/A})_{/\id_A} \cong \qK_{/A},
   \end{equation}
   the isomorphism because cocartesian fibrations over the terminal object are just objects in the $\infty$-cosmos, in this case $\eK_{/A}$.
   Now the domain of this comprehension functor receives a map
   \begin{equation*}
     \Fun_{\eK}(1,A) \to \Fun_{\eK_{/A}}(\id_A,\pi_0),
   \end{equation*}
   defined on objects by sending $a \colon 1 \to A$ to
   \begin{equation*}
     A \cong 1 \times A \xrightarrow{a \times \id_A} A \times A.
   \end{equation*}
   Composing with~\eqref{eq:pre.yoneda} defines a functor
   $\yoneda\colon\Fun_{\eK}(1,A)\to \qK_{/A}$. This acts on a vertex $a\colon 1\to A$
   to return the pullback of $A^\cattwo\tfib A\times A$ along $a\times
   \id_A\colon 1\times A\to A\times A$, that being the module
   $p_0\colon A\comma a\tfib A$ from $1$ to $A$ of Definition \ref{defn:rep-modules}. Consequently, the
   codomain of the functor $\yoneda$ restricts to define a functor
   \begin{equation*}
     \yoneda\colon \Fun_{\eK}(1,A) \to {}_1\qMod(\qK)_A \subset \qK_{/A}
   \end{equation*}
   which is the \emph{covariant Yoneda embedding}.
 \end{defn}
 
 The contravariant Yoneda embedding is an instance of the covariant Yoneda embedding in an appropriate dual.

 \begin{defn}[contravariant Yoneda embedding]\label{defn:contravariant-yoneda-embedding} 
 By applying the covariant Yoneda construction described above in the dual $\infty$-cosmos $\eK\co$ we obtain a dual functor, which this time maps each $a\colon1\to A$ to the groupoidal cocartesian fibration $p_1\colon a\comma A\tfib A$. This gives rise to an embedding
 \begin{equation*}
   \yoneda\colon \Fun_{\eK\co}(1,A) \to {}_1\qMod(\qK\co)_{A} \subset
   \qK\co_{/A}.
 \end{equation*}
  
  Because the duality isomorphism $\eK\co_{/A} \cong (\eK_{/A})\co$ interchanges cocartesian and cartesian fibrations, it carries modules from $1$ to $A$ to modules from $A$ to $1$. Hence, it follows that $ {}_1\qMod(\qK\co)_{A}$ is isomorphic to $ {}_A\qMod(\qK)_{1}\co$. Consequently the embedding above can be rewritten as
  \begin{equation*}
    \yoneda\colon \Fun_{\eK}(1,A)\op \to {}_A\qMod(\qK)_{1}\co
    \subset \qK_{/A}\co
  \end{equation*}
  and this is known as the \emph{contravariant Yoneda embedding}.
\end{defn}

\begin{defn}[generalised Yoneda embeddings]\label{defn:generalised-yoneda}
  The covariant Yoneda embedding can be generalised to replace the terminal
  $\infty$-category $1$ in Definition~\ref{defn:covariant-yoneda-embedding} by a
  generic $\infty$-category $D \in \eK$. Indeed, we can make this generalisation
  simply by applying the covariant Yoneda construction to the promoted object
  $\pi_1\colon D\times A\tfib D$ in the slice $\infty$-cosmos $\eK_{/D}$. On
  observing that the iterated slice $(\eK_{/D})_{/\pi_1\colon D\times A\tfib D}$
  is isomorphic to the slice $\eK_{/D\times A}$, we find that this gives us the
  \emph{covariant generalised Yoneda embedding\/} which is of the following form:
  \begin{equation*}
    \yoneda\colon\Fun_{\eK}(D,A)\cong \Fun_{\eK_{/D}}
    \left(
      \colarrow{D}{\id_D}{D}, \colarrow{D\times A}{\pi_1}{D}
    \right)
    \xrightarrow{\mkern40mu}
    (\qK_{/D})_{/\pi_1\colon D\times A \tfib D} \cong \qK_{/D\times A}
  \end{equation*}
  By the explicit description of the action of Yoneda on vertices given in Definition~\ref{defn:covariant-yoneda-embedding}, this acts on a $0$-arrow $f\colon D\to A$ to carry it to the representable fibration on $(\id,f)\colon D\to D\times A$ in $\eK_{/D}$. A simple computation reveals that this is simply the object $A\comma f\tfib D\times A$, so it follows that our generalised embedding restricts to the full simplicial subcategory ${}_D\qMod(\qK)_A\subset\eK_{D\times A}$ of modules in its codomain.

  The corresponding \emph{contravariant generalised Yoneda embedding\/} is obtained analogously by applying the construction of Definition~\ref{defn:contravariant-yoneda-embedding} to the object $\pi_0\colon A\times D\tfib D$ in the sliced $\infty$-cosmos $\eK_{/D}$. This gives a functor of quasi-categories of the following form:
  \begin{equation*}
    \yoneda\colon \Fun_{\eK}(D,A)\op \xrightarrow{\mkern40mu} {}_A\qMod(\qK)_{D}\co
    \subset \qK_{/A\times D}\co
  \end{equation*}
\end{defn}

\begin{rmk}[generalised Yoneda in explicit terms]\label{rmk:explicit-generalised-yoneda}
  The generalised Yoneda embedding $\yoneda\colon\Fun_{\eK}(D,A)\to\prescript{}{D}\qMod(\qK)_{A}$ carries each $a\colon D\to A$ to the module $(p_1,p_0)\colon A\comma a\tfib D\times A$. Furthermore, Theorem~\ref{thm:comprehension} tells us that it acts on a $1$-arrow $\alpha\colon a\to b$ in $\Fun_{\eK}(D,A)$ by forming the following cartesian lift
  \[ 
    \begin{xy}
      0;<1.2cm,0cm>:
      *{\xybox{
          \POS(1,0)*+{D \times A}="one"
          \POS(0,1)*+{D \times A}="two"
          \POS(4,0.5)*+{A \times A}="three"
          \ar@{=} "one";"two"
          \ar@/_5pt/ "one";"three"_{b\times A}^(0.3){}="otm"
          \ar@/^10pt/ "two";"three"^{a\times A}_(0.4){}="ttm"|(0.24){\hole}
          \ar@{=>} "ttm"-<0pt,10pt> ; "otm"+<0pt,10pt> ^{\alpha \times \id}
          \POS(1,2.5)*+{A \comma b}="one'"
          \POS(1,2.5)*{\pbcorner}
          \POS(0,3.5)*+{A \comma a}="two'"
          \POS(0,3.6)*{\pbcorner}
          \POS(4,3)*+{A^\cattwo}="three'"
          \ar@/_5pt/ "one'";"three'"^(0.3){}="otm'"
          \ar@/^10pt/ "two'";"three'"^{}_(0.4){}="ttm'"
          \ar@{->>} "one'";"one"_(0.3){p_0}
          \ar@{->>} "two'";"two"_{p_0}
          \ar@{->>} "three'";"three"^{(p_1,p_0)}
          \ar@{..>} "two'";"one'"^{\labelstyle \yoneda(\alpha)}
          \ar@{=>} "ttm'"-<0pt,10pt> ; "otm'"+<0pt,10pt> ^{\chi_{\alpha \times \id_A}}
        }}
    \end{xy}
  \]
  along $(p_1,p_0)\colon A^{\cattwo}\tfib A\times A$ regarded as a cartesian fibration in the slice $\eK_{/A}$. The resulting functor $\yoneda(\alpha)\colon A\comma a\to A\comma b$ could also be annotated as $A\comma \alpha$, since it may also be described as having been induced by the 2-cell $\alpha\colon a\Rightarrow b$ in the manner described in Observation~\ref{obs:trans.induce.comma}.
\end{rmk}

\begin{prop}[generalised Yoneda and cosmological functors]\label{prop:gen-Yoneda-pres}
For any cosmological functor $G\colon\eK\to\eL$ is a cosmological functor and $\infty$-categories $A$ and $D$ in $\eK$, the generalised Yoneda embeddings fit into an essentially commutative square:
  \begin{equation*}
    \xymatrix@=2em{
      {\Fun_{\eK}(D,A)}\ar[r]^-{\yoneda}\ar[d]_{G}
      \ar@{}[dr]|{\cong} & {{}_D\qMod(\qK)_A}\ar[d]^{G} \\
      {\Fun_{\eL}(G(D),G(A))}\ar[r]_-{\yoneda} & {{}_{G(D)}\qMod(\qL)_{G(A)}}
    }
  \end{equation*}
\end{prop}
\begin{proof}
Any cosmological functor preserves comma objects, adjunctions, and groupoidal objects, so it follows that it preserves (co)cartesian fibrations, (co)cartesian arrows, cartesian functors, and modules. In other words, it preserves all of the structures used in the comprehension construction. Furthermore, if $A$ is an object in $\eK$ then $G$ carries the cocartesian fibration~\eqref{eq:pre.yoneda.2} used to define the Yoneda embedding $\yoneda\colon\Fun_{\eK}(1,A)\to\qK_{/A}$, as in Definition~\ref{defn:covariant-yoneda-embedding}, to the cocartesian fibration used to define the Yoneda embedding $\yoneda\colon\Fun_{\eL}(1,G(A))\to\qL_{/G(A)}$. These facts, combined with the essential uniqueness property of the comprehension construction, Observation~\refVI{obs:unique-comprehension}, lead us to the conclusion that there exists an essentially commutative square
\[
    \xymatrix@=2em{
      {\Fun_{\eK}(1,A)}\ar[r]^-{\yoneda}\ar[d]_{G}
      \ar@{}[dr]|{\cong} & {\qK_{/A}}\ar[d]^{G} \\
      {\Fun_{\eL}(1,G(A))}\ar[r]_-{\yoneda} & {\qL_{/G(A)}}
    }
\]
  relating the Yoneda embeddings associated with $A$ in $\eK$ and $G(A)$ in $\eL$ (see also Proposition~\refVI{prop:comprehension-cou}). By applying this result to the induced cosmological functor of slices $G\colon\eK_{/D}\to\eL_{/G(D)}$ this result extends to generalised Yoneda embeddings, giving an essentially commutative square:
  \begin{equation*}
    \xymatrix@=2em{
      {\Fun_{\eK}(D,A)}\ar[r]^-{\yoneda}\ar[d]_{G}
      \ar@{}[dr]|{\cong} & {\qK_{/D\times A}}\ar[d]^{G} \\
      {\Fun_{\eL}(G(D),G(A))}\ar[r]_-{\yoneda} & {\qL_{/G(D)\times G(A)}}
    }
  \end{equation*}
 Since the cosmological functor $G$ carries modules in $\eK$ to modules in $\eL$. It follows that we may restrict the square above to give the essentially commutative square of the statement.
\end{proof}

One application of Proposition \ref{prop:gen-Yoneda-pres} is particularly worth of note:

\begin{lem}\label{lem:gen-Yoneda-precomp}
  Suppose that $A$ is an object and $f\colon C\to D$ is a functor in $\eK$. Then there exists an essentially commutative square of generalised Yoneda embeddings:
  \begin{equation*}
    \xymatrix@R=2em@C=3em{
      {\Fun_{\eK}(D,A)}\ar[r]^-{\yoneda}\ar[d]_{\Fun_{\eK}(f,A)}
      \ar@{}[dr]|{\cong} & {{}_D\qMod(\qK)_A}\ar[d]^{(f\times A)^*} \\
      {\Fun_{\eK}(C,A)}\ar[r]_-{\yoneda} & {{}_C\qMod(\qK)_A}
    }
  \end{equation*}
\end{lem}

\begin{proof}
  The generalised Yoneda embeddings in the statement may be constructed by promoting $A$ to an object $D\times A\tfib D$ (resp. $C\times A\tfib C$) in the sliced $\infty$-cosmos $\eK_{/D}$ (resp. $\eK_{/C}$) and applying the Yoneda embedding construction of Definition~\ref{defn:covariant-yoneda-embedding} there. Pullback along $f\colon C\to D$ defines a cosmological functor $f^*\colon \eK_{/D}\to \eK_{/C}$ which carries $D\times A\tfib D$ to $C\times A\tfib C$, and it is easily checked that Proposition \ref{prop:gen-Yoneda-pres} specialises in the case of the cosmological functor $f^*$ and the Yoneda embedding derived from $D\times A\tfib D$  to the square given in the statement.
\end{proof}

\section{The formal theory of \texorpdfstring{$\infty$}{infinity}-categories}\label{sec:formal}

The motivation for $\infty$-cosmology is that it enables us to develop the theory of $\infty$-categories ``formally''; in particular, independently of the semantics of any particular model. In this section, which is part review and part new material, we introduce those aspects of the formal theory of $\infty$-categories that we will need later in this paper.

In \S\ref{ssec:limits}, we review the definitions of limits and colimits of diagrams valued in an $\infty$-category, introducing the main subject of this paper. There is one new result which appears in this section: Proposition \ref{prop:ff.and.sg.limit.pres}, which demonstrates that fully faithful and strongly generating functors preserve all limits that exist in their domain and codomain. In \S\ref{ssec:ff-and-sg} we define these notions and prove the theorems which allow us to find examples of $\infty$-functors with these properties.



\subsection{Fully faithful and strongly generating functors}\label{ssec:ff-and-sg}

In this section, we say what it means for a functor between $\infty$-categories to be fully faithful and strongly generating and then provide alternate characterisations of these notions in the quasi-categorical model, which will allow us to develop examples.

\begin{defn}[fully-faithful]\label{defn:fully.faithful}
  We say that a functor of $\infty$-categories $f\colon A\to B$ is \emph{fully-faithful\/} if the functor $\bar{f}= \comma(1,f,1)\colon A^\cattwo\to f\comma f$ induced, as in Proposition~\ref{prop:trans-comma}, by the commutative diagram
  \begin{equation*}
    \xymatrix@R=1.5em@C=2em{
      {A}\ar@{=}[r]\ar@{=}[d] & {A}\ar[d]^{f} & {A}\ar@{=}[l]\ar@{=}[d] \\
      {A}\ar[r]_{f} & {B} & {A}\ar[l]^{f}
    }
  \end{equation*}
  is an equivalence.
\end{defn}

\begin{defn}[strong generator]\label{defn:strong.gen} We say that a functor of $\infty$-categories $f\colon A\to B$ is \emph{strongly generating\/} if it satisfies the property that a 2-cell
  \begin{equation*} \xymatrix@R=1em@C=6em{ {X}\ar@/^2ex/[r]^{h}\ar@/_2ex/[r]_{k} \ar@{}[r]|{\Downarrow\beta} & {B} }
  \end{equation*} is invertible whenever the functor $f\comma\beta \colon f\comma h\to f\comma k$, as described in Observation~\ref{obs:trans.induce.comma}, is an equivalence.
\end{defn}

Our next aim is to  provide concrete characterisations of those functors $f\colon \qA\to\qB$ of quasi-categories that satisfy the abstract properties described in Definitions~\ref{defn:fully.faithful} and~\ref{defn:strong.gen}. These both rely upon the well-known fibre-wise characterisation of equivalences between (co)cartesian fibrations developed in the next proposition, which we now recall:

\begin{prop}\label{prop:equivalence-of-fibrations} A cartesian functor
  \[
    \xymatrix{ \qE\ar[rr]^g \ar@{->>}[dr]_p & & \qF \ar@{->>}[dl]^q \\ & \qB}
  \]
  between cocartesian fibrations of quasi-categories is an equivalence in $\qCat_{\!/\qB}$ if and only if it is a fibrewise equivalence, meaning that for each object $b \in \qB$ the functor $g_b \colon \qE_b \to \qF_b$ induced between corresponding fibres is an equivalence.
\end{prop}

Proofs can be found in \cite[3.3.1.5]{Lurie:2009fk}, \cite[2.9]{AyalaFrancis:2017fo}, or \cite[4.3.2]{RiehlVerity:2018cl-v2}. Our interest in this result arises from the following corollary:

\begin{cor}[equivalences of modules are determined fibre-wise]\label{cor:mod.equiv.fibrewise}
  Suppose that we are given two modules $\qE,\qF$ from $\qA$ to $\qB$ between quasi-categories and a functor
  \begin{equation*}
    \xymatrix@C=1em@R=2em{
      {\qE}\ar[rr]^{g}\ar@{->>}[dr]_{(q,p)} && {\qF}\ar@{->>}[dl]^{(s,r)} \\
      & {\qA\times\qB} &
    }
  \end{equation*}
  between them in the slice $\qCat_{\!/\qA\times \qB}$. Then $g$ is an equivalence in $\qCat_{\!/\qA\times \qB}$ if and only if it is a fibre-wise equivalence, in the sense that for each pair of objects $a\in\qA$ and $b\in\qB$ the induced map of bi-fibres $g_{b,a}\colon \qE(b,a)\to \qF(b,a)$ is an equivalence of Kan complexes.
\end{cor}

\begin{proof}
  We know, by Lemma~\ref{lem:module.legs}, that the left-hand legs $q\colon \qE\tfib\qA$ and $s\colon \qF\tfib\qA$ of the given modules are cocartesian fibrations and that our functor $g\colon \qE\to \qF$ is a cartesian functor between them. It follows, by Proposition~\ref{prop:equivalence-of-fibrations}, that $g$ is an equivalence if and only if its action on the fibres $g_{\id_{\qB},a}\colon \qE(\id_{\qB},a)\to\qF(\id_{\qB},a)$ over each object $a\in \qA$ is an equivalence. Now $\qE(\id_{\qB},a)$ and $\qF(\id_{\qB},a)$ are modules from $1$ to $\qB$ so their right-hand legs are cartesian fibrations and $g_{\id_{qB},a}$ is a cartesian functor between them. Consequently, we may apply the manifest dual of Proposition~\ref{prop:equivalence-of-fibrations} to show that $g_{\id_{\qB},a}$ is an equivalence if and only if its action $g_{b,a}\colon\qE(b,a)\to\qF(b,a)$ on the fibres over each object $b\in\qB$ is an equivalence. The stated result follows immediately. 
\end{proof}

Now we may proceed to characterising the fully-faithful and strongly generating functors between quasi-categories in terms of the hom-space modules of Example \ref{ex:hom.spaces}, specialising Definitions~\ref{defn:fully.faithful} and~\ref{defn:strong.gen}:

\begin{prop}[fully-faithful functors of quasi-categories]\label{prop:qcat.fully.faithful}
  A functor $f\colon\qA\to\qB$ in the $\infty$-cosmos $\qCat$ of quasi-categories is fully-faithful  if and only if for all objects $a,b\in \qA$ its action $f_{a,b}\colon\Hom_{\qA}(a,b)\to\Hom_{\qB}(fa,fb)$ on hom-spaces is an equivalence of Kan complexes.
\end{prop}

\begin{proof}
  The induced functor $\bar{f}\colon \qA^{\cattwo}\to f\comma f$ introduced in Definition~\ref{defn:fully.faithful} is a map in the slice $\qCat_{\!/\qA\times \qA}$ between the modules $A^{\cattwo}$ and $f\comma f$. Furthermore, its action on the bi-fibre over objects $a,b\in\qA$ is the action $f_{a,b}\colon\Hom_{\qA}(a,b)\to\Hom_{\qB}(fa,fb)$ of $f$ on the hom-space between $a$ and $b$ defined in Example~\ref{ex:hom.spaces}. Consequently, the stated result follows immediately from Corollary~\ref{cor:mod.equiv.fibrewise}.
\end{proof}

\begin{prop}[generalised Yoneda is fully-faithful] \label{prop:yoneda.fully.faithful}
For any $\infty$-categories $D$ and $A$ the generalised Yoneda embeddings
  \begin{equation*}
    \yoneda\colon \Fun_{\eK}(D,A)\xrightarrow{\mkern40mu} {}_D\qMod(\qK)_{A}\qquad\text{and}\qquad    \yoneda\colon \Fun_{\eK}(D,A)\op \xrightarrow{\mkern40mu} {}_A\qMod(\qK)_{D}\co
  \end{equation*}
are fully faithful as functors of quasi-categories.
\end{prop}
\begin{proof}
  Theorem \refVI{thm:yoneda-ff} proves that the covariant and contravariant Yoneda embeddings are fully faithful, in the sense that their actions on hom-spaces are equivalences of Kan complexes.   Proposition~\ref{prop:qcat.fully.faithful} then tells us that they are indeed fully-faithful as functors in the $\infty$-cosmos $\qCat$ in the sense of Definition~\ref{defn:fully.faithful}. Since Definition \ref{defn:generalised-yoneda} constructs the generalised Yoneda embedding functors as a special case of the ordinary Yoneda embeddings, this establishes the general result.
\end{proof}

\begin{prop}[strongly generating functors of quasi-categories]\label{prop:qcat.strongly.generating}
  A functor $f\colon\qA\to\qB$ in the $\infty$-cosmos $\qCat$ of quasi-categories is strongly generating if and only if an arrow $\beta\colon b\to c$ in $\qB$ is invertible precisely when its action by post-composition on hom-spaces $\Hom_{\qB}(fa,\beta)\colon\Hom_{\qB}(fa,b)\to\Hom_{\qB}(fa,c)$ is an equivalence of Kan complexes for all objects $a\in\qA$.
\end{prop}

\begin{proof}
  Given functors $h,k\colon\qX\to\qB$ and a 2-cell $\beta\colon h\Rightarrow k$ in $\qCat$ we know, by the discussion in Example~\ref{ex:hom.spaces}, that the action of the induced map $f\comma\beta\colon f\comma h\to f\comma k$ on the fibres over objects $a\in \qA$ and $x\in \qX$ is the action of the component $\beta x\colon hx\to gx\in\qB$ by post-composition on hom-spaces $\Hom_{\qB}(fa,\beta x)\colon \Hom_{\qB}(fa,hx)\to \Hom_{\qB}(fa,kx)$. So we know, by Corollary~\ref{cor:mod.equiv.fibrewise}, that $f\comma\beta$ is an equivalence if and only if $\Hom_{\qB}(fa,\beta x)$ is an equivalence of Kan complexes for all objects $a\in\qA$ and $x\in\qX$. Note also that Corollary~\refI{cor:pointwise-equiv} tells us that $\beta$ is an invertible 2-cell if and only if for each object $x\in\qX$ its component $\beta x\colon hx\to kx$ is an isomorphism in $\qB$. The stated result now follows from combining these facts.
\end{proof}

\begin{obs}\label{obs:gen-set}
  The characterisation of the Proposition \ref{prop:qcat.strongly.generating} reveals that a functor $f\colon\qA\to\qB$ of quasi-categories is strongly generating iff the set of vertices $\{fa\in \qB_0\colon a\in\qA_0\}$ has the property that ``homming'' out of the vertices in that set detects isomorphisms in $\qB$. In particular, this characterisation says nothing about the rest of the structure of the functor $f$.

  In the quasi-categorical setting, we say that some set $X$ of vertices in a quasi-category $\qB$ is strongly generating if it has this isomorphism detection property. Then we observe that a functor $f\colon\qA\to\qB$ is strongly generating iff it maps surjectively onto some strongly generating set $X\subseteq\qB_0$.
\end{obs}

To find further examples of fully faithful and strongly generating functors between quasi-categories, recall from Definition \refVII{defn:ho-notions} that a simplicial functor $F\colon\eC\to\eD$ between Kan-complex-enriched categories is
  \begin{enumerate}[label=(\roman*)]
  \item \emph{homotopically fully-faithful} if its action $F\colon\Map_{\eC}(A,B)\to\Map_{\eD}(FA,FB)$ on each hom-space is an equivalence of Kan complexes, and
  \item \emph{homotopically strongly generating} if a $0$-arrow $e\colon A\to B$ in $\eD$ is an equivalence if and only if for all objects $C$ in $\eC$ the map $\Map_{\eD}(FC,e)\colon\Map_{\eD}(FC, A)\to\Map_{\eD}(FC,B)$ is an equivalence of Kan complexes.
  \end{enumerate}

\begin{prop}\label{prop:ho-notions}
Consider a simplicial functor $F\colon\eC\to\eD$ of Kan-complex-enriched categories together with the functor of quasi-categories   $f\colon\qC\to\qD$ constructed by application of the homotopy coherent nerve construction. Then
  \begin{enumerate}[label=(\roman*)]
  \item\label{lab:ho-notions.1} $f$ is fully-faithful if and only if $F$ is homotopically fully-faithful, and
  \item\label{lab:ho-notions.2} $f$ is strongly generating if and only if $F$ is homotopically strongly generating.
  \end{enumerate}
\end{prop}

\begin{proof}
  Recall from Corollary \refVI{cor:fun-to-hom} that there are canonical equivalences $\Map_{\eC}(A,B)\we\Hom_{\qC}(A,B)$ between mapping spaces of $\eC$ and hom-spaces of its homotopy coherent nerve $\qC$, which are natural in the sense that the square
  \begin{equation*}
    \xymatrix@R=2em@C=4em{
      {\Map_{\eC}(A,B)} \ar[r]^-{F}\ar[d]_{\simeq} &
      {\Map_{\eD}(FA,FB)} \ar[d]^{\simeq} \\
      {\Hom_{\qC}(A,B)} \ar[r]_-{f}
      & {\Hom_{\qD}(FA,FB)}
    }
  \end{equation*}
  commutes for each pair of objects $A,B\in \eC$. So by composition and cancellation of equivalences of Kan complexes, it follows that the upper horizontal in this square is an equivalence if and only if the lower horizontal is; thus establishing \ref{lab:ho-notions.1} via Proposition~\ref{prop:qcat.fully.faithful}.

  Now observe that a $0$-arrow $e\colon A\to B$ is an equivalence in $\eD$ if and only if it is an isomorphism in the quasi-category $\qD$. Also observe that the square
  \begin{equation*}
    \xymatrix@R=2em@C=5em{
      {\Map_{\eD}(FC,A)} \ar[r]^-{\Map_{\eC}(FC,e)}\ar[d]_{\simeq} &
      {\Map_{\eD}(FC,B)} \ar[d]^{\simeq} \\
      {\Hom_{\qD}(FC,A)} \ar[r]_-{\Hom_{\qC}(FC,e)}\ar@{}[ur]|{\cong}
      & {\Hom_{\qD}(FC,B)}
    }
  \end{equation*}
  commutes up to isomorphism for each $C\in\eC$. This result follows straightforwardly from the concrete construction of the vertical equivalences in this square given in Propositions~\refVI{prop:fun-to-r-hom} and~\refVI{prop:hom-space-comparison}. Here again it follows that the upper horizontal map $\Map_{\eC}(FC,e)$ is an equivalence of Kan complexes if and only if $\Hom_{\qC}(FC,e)$ is an equivalence of Kan complexes. Combining these observations we establish \ref{lab:ho-notions.2} via Proposition~\ref{prop:qcat.strongly.generating}.
\end{proof}

\begin{lem}\label{lem:rep-ho-generate}
  If $\qA$ is a quasi-category then the set of representable cartesian fibrations $\{p_{\qA}\colon \qA\comma a\tfib \qA\mid a\in\qA\}$ is homotopically strongly generating in ${}_1\qMod(\qCat)_{\qA}$. Consequently, this set of representables is strongly generating in the corresponding quasi-category ${}_1\qMod(\qqCat)_{\qA}$.
\end{lem}

\begin{proof}
  Suppose that we are given a 0-arrow
  \[
    \xymatrix{ \qE\ar[rr]^g \ar@{->>}[dr]_p & & \qF \ar@{->>}[dl]^q \\ & \qB}
  \]
  in $\eM\defeq {}_1\qMod(\qCat)_{\qA} = \Cart(\qCat)\gr_{/\qA}$ with the property that for all vertices $a\colon 1\to \qA$ the map
  \begin{equation*}
    \Fun_{\eM}(\yoneda(a), g)\colon \Fun_{\eM}(\yoneda(a), p\colon \qE\tfib \qB)\to \Fun_{\eM}(\yoneda(a), q\colon \qF\tfib \qB)
  \end{equation*}
  is an equivalence of Kan complexes. The Yoneda lemma for groupoidal cartesian fibrations, Corollary \refIV{cor:groupoidal-yoneda}, applies to show that this map of functor spaces is equivalent to the action $g_a\colon E_a\to F_a$ of $g$ on fibres over $a$, and so we may apply  Proposition~\ref{prop:equivalence-of-fibrations} to infer that $g$ is an equivalence as required. The conclusion follows from Proposition \ref{prop:ho-notions}.
\end{proof}

\subsection{Limits and colimits in an \texorpdfstring{$\infty$}{infinity}-category}\label{ssec:limits}

Via the nerve embedding, diagrams indexed by small categories are among the diagrams indexed by small simplicial sets. The simplicial cotensors of axiom \ref{defn:cosmos}\ref{defn:cosmos:a} are used to define $\infty$-categories of diagrams.

\begin{defn}[diagram $\infty$-categories]\label{defn:diagram-cats} If $J$ is a small simplicial set and $A$ is an $\infty$-category, then the $\infty$-category $A^J$ is naturally thought of as being the \emph{$\infty$-category of $J$-indexed diagrams in $A$}.
\end{defn}

\begin{defn}\label{defn:all-limits} An $\infty$-category $A$ \emph{admits all limits of shape $J$} if the constant diagram functor $\Delta \colon A \to A^J$,constructed by applying the contravariant functor $A^{(-)}$ to the unique simplicial map $J\to \Del^0$, has a right adjoint:
\[ \adjdisplay \Delta -| \lim : A^J -> A.\]
\end{defn}

This definition is insufficiently general since many $\infty$-categories will have some, but not all, limits of diagrams of a particular indexing shape. The ``partial adjunctions'' of Definition~\ref{defn:absolute-right-lifting} precisely address this problem:

\begin{defn}[limits of families]\label{defn:limit}
  If $J$ is a small simplicial set and $A$ and $D$ are $\infty$-categories then we can regard a functor $d\colon D\to A^J$ as being an \emph{family of $J$-indexed diagrams in $A$}. We say that the members of such a family \emph{admits a family of limits} $\ell\colon D\to A$ if there exists an absolute right lifting diagram:
  \begin{equation}\label{eq:lim-diagram-defn}
    \xymatrix{
      \ar@{}[dr]|(.7){\Downarrow\lambda} & A\ar[d]^{\Delta_A} \\
      D \ar[r]_d \ar[ur]^{\ell} & A^J}
  \end{equation}

  In the case where $D$ is taken to be the terminal $\infty$-category $1$, we think of $d$ as being a single diagram of shape $J$, $\ell$ as its limit, and $\lambda$ as the limiting cone.
\end{defn}

\begin{defn}[$\infty$-categories of cones]\label{defn:cones} For any diagram $d \colon 1 \to A^J$ of shape $J$ in an $\infty$-category $A$, the \emph{$\infty$-category of cones over $d$} is the comma $\infty$-category $p_0 \colon \Delta \comma d \tfib A$ formed by the pullback 
\[
    \xymatrix@=2.5em{
      {\Delta \comma d}\pbexcursion \ar[r]\ar@{->>}[d]_{(p_1,p_0)} &
      {A^{J \times \cattwo}} \ar@{->>}[d]^{(p_1,p_0)} \\
      {1 \times A} \ar[r]_-{d\times \Delta} & {A^J\times A^J}
    }
\]
Dually, the \emph{$\infty$-category of cones under $d$} is the comma $\infty$-category $p_1 \colon d \comma \Delta \tfib A$.
\end{defn}

The following result recasts Definition \ref{defn:limit} in terms of fibred equivalences of comma $\infty$-categories by specialising Proposition  \ref{prop:absliftingtranslation}:

\begin{prop}[{\refI{prop:absliftingtranslation2}}]\label{prop:limit-as-module-equivalence}
  Given a small simplicial set $J$ and $\infty$-categories $D$ and $A$, then a family $d\colon D\to A^J$ of $J$-indexed diagrams admits a family of limits $\ell\colon D\to A$ if and only if the the $\infty$-category of cones $\Delta\comma d$ is equivalent to $A \comma \ell$ over $D \times A$.
\end{prop}

Colimits are characterised dually by \emph{absolute left lifting diagrams}; that is to say a triangle
  \begin{equation}\label{eq:abs-left-lifting}
    \xymatrix{
      \ar@{}[dr]|(.7){\Uparrow\lambda} & B \ar[d]^f \\
      C \ar[ur]^\ell \ar[r]_g & A}
  \end{equation}
  in which the direction of the 2-cell is switched relative to that in~\eqref{eq:abs-right-lifting} and which enjoys a universal property akin to that in~\eqref{eq:abs-rl-univ} but with the sense of all 2-cells reversed. 
   Propositions \ref{prop:absliftingtranslation} and \ref{prop:limit-as-module-equivalence} dualise to characterise absolute left lifting diagrams and colimits as fibred equivalences between comma $\infty$-categories.

\begin{prop}\label{prop:fully-faithful-reflects} A fully faithful functor $f \colon A \to B$ reflects all limits or colimits that exist in $B$.
\end{prop}
\begin{proof} The statement for limits asserts that given any family of diagrams $d \colon D \to A^J$ of shape $J$ in $A$, any functor $\ell \colon D \to A$ and cone $\lambda \colon \Delta \ell \To d$ as below-left so that the whiskered composite with $f^J \colon A^J \to B^J$ displayed below is an absolute right lifting diagram
\[
\begin{tikzcd} 
\arrow[dr, phantom, "\scriptstyle\Downarrow\lambda" pos=.85] & A \arrow[d, "\Delta"] \arrow[r, "f"] & B \arrow[d, "\Delta"] \\ D \arrow[ur, "\ell"] \arrow[r, "d"'] & A^J \arrow[r, "f^J"'] & B^J
\end{tikzcd}
\]
then $(\ell,\lambda)$ defines an absolute right lifting of $d \colon D \to A^J$ through $\Delta \colon A \to A^J$. By Proposition \ref{prop:absliftingtranslation} applied to Definition \ref{defn:fully.faithful}, to say that $f$ is fully faithful is to say that $\id_A \colon A \to A$ defines an absolute right lifting of $f$ through itself. So by Lemma \ref{lem:comp-canc-abs-lift} and the hypothesis just stated, the composite diagram below-left is an absolute right lifting diagram, and by 2-functoriality of the simplicial cotensor with $J$, the diagram below-left coincides with the diagram below-right:
\[
\begin{tikzcd} & & A \arrow[d, "f"] & & & A \arrow[d, "\Delta"] \\  \arrow[dr, phantom, "\scriptstyle\Downarrow\lambda" pos=.85] & A \arrow[d, "\Delta"] \arrow[ur, equals] \arrow[r, "f"'] & B \arrow[d, "\Delta"] \arrow[r, phantom, "="] &~ & \arrow[dr, phantom, "\scriptstyle\Downarrow\lambda" pos=.25] & A^J \arrow[d, "f^J"] \\ D \arrow[ur, "\ell"] \arrow[r, "d"'] & A^J \arrow[r, "f^J"'] & B^J & D \arrow[r, "d"'] \arrow[uurr, "\ell"] & A^J \arrow[r, "f^J"'] \arrow[ur, equals] & B^J
\end{tikzcd}
\]
By Definition \ref{defn:fully.faithful} to say that $f$ is fully faithful is to say that $\bar{f} \colon A^\cattwo \we f \comma f$ is a fibred equivalence over $A \times A$. Applying $(-)^J \colon \eK \to \eK$, yields a fibred equivalence $\bar{f^J} \colon (A^J)^\cattwo \we f^J \comma f^J$ over $A^J \times A^J$, proving that if $f \colon A \to B$ is fully faithful, then $f^J \colon A^J \to B^J$ is also. Hence by Proposition \ref{prop:absliftingtranslation}, $\id_{A^J} \colon A^J \to A^J$ defines an absolute right lifting of $f^J$ through itself. Applying Lemma \ref{lem:comp-canc-abs-lift}  again, we now conclude that $(\ell,\lambda)$ is an absolute right lifting of $d$ through $\Delta$ as required.
\end{proof}

\begin{prop}\label{prop:ff.and.sg.limit.pres}
  Suppose that we are given a functor of $\infty$-categories $f\colon A\to B$ which is both fully-faithful and strongly generating. Assume also that  $d\colon D\to A^J$ is a family of $J$ indexed diagrams which admits a limit in $A$ and that the transformed diagram $\xymatrix@1{{D}\ar[r]^{d} & {A^J}\ar[r]^{f^J} & {B^J}}$ admits a limit in $B$. Then the functor $f$ preserves the limit of the family of diagrams $d$.
\end{prop}

\begin{proof}
  Consider the following diagrams in the homotopy 2-category $\ho_*\eK$
  \begin{equation*}
    \xymatrix@R=2em@C=3em{
      {}\ar@{}[dr]|(0.7){\Downarrow\lambda} & {A} \ar[r]^{f}\ar[d]^{{\Delta}} &
      {B}\ar[d]^{{\Delta}} & {\mkern20mu} &
      {}\ar@{}[dr]|(0.7){\Downarrow\lambda'} & {B} \ar[d]^{{\Delta}}  \\
      {D}\ar[r]_{d}\ar[ur]^{\ell} & {A^J}\ar[r]_{f^J} & {B^J} & {} &
      {D}\ar[r]_{f^Jd}\ar[ur]^{\ell'} & {B^J}
    }
  \end{equation*}
  in which the triangles display $\ell$ and $\ell'$ as the limits of the diagrams $d$ and $f^Jd$ respectively and the commutative square expresses the naturality of the family of diagonal maps ${\Delta}\colon A\to A^J$. Applying the universal property of the right-hand lifting diagram to the pasted diagram on the left we obtain a unique 2-cell $\alpha\colon f\ell\Rightarrow\ell'$ with $\lambda'\cdot{\Delta}\alpha = f^J\lambda$, and our result follows if we can show that $\alpha$ is invertible. To that end consider the following diagrams
  \begin{equation}\label{eq:some.squares}
    \vcenter{
      \xymatrix@=1.5em{
        {D}\ar[r]^{\ell}\ar@{=}[d] & {A}\ar[d]^f &  {A}\ar@{=}[l]\ar@{=}[d]\\
        {D}\ar[r]^{f\ell}\ar@{=}[d] & {B}\ar@{=}[d] & {A}\ar[l]_{f}\ar@{=}[d]\\
        {D}\ar[r]^{\ell'}\ar@{=}[d] & {B}\ar[d]^{{\Delta}} &
        {A}\ar[l]_{f}\ar@{=}[d]\\
        {D}\ar[r]_{f^Jd} & {B^J} & {A}\ar[l]^{{\Delta}f}
        \ar@{}"3,1";"2,2"\ar@{<=}?(0.35);?(0.65)^{\alpha}
        \ar@{}"4,1";"3,2"\ar@{<=}?(0.35);?(0.65)^{\lambda'}
      }}\mkern10mu\text{(A)}\mkern100mu
    \vcenter{
      \xymatrix@=1.5em{
        {D}\ar[r]^{\ell}\ar@{=}[d] & {A}\ar[d]^{{\Delta}} &
        {A}\ar@{=}[l]\ar@{=}[d]\\
        {D}\ar[r]_{d}\ar@{=}[d] & {A^J}\ar[d]_{f^J} &
        {A}\ar@{=}[d]\ar[l]^(0.45){{\Delta}}\\
        {D}\ar[r]_{f^Jd} & {B^J} & {A}\ar[l]^(0.45){f^J{\Delta}} 
        \ar@{}"2,1";"1,2"\ar@{<=}?(0.35);?(0.65)^{\lambda}
      }
    }\mkern10mu\text{(B)}
  \end{equation}
  the rows of which are diagrams of the form discussed in Observation~\ref{obs:trans.induce.comma}. These give rise to the following diagram  of comma objects and induced functors between them
  \begin{equation}\label{eq:induced.square}
    \xymatrix@R=2em@C=3em{
      {A\comma\ell}\ar[d]_{\text{\ref{item:equivs.a}}}
      \ar[r]^{\text{\ref{item:equivs.c}}} &
      {f\comma f\ell}\ar[r]^{\text{(e)}} &
      {f\comma\ell'}\ar[d]^{\text{\ref{item:equivs.b}}}\\
      {{\Delta}\comma d}\ar[r]_-{\text{\ref{item:equivs.d}}} &
      {f^J{\Delta}\comma f^Jd}\ar@{=}[r] &
      {{{\Delta}f\comma f^Jd}}
      \ar@{}"1,1";"2,3"|*{\cong}
    }
  \end{equation}
  in the slice over $D\times A$. Here the functors on the path along the top and then down the right-hand side are those induced by the successive rows of diagram (A) in~\eqref{eq:some.squares}. Those on the path down the left-hand side and then along the bottom are those induced by the successive rows of diagram (B) in~\eqref{eq:some.squares}. Notice, however, that each pasting composite of a column of (A) is equal to the pasting composite of the corresponding column of (B), which for the right hand columns follows from the defining equality $\lambda'\cdot{\Delta}\alpha = f^J\lambda$ of $\alpha$. So the functoriality comments of Observation~\ref{obs:trans.induce.comma} apply to give the invertible 2-cell relating the two legs of the square in~\eqref{eq:induced.square}. Considering the functors labelled \ref{item:equivs.a}--\ref{item:equivs.d} in~\eqref{eq:induced.square} we see that each one is an equivalence because it
  \begin{enumerate}[label=(\alph*)]
  \item\label{item:equivs.a} is the functor induced by the absolute lifting $\lambda\colon{\Delta}\ell\Rightarrow d$,
  \item\label{item:equivs.b} may be constructed by taking the equivalence $B\comma\ell'\to{\Delta}\comma f^Jd$ induced by the absolute lifting $\lambda'\colon\Delta^B\ell'\Rightarrow f^Jd$ and pulling it back along the functor $D\times f\colon D\times A\to D\times B$,
  \item\label{item:equivs.c} may be constructed by taking the induced functor $A^{\cattwo}\to f\comma f$, which is an equivalence by the assumption that $f$ fully-faithful, and pulling it back along the functor $\ell\times A\colon D\times A\to A\times A$,
  \item\label{item:equivs.d} may be constructed by taking the induced functor $(A^J)^{\cattwo}\to f^J\comma f^J$, which is an equivalence because it may be constructed by applying the comma preserving functor $(-)^J$ to the equivalence $f\colon A^{\cattwo}\to f\comma f$, and pulling it back along the functor $d\times{\Delta}\colon D\times A\to A^J\times A^J$.
  \end{enumerate}
  Applying the composition and cancellation rules for equivalences to~\eqref{eq:induced.square}, we may infer that the functor labelled (e) there is also an equivalence. This latter functor is that induced by $\alpha\colon f\ell\to \ell'$ as in Definition~\ref{defn:strong.gen}, from which it follows, by the strong generation assumption on $f$, that $\alpha$ is an isomorphism as required.
\end{proof}


\section{Construction of limits and colimits}\label{sec:construction} 

We are finally ready to assemble the results of the previous sections and prove our main theorems, extending results proven for quasi-categories by Lurie in \cite[\S 6]{Lurie:2009fk}. Our contribution is to supply proofs independent of his that apply natively to arbitrary $\infty$-cosmoi. An alternate approach might be to use the Yoneda lemma to extend Lurie's results to the general $\infty$-cosmological setting.  

In \S\ref{ssec:complete-quasi}, we give a comprehensive review of the main theorems from \cite{RiehlVerity:2018rq} which allow us to verify that the codomain of the generalised Yoneda embedding is a complete quasi-category and explicitly calculate its quasi-categorical limits as pseudo homotopy limits in a Kan-complex-enriched category, with an expanded array of applications.  In \S\ref{ssec:yoneda-pres}, we use this result to show that the generalised Yoneda embedding $\yoneda\colon\Fun_{\eK}(D,A)\to{}_D\qMod(\qK)_A$ for any pair of $\infty$-categories in any $\infty$-cosmos preserves those limits in $\Fun_{\eK}(D,A)$ derived from the limit of a diagram $d \colon D \to A^J$ in $\eK$. We apply these results in \S\ref{ssec:colimit-construction} to prove  Theorem \ref{thm:limit-construction} and its dual: reducing the question of whether an $\infty$-category $A$ is complete to the question of whether it has products and pullbacks.

\subsection{Complete and cocomplete quasi-categories}\label{ssec:complete-quasi}

We recall the main theorem from \S\refVII{sec:complete} together with a few of its consequences.

\begin{thm}[{\refVII{thm:nerve-completeness},\refVII{thm:nerve-completeness-converse}}]\label{thm:nerve-completeness}
For any Kan-complex-enriched category $\eC$ and simplicial set $X$, if a homotopy coherent diagram $D \colon \gC[X] \to \eC$ admits a pseudo homotopy limit in $\eC$, then the corresponding limit cone $\gC[\Del^0\join X] \to \eC$ transposes to define a limit cone over the transposed diagram $d \colon X \to \qC$ in the homotopy coherent nerve of $\eC$. Conversely, if the diagram $d$ admits a limit in the quasi-category $\qC$, then the limit cone $\Del^0\join X \to \qC$ transposes to define a pseudo homotopy limit cone over $D$ in $\eC$. 

  Consequently, the quasi-category $\qC$ is complete if and only if $\eC$ admits pseudo homotopy limits for all simplicial sets $X$.
\end{thm}

This result is essentially the same as Lurie's  \cite[4.2.4.1]{Lurie:2009fk}, though employ a different model for the point-set level homotopy limits, similar to the homotopy limits explored in \cite{Steimle:2011hc}, that we find particular amenable to the sort of analysis we require here. In particular, the data of a pseudo homotopy limit cone is the direct transpose of the data of a quasi-categorical limit cone. 

In classical homotopy theory, a \emph{homotopy coherent diagram} is a simplicial functor $\gC{X} \to \eC$ whose domain is the \emph{homotopy coherent realisation} of the simplicial set $X$, a simplicial category formed as the left adjoint to the homotopy coherent nerve:
\[ \adjdisplay \gC -| \hN : \sCat -> \sSet.\]
 The \emph{pseudo homotopy limits} in the statement of Theorem \ref{thm:nerve-completeness} refer to a flexible weighted homotopy limit with a particular weight  $W_X$ appropriate to diagrams of this shape, with the property that a homotopy coherent diagram of shape $X$ and a $W_X$-shaped cone over that diagram together assemble into a simplicial functor $\gC[\Del^0 \join X] \to \eC$ (see  \S\refVII{ssec:collage}). 

\begin{defn}[{\refVII{defn:weight-for-pseudo-limits}, \refVII{lem:pseudo-is-flexible}}]\label{defn:weight-for-pseudo-limits}
For any simplicial set $X$, the  \emph{weight for the pseudo limit of a homotopy coherent diagram of shape $X$} is the functor
  \begin{equation*}
    \xymatrix@R=0em@C=5em{
      \gC{X}\ar[r]^{W_X} & {\SSet}
    }\qquad\text{given~by}\qquad W_X(x) \defeq \Fun_{\gC[\Del^0\join X]}(\bot,x). 
  \end{equation*}
  Lemma \refVII{lem:pseudo-is-flexible} verifies that this is a flexible weight.  The $W_X$-weighted limit of a homotopy coherent diagram of shape $X$ is then referred to as the \emph{pseudo limit} of that diagram.
  \end{defn}
  
Theorem \ref{thm:nerve-completeness} implies its dual by replacing the Kan-complex-enriched category with its opposite. Thus, it also enables us to deduce cocompleteness results, such as the following theorem, first proven by Barnea, Harpaz, Horel \cite[2.5.9]{BarneaHarpazHorel:2017pc}.

\begin{prop}[{\refVII{prop:qcat-simplicial-model}}]\label{prop:qcat-simplicial-model}
  If $\eM$ is a simplicial model category then the quasi-category $\qM$, defined as the homotopy coherent nerve of the full simplicial subcategory of fibrant-cofibrant objects, is small complete and cocomplete.
\end{prop}  
  
The groupoidal objects in an $\infty$-cosmos define a Kan-complex-enriched category that is again an $\infty$-cosmos (Proposition \refVII{prop:gpdal-infty-cosmos}) and in particular admits flexible weighted limits. Consequently:
 
\begin{prop}[{\refVII{prop:qcat-of-space-complete}}]\label{prop:qcat-of-space-complete} For any $\infty$-cosmos $\eK$, the large quasi-category $\qS_{\eK}$ of groupoidal $\infty$-categories in $\eK$ is complete and closed under all small limits in the quasi-category $\qK$.
\end{prop}

Recall from Notation \ref{ntn:qcat-ntn} that the quasi-category $\qK \defeq Ng_*\eK$ of $\infty$-categories in an $\infty$-cosmos $\eK$ is defined by passing to the $(\infty,1)$-categorical core before applying the homotopy coherent nerve. In  \refVII{ssec:htpy-limit}, we prove: 

\begin{prop}[{\refVII{cor:infty-one-core-flex}}]\label{prop:infty-one-core-flex}
The $(\infty,1)$-core of an $\infty$-cosmos admits flexible weighted homotopy limits.
\end{prop}

Consequently:

\begin{prop}[{\refVII{prop:qcat-of-cosmos-complete}}]\label{prop:qcat-of-cosmos-complete} For any $\infty$-cosmos $\eK$, the large quasi-category $\qK$ of $\infty$-categories in $\eK$ is small complete.
\end{prop}

Szumi\l{}o proves a similar result in the context of (unenriched) cofibration categories  \cite{Szumilo:2014tm}.

\begin{ex}[completeness of quasi-categories of cartesian fibrations]\label{ex:comp-fibs}
  In Proposition \ref{prop:cartesian-cosmoi}, we demonstrated that the  categories $\Cart(\eK)_{/B}$ and $\coCart(\eK)_{/B}$ of (co)cartesian fibrations and cartesian functors over an object $B$ support an $\infty$-cosmos structure created by the inclusions $\Cart(\eK)_{/B}\inc\eK_{/B}$ and $\coCart(\eK)_{/B}\inc\eK_{/B}$. Consequently, Propositions \ref{prop:qcat-of-cosmos-complete} and \ref{prop:qcat-of-space-complete} apply to show that the Kan-complex-enriched categories  $g_*(\Cart(\eK)_{/B})$ and $g_*(\coCart(\eK)_{/B})$ and their subcategories of groupoidal objects $\Cart(\eK)\gr_{/B}$ and $\coCart(\eK)\gr_{/B}$ are closed in $g_*(\eK_{/B})$ under flexible weighted limits, and thus that the corresponding quasi-categories $\Cart(\qK)_{/B}$, $\coCart(\qK)_{/B}$, $\Cart(\qK)\gr_{/B}$ and $\coCart(\qK)\gr_{/B}$ are complete and closed under limits in the quasi-category $\qK_{/B}$.
  \end{ex}
  
  \begin{ex}\label{ex:comp-mods}
Definition~\ref{defn:module} tells us that ${}_A\qMod(\eK)_B$ may be expressed as an intersection of three simplicial subcategories $g_*(\Cart(\eK_{/A})_{/A\times B\to A})$, $g_*(\coCart(\eK_{/B})_{/A\times B\to B})$ and $\eK_{/A\times B}\gr$ in the $\infty$-cosmos $\eK_{/A\times B}$. As noted in the Example \ref{ex:comp-fibs} and Proposition \ref{prop:qcat-of-space-complete}, each of these simplicial subcategories is closed in $g_*(\eK_{/A\times B})$ under flexible weighted limits. It follows that ${}_A\qMod(\eK)_B$ is also closed in $g_*(\eK_{/A\times B})$ under small flexible weighted limits and so Theorem~\ref{thm:nerve-completeness} applies to show that the quasi-category ${}_A\qMod(\qK)_B$ is closed in $\qK_{/A\times B}$ under small limits.
\end{ex}

\begin{rmk}\label{rmk:cosmos-pres-qcat-limits}  Any cosmological functor $F\colon\eK\to\eL$ induces corresponding simplicial functors between the respective Kan-complex-enriched categories of (co)cartesian fibrations, groupoidal (co)cartesian fibrations and modules. Since $F$ preserves the products, pullbacks, and sequential limits used to construct the flexible weighted limits in each of those categories, cosmological functors preserve  flexible weighted limits. Consequently the corresponding functors of the complete quasi-categories constructed by applying the homotopy coherent nerve construction all preserve small limits.
\end{rmk}

By Propositions \ref{prop:flexible-weights-are-htpical} and \ref{prop:infty-one-core-flex}, $\infty$-cosmoi and their groupoidal cores admit pseudo (homotopy) limits of homotopy coherent diagrams. In preparation for \S\ref{ssec:colimit-construction}, we calculate pseudo homotopy limits --- applying Definition \ref{defn:flexible-hty-limit} to the weight of Definition \ref{defn:weight-for-pseudo-limits} --- for  simple but important diagram shapes.  Consider as the indexing 1-category $\qJ$ either:
\begin{itemize}
\item  a discrete category, 
\item the pullback shape $\pbshape$, or
\item  the category $\omega\op$ indexing inverse sequences.
\end{itemize}
 In each case, $\qJ$ is a free category on an underlying graph $\qG\inc\qJ$ of ``atomic'' arrows, which we regard as a 1-skeletal simplicial set. As the following lemma explains, in such contexts, strict diagrams $\qJ \to \eC$ are automatically ``homotopy coherent.''
 
 \begin{lem}\label{itm:free-is-hty-coh} Let $\qJ$ be a 1-category freely generated by the graph $\qG \inc \qJ$. 
 \begin{enumerate}[label=(\roman*)]
 \item\label{itm:free-is-hty-coh-i} The homotopy coherent realisation $\gC[\qG]$ is isomorphic to $\qJ$, regarded as a simplicial category with discrete hom sets. Hence diagrams $\qJ \to \eC$ in a Kan complex enriched category, correspond bijectively to diagrams $\qG \to \hN\eC$ in the homotopy coherent nerve.
 \item\label{itm:free-is-hty-coh-ii} For any Kan complex enriched category $\eC$, the quasi-categories $\hN\eC^{\qJ}$ and $\hN\eC^{\qG}$ of diagrams are equivalent. Hence up to equivalence, we can represent a quasi-categorical diagram $\qJ \to \hN\eC$ by a point-set diagram $\qJ \to \eC$.
 \end{enumerate}
 \end{lem}
 \begin{proof}
The isomorphism $\gC[\qG]\cong\qJ$ of \ref{itm:free-is-hty-coh-i} is easily recognised  from the explicit description of the homotopy coherent realisation functor given in Proposition \refVII{prop:gothic-C}: the homotopy coherent realisation of any 1-skeletal simplicial set is the free discrete category generated by this graph.  Hence, diagrams $\qJ \cong \gC[\qG] \to \eC$ in a Kan complex enriched category, correspond bijectively to diagrams $\qG \to \hN\eC$ in the homotopy coherent nerve. 

For \ref{itm:free-is-hty-coh-ii}, since $\hN\eC$ is a quasi-category and $\qG \inc \qJ$ is inner anodyne when considered as a monomorphism of simplicial sets, the quasi-categories $\hN\eC^{\qJ}$ and $\hN\eC^{\qG}$ of diagrams are equivalent. So, up to equivalence, we can represent any diagram $\qJ \to \hN\eC$ by its restriction $\qG\inc\qJ\to \hN\eC$, which transposes to a strictly commuting diagram in $\eC$ by  \ref{itm:free-is-hty-coh-i}.
\end{proof}

\begin{defn}\label{defn:strict-pseudo-cones} When $\qJ$ is the free category generated by a graph $\qG$,   a \emph{strictly commuting pseudo cone} over a diagram $F \colon \qJ \cong \gC[\qG] \to \eC$ is formed by restricting a strict cone $\alpha \colon \Delta L \To F(-)$, presenting as a simplicial natural transformation $\alpha \colon 1 \to \Map_{\eC}(L,F(-))$, along the unique map $W_{\qG} \to 1$ of weights:
\[
\xymatrix{ W_{\qG} \ar[r]^{!} & 1 \ar[r]^-{\alpha} & \Map_{\eC}(L,F-).}\]
\end{defn}

We have the following result:

\begin{prop}\label{prop:strict-pseudo-cones} Suppose that $\qJ$ is a discrete category, the pullback shape $\pbshape$, or $\omega\op$ with generating subgroup $\qG\inc\qJ$, and  $F \colon \qJ \to \eK$ is a diagram valued in an $\infty$-cosmos $\eK$, in which one of the maps is an isofibration in case of $\pbshape$ and all of the maps are isofibrations in the case of $\omega\op$. Then the strictly commuting pseudo cone formed from the limit cone $\pi \colon \Delta \lim(F) \to F(-)$ presents $\lim(F)$ as a pseudo homotopy limit of the diagram $F$ in the $(\infty,1)$-categorical core $g_*\eK$.
\end{prop}

\begin{proof}
Given $X \in \eK$, post-composition by the weighted cone $\pi!\colon W_{\qG}\to \Fun_{g\eK}(\lim(F), F-)$ determines a map
  \begin{equation}\label{eq:cos-lim-flex-1}
    \xymatrix@R=0em@C=6em{
      {\Fun_{g_*\eK}(X,\lim(F))}\ar[r] &
      {\wlim{W_{\qG}}{\Fun_{g_*\eK}(X,F-)}}
    }
  \end{equation}
  of Kan complexes, and our task is to show that this is an equivalence. Notice, however, that we have
  \begin{align*}
    \Fun_{g_*\eK}(X,\lim(F)) &= g(\Fun_{\eK}(X,\lim(F)))
    \cong g(\lim(\Fun_{\eK}(X,F-))) \\
    &\cong \lim(g(\Fun_{\eK}(X,F-))) =
    \lim(\Fun_{g_*\eK}(X,F-)) \cong \wlim{1}{\Fun_{g_*\eK}(X,F-)}
  \end{align*}
  in which the first isomorphism follows because $\lim(F)$ is a simplicially enriched limit in $\eK$ and the second because the groupoid core functor is a right adjoint on underlying categories. Under this isomorphism it is easily checked that the map in~\eqref{eq:cos-lim-flex-1} is isomorphic to the map
  \begin{equation*}
    \xymatrix@R=0em@C=6em{
      {\wlim{1}{\Fun_{g_*\eK}(X,F-)}}\ar[r] &
      {\wlim{W_{\qG}}{\Fun_{g_*\eK}(X,F-)}}
    }
  \end{equation*}
  induced by the unique map of weights $!\colon W_{\qG}\to 1$. It follows that it is enough to show that for any diagram $F\colon\qJ\to \Kan$ of the appropriate kind in Kan complexes the induced map $\wlim{!}{F}\colon\wlim{1}{F}\to\wlim{W_{\qG}}{F}$ from the strict limit of $F$ to the pseudo limit of $F$ is an equivalence, which is achieved by the next three lemmas.
  \end{proof}
  
\begin{lem}\label{lem:strict-pseudo-products} For any family of objects $\{A_i\}$ in a simplicial category with products, the strict limit cone $\pi \colon \prod_i A_i \tfib A_i$ defines a pseudo homotopy limit cone.
\end{lem}
\begin{proof} In this case the result is trivial because the weight for the pseudo limit of a discrete diagram is isomorphic to the terminal weight.
\end{proof}

In particular, Lemma \ref{lem:strict-pseudo-products} applies to the $\infty$-cosmos $\Kan$ of Kan complexes.

\begin{lem}\label{lem:strict-pseudo-pullbacks} The strict pseudo cone formed from the pullback cone over a diagram of Kan complexes and Kan fibrations
\[
\xymatrix{ \qP \ar[r] \ar@{->>}[d] \pbexcursion & \qC \ar@{->>}[d]^p \\ \qB \ar[r]_f & \qA}\]
defines a pseudo homotopy limit cone in $\Kan$.
\end{lem}
\begin{proof}
Unpacking Definition \ref{defn:weight-for-pseudo-limits}, the weight $W_\pbshape \colon \gC\pbshape\cong\pbshape \to \SSet$ for pseudo limits over the pullback shape is given by the simplicial functor which maps the outer objects of $\pbshape$ to $\Del^0$ and the middle object to $\pbshape\op $.
From the pullback diagram in the statement, we derive the following diagram
    \begin{equation*}
      \xymatrix@R=2em@C=2.5em{
        {\qC}\ar@{->>}[r]^{p}\ar@{=}[d] & {\qA}\ar@{=}[r]\ar@{=}[d] &
        {\qA}\ar@{=}[r]\ar[d]^-{\Delta}_{\simeq} &
        {\qA}\ar@{=}[d] & {\qB}\ar[l]_{f}\ar@{=}[d]\\
        {\qC}\ar@{->>}[r]_{p} & {\qA} & {\qA^{\pbshape\op}}\ar@{->>}[l]^{\simeq}
        \ar@{->>}[r]_{\simeq} & {\qA} & {\qB}\ar[l]^{f}
      }
    \end{equation*}
    Here the upper row is a wide pullback diagram whose limit is simply the pullback of the original diagram. The lower row is the wide pullback diagram whose limit is the end that computes the limit weighted by $W_{\pbshape}$. The middle component of the transformation from top to bottom is an equivalence because $\pbshape$ is contractible in the Kan model structure and $\qA$ is a Kan complex; since pullbacks of isofibrations are equivalence invariant constructions, it follows  that the induced map between the wide pullbacks of these diagrams is an equivalence as required.
\end{proof}

\begin{lem}\label{lem:strict-pseudo-towers} The strict pseudo cone formed from the limit cone over a sequence of Kan fibrations between Kan complexes
\[
\xymatrix@1{{\cdots}\ar@{->>}[r]^{p_n} & {\qA_n}\ar@{->>}[r]^{p_{n-1}} & {\cdots}\ar@{->>}[r]^{p_1} & {\qA_1}\ar@{->>}[r]^{p_0} & {\qA_0}}\]
defines a pseudo homotopy limit cone.
\end{lem}
\begin{proof}
The diagram shape in the statement is the ordered set $\mathbb{N}\op$ with objects $n$ and non-identity edges $n+1 \to n$ and the weight $W_{\mathbb{N}\op} \colon \gC\mathbb{N}\op\cong\omega\op \to \SSet$ maps each object $n$ to the 1-skeletal simplicial set $\mathbb{N}$ with connecting map from one integer to its predecessor  given by the successor map $s\colon \mathbb{N}\to \mathbb{N}$. From the given sequence of Kan fibrations we may derive the following commutative diagram:
    \begin{equation*}
      \xymatrix@R=2em@C=2.5em{
        {\cdots}\ar@{=}[r] & {\qA_n}\ar@{->>}[r]^{p_{n-1}}
        \ar[d]^{\Delta}_{\simeq} &
        {\cdots}\ar@{=}[r] & {\qA_2}\ar@{->>}[r]^{p_1}
        \ar[d]^{\Delta}_{\simeq} &
        {\qA_1}\ar@{=}[r]\ar[d]^{\Delta}_{\simeq} &
        {\qA_1}\ar[d]^{\Delta}_{\simeq}
        \ar@{->>}[r]^{p_0} & {\qA_0}\ar[d]^{\Delta}_{\simeq}
        \ar@{=}[r] & {\qA_0}\ar[d]^{\Delta}_{\simeq}\\
        {\cdots} & {\qA_n^{\mathbb{N}}}\ar@{->>}[r]_{p_{n-1}^{\mathbb{N}}}
        \ar@{->>}[l]^{\qA_n^s} &
        {\cdots} & {\qA_2^{\mathbb{N}}}\ar@{->>}[r]_{p_1^{\mathbb{N}}}
        \ar@{->>}[l]^{\qA_2^s} &
        {\qA_1^{\mathbb{N}}} & {\qA_1^{\mathbb{N}}}\ar@{->>}[r]_{p_0^{\mathbb{N}}}
        \ar@{->>}[l]^{\qA_1^s} &
        {\qA_0^{\mathbb{N}}} & {\qA_0^{\mathbb{N}}}\ar@{->>}[l]^{\qA_0^s}\\
      }
    \end{equation*}
    Here the upper row is a wide pullback diagram whose limit is simply the limit of the original diagram. The lower row is the wide pullback diagram whose limit is the end that computes the limit weighted by $W_{\mathbb{N}\op}$. The component of the transformation from top to bottom are equivalences because $\mathbb{N}$ is contractible in the Kan model structure and each $\qA_n$ is a Kan complex. Again it follows from the equivalence invariance of pullbacks of isofibrations and limits of towers of isofibrations that the induced map between the wide pullbacks of these diagrams is an equivalence as required.
\end{proof}

\subsection{Preservation of limits by generalised Yoneda}\label{ssec:yoneda-pres}

As the last stop on our tour of limit preservation properties we study classes of limits which are preserved by generalised Yoneda embeddings, this time in the setting of a general $\infty$-cosmos, not necessarily biequivalent to $\qCat$.

\begin{obs}\label{obs:limits-are-limits}
  Suppose that $A$ and $D$ are objects of the $\infty$-cosmos $\eK$ and that $K$ and $J$ are simplicial sets. Then transposition under cotensoring provides a bijection between triangles
  \begin{equation}\label{eq:limits-are-limits}
    \vcenter{
      \xymatrix@R=2em@C=3em{
        \ar@{}[dr]|(.7){\Downarrow\lambda} &
        {A^K}\ar[d]^-{\Delta} \\
        D \ar[r]_-{d} \ar[ur]^{\ell} & {(A^K)^J}
      }}
    \mkern30mu\leftrightsquigarrow\mkern30mu
    \vcenter{
      \xymatrix@=2em{
        \ar@{}[dr]|(.7){\Downarrow\hat\lambda} &
        \Fun_{\eK}(D,A)\ar[d]^-{\Delta} \\
        K \ar[r]_-{\hat{d}} \ar[ur]^{\hat\ell} & \Fun_{\eK}(D,A)^J
      }}
  \end{equation}
  in $\eK$ and $\qCat$ respectively. Moreover,  the triangle on the left is a right lifting diagram in $\eK$ if and only if the triangle on the right is a right lifting diagram in $\qCat$.
  \end{obs}

  Notice here that we have, as yet, said nothing about whether these lifts are absolute. To rectify this omission, we first introduce the following definitions:

\begin{defn}\label{defn:limit-stable-under-precomp}
  Suppose that $A$ and $D$ are objects of the $\infty$-cosmos $\eK$, that $K$ and $J$ are simplicial sets and that we are given an absolute lifting diagram
  \begin{equation*}
   \xymatrix@=2em{
      \ar@{}[dr]|(.7){\Downarrow\hat\lambda} &
      \Fun_{\eK}(D,A)\ar[d]^\Delta \\
      K \ar[r]_-{\hat{d}} \ar[ur]^{\hat\ell} & \Fun_{\eK}(D,A)^J
    }
  \end{equation*}
  which presents a limit of a $K$-indexed family $d$ of diagrams of shape $J$ in the functor space $\Fun_{\eK}(D,A)$. We say that this family of limits is \emph{stable under precomposition\/} in $\eK$ iff for each functor $f\colon C\to D$ in $\eK$ it is preserved by the precomposition functor $\Fun_{\eK}(f,A)\colon\Fun_{\eK}(D,A)\to\Fun_{\eK}(C,A)$: i.e., 
      \begin{equation*}
      \xymatrix@=2em{
        \ar@{}[dr]|(.7){\Downarrow\hat\lambda} &
        \Fun_{\eK}(D,A)\ar[d]^-{\Delta}\ar[rr]^{\Fun_{\eK}(f,A)} &&
        \Fun_{\eK}(C,A)\ar[d]^-{\Delta} \\
        K \ar[r]_-{\hat{d}} \ar[ur]^{\hat\ell} &
        \Fun_{\eK}(D,A)^J\ar[rr]_{\Fun_{\eK}(f,A)^J} &&
        {\Fun_{\eK}(C,A)^J}}
    \end{equation*}
        is an absolute right lifting diagram in $\qCat$.
\end{defn}

\begin{defn}\label{defn:pointwise-limit}
  Suppose that $A$ and $D$ are objects of the $\infty$-cosmos $\eK$, that $K$ and $J$ are simplicial sets, then we say that an absolute right lifting diagram
  \begin{equation}\label{eq:pointwise-limit}
    \xymatrix@R=2em@C=3em{
      \ar@{}[dr]|(.7){\Downarrow\lambda} &
      {A^K}\ar[d]^-{\Delta} \\
      D \ar[r]_-{d} \ar[ur]^{\ell} & {(A^K)^J}
    }    
  \end{equation}
  in $\eK$ \emph{displays a family of pointwise limits} in $A^K$ iff for each simplicial map $g\colon L\to K$ it is preserved the functor $A^g\colon A^K\to A^L$: i.e., 
      \begin{equation*}
      \xymatrix@R=2em@C=3em{
        \ar@{}[dr]|(.7){\Downarrow\lambda} &
        {A^K}\ar[d]^-{\Delta}\ar[r]^{A^g} & {A^L}\ar[d]^-{\Delta}\\
        D \ar[r]_-{d} \ar[ur]^{\ell} &
        {(A^K)^J}\ar[r]_{(A^g)^J} & {(A^L)^J}
      }
    \end{equation*}
    is an absolute right lifting diagram in $\eK$.
\end{defn}

Using these definitions: 
  
  \begin{lem}\label{lem:limits-are-limits} Consider a transposed pair of triangles
\[      \vcenter{
      \xymatrix@R=2em@C=3em{
        \ar@{}[dr]|(.7){\Downarrow\lambda} &
        {A^K}\ar[d]^-{\Delta} \\
        D \ar[r]_-{d} \ar[ur]^{\ell} & {(A^K)^J}
      }}
    \mkern30mu\leftrightsquigarrow\mkern30mu
    \vcenter{
      \xymatrix@=2em{
        \ar@{}[dr]|(.7){\Downarrow\hat\lambda} &
        \Fun_{\eK}(D,A)\ar[d]^-{\Delta} \\
        K \ar[r]_-{\hat{d}} \ar[ur]^{\hat\ell} & \Fun_{\eK}(D,A)^J
      }}
      \] as in \eqref{eq:limits-are-limits}.
 The triangle on the left is an absolute right lifting diagram that displays a family of pointwise limits if and only if 
  the transposed triangle on the right is an absolute right lifting diagram that is stable under precomposition in $\eK$.
\end{lem}
\begin{proof} The simplicially enriched cotensor/hom adjunction descends to a 2-adjunction between the homotopy 2-categories $\ho_*\eK$ and $\ho_*\qCat$. The remaining details are a straightforward exercise in adjoint transposition across a 2-adjunction.
\end{proof}

\begin{rmk}
When $K$ is a point, Lemma \ref{lem:limits-are-limits} says that $D$-indexed families of limits in an $\infty$-category $A$ can be characterized as limits in $\Fun_{\eK}(D,A)$, which are stable under pre-composition. This gives a representable characterization of limits of families of diagrams in $\infty$-cosmoi.
\end{rmk}

Our terminology of ``pointwise limit'' is explained by the following observation:

\begin{lem}[pointwise limits are determined pointwise]
  A triangle of the form displayed in~\eqref{eq:pointwise-limit} displays a family of pointwise limits in $A^K$ iff for all vertices $k\in K$ the composite triangle
  \begin{equation*}
    \xymatrix@R=2em@C=3em{
      \ar@{}[dr]|(.7){\Downarrow\lambda} &
      {A^K}\ar[d]^-{\Delta}\ar[r]^{A^k} & {A}\ar[d]^-{\Delta}\\
      D \ar[r]_-{d} \ar[ur]^{\ell} &
      {(A^K)^J}\ar[r]_{(A^k)^J} & {A^J}
    }
  \end{equation*}
  is an absolute right lifting diagram displaying a family of limits in $A$.
\end{lem}

\begin{proof}
  Using Lemma~\ref{lem:limits-are-limits} we may transform this problem into a corresponding one in the world of quasi-categories. Specifically we must show that the dual family of diagrams $\hat{d}\colon K\to \Fun_{\eK}(D,A)^J$ admits a limit that is stable under precomposition in $\eK$ if and only if for each $k\in K$ the individual diagram $dk\in\Fun_{\eK}(D,A)$ possesses a limit which is stable under precomposition. This result, however, follows easily from Proposition~\refI{prop:families.of.diagrams}.
\end{proof}

The synthetic theory of $\infty$-categories developed in \S\ref{sec:formal} a result first proven by Lurie  \cite[5.1.3.2]{Lurie:2009fk}.

\begin{lem}[preservation of limits by quasi-categorical Yoneda]\label{lem:qcat-Yoneda-pres-limits}
For any quasi-category $\qA$, the Yoneda embedding $\yoneda\colon \qA\cong\Fun_{\qCat}(1,\qA)\to {}_1\qMod(\qqCat)_{\qA}$ is fully-faithful, strongly generating and it preserves all families of small limits that exist in $\qA$.
\end{lem}

\begin{proof}
  The first two properties posited in the statement are given by Proposition~\ref{prop:yoneda.fully.faithful} and Lemma~\ref{lem:rep-ho-generate} respectively. Since Example \ref{ex:comp-mods} demonstrates that  ${}_1\qMod(\qqCat)_{\qA}$ is complete, those results allow us to apply Proposition~\ref{prop:ff.and.sg.limit.pres} to establish limit preservation.
\end{proof}

We use Lemma \ref{lem:qcat-Yoneda-pres-limits} to prove the analogous result for $\infty$-categories in an arbitrary $\infty$-cosmos which are  considerably more subtle to establish.

\begin{prop}[preservation of limits by generalised Yoneda]\label{prop:gen.yoneda.pres.lim}
  The generalised Yoneda embedding $\yoneda\colon\Fun_{\eK}(D,A)\to{}_D\qMod(\qK)_A$ preserves any family of limits which is stable under precomposition in $\eK$.
\end{prop}

By Lemma \ref{lem:limits-are-limits}, Proposition \ref{prop:gen.yoneda.pres.lim} asserts that those limit diagrams in $\Fun(D,A)$ in the $\infty$-cosmos of quasi-categories that arise from corresponding limit diagrams in $A$ in the $\infty$-cosmos $\eK$ are the ones preserved by the generalised Yoneda embedding---even despite the fact that this ``external'' Yoneda embedding is a functor of quasi-categories rather than a functor in $\eK$. The upshot is that the generalised Yoneda embedding respects the limits it recognises arise from the original $\infty$-cosmos. 

\begin{proof}
  Given a functor $f\colon C\to D$ in $\eK$ we contemplate the following diagram:
  \begin{equation}\label{eq:gen-Yoneda-pres-limits}
    \xymatrix@R=2em@C=2em{
      {\Fun_{\eK}(D,A)}\ar[r]^-{\yoneda}\ar[d] \ar@{-->}@/_17ex/[dd]_{\Fun_{\eK}(f,A)} &
      {{}_D\qMod(\qK)_A}\ar[d]\ar@{}[dl]|{\cong} \ar@{-->}@/^15ex/[dd]^{\Fun_{\eK}(C,-)_f} \\
      {\Fun_{\qCat}(\Fun_{\eK}(C,D), \Fun_{\eK}(C,A))}
      \ar[r]^-{\yoneda}\ar[d] &
      {{}_{\Fun_{\eK}(C,D)}\qMod(\qqCat)_{\Fun_{\eK}(C,A)}}\ar[d]
      \ar@{}[dl]|{\cong}\\
      {\Fun_{\eK}(C,A)\cong\Fun_{\qCat}(1, \Fun_{\eK}(C,A))} \ar[r]_-{\yoneda} &
      {{}_{1}\qMod(\qqCat)_{\Fun_{\eK}(C,A)}}
    }
  \end{equation}
  Here the upper square arises by applying Proposition~\ref{prop:gen-Yoneda-pres} to the cosmological functor $\Fun_{\eK}(C,-)\colon \eK\to\qCat$ and the lower square is constructed by applying Lemma~\ref{lem:gen-Yoneda-precomp} to the map $f\colon 1\to\Fun_{\eK}(C,D)$ that picks out the functor $f\colon C\to D$.

  By inspection we see that the left-hand vertical composite in the diagram above is simply the precomposition functor $\Fun_{\eK}(f,A)\colon\Fun_{\eK}(D,A)\to\Fun_{\eK}(C,A)$. Its right-hand vertical composite acts to carry a module $(q,p)\colon E\tfib D\times A$ to a groupoidal cartesian fibration $p_f\colon\Fun_{\eK}(C,E)_f\tfib \Fun_{\eK}(C,A)$ whose total space is the fibre given by the following pullback:
  \begin{equation*}
    \xymatrix@R=2em@C=4em{
      {\Fun_{\eK}(C,E)_f}\pbexcursion\ar[r]\ar[d]\ar@/^2.5ex/@{->>}[rr]^{p_f} &
      {\Fun_{\eK}(C,E)}\ar@{->>}[d]^{\Fun_{\eK}(C,q)} 
      \ar@{->>}[r]_{\Fun_{\eK}(C,p)} & {\Fun_{\eK}(C,A)} \\
      {1}\ar[r]_-{f} & {\Fun_{\eK}(C,D)}}
  \end{equation*}
  Correspondingly, this functor carries a module map
  \begin{equation*}
    \xymatrix@R=2em@C=1em{
      {E}\ar[rr]^g\ar@{->>}[dr]_{(q,p)} &&
      {F}\ar@{->>}[dl]^{(v,u)} \\
      & {D\times A} & \\
    }
  \end{equation*}
  to the induced map of fibres $\Fun_{\eK}(C,g)_f\colon\Fun_{\eK}(C,E)_f\to \Fun_{\eK}(C,F)_f$. Consequently, we shall use the notation $\Fun_{\eK}(C,-)_f$ to denote that right-hand vertical composite.

  Now given a module map as above we know, from
Lemma~\ref{lem:module.legs}, that the legs $q\colon E\tfib D$ and
  $v\colon F\tfib D$ are cocartesian fibrations in $\eK$, so they are carried to
  cocartesian fibrations $\Fun_{\eK}(C,q)\colon \Fun_{\eK}(C,E)\tfib
  \Fun_{\eK}(C,D)$ and $\Fun_{\eK}(C,v)\colon \Fun_{\eK}(C,F)\tfib
  \Fun_{\eK}(C,D)$ of quasi-categories by the cosmological functor
  $\Fun_{\eK}(C,-)\colon\eK\to\qCat$. It follows that we may apply
  Proposition~\ref{prop:equivalence-of-fibrations} to show that the functor
  $\Fun_{\eK}(C,g)$, which is a cartesian functor between $\Fun_{\eK}(C,q)$ and
  $\Fun_{\eK}(C,v)$, is an equivalence if and only if each of its fibres
  $\Fun_{\eK}(C,g)_f\colon \Fun_{\eK}(C,E)_f\to \Fun_{\eK}(C,F)_f$ is an
  equivalence. Furthermore equivalences are defined representably in $\eK$, that
  is $g\colon E\to F$ is an equivalence in $\eK$ if and only if
  $\Fun_{\eK}(C,g)\colon\Fun_{\eK}(C,E)\to\Fun_{\eK}(C,F)$ is an equivalence of
  quasi-categories for all objects $C\in\eK$. Combining these two facts we find
  that $g\colon E\to F$ is an equivalence of modules in ${}_D\qMod(\eK)_A$ (or
  equally an isomorphism in the quasi-category ${}_D\qMod(\qK)_A$) if and only
  if for all objects $C$ and arrows $f\colon C\to D$ the map
  $\Fun_{\eK}(C,g)_f\colon \Fun_{\eK}(C,E)_f\to \Fun_{\eK}(C,F)_f$ is an
  equivalence of quasi-categories. Thus, we have shown that the family of
  functors $$\{\Fun_{\eK}(C,-)_f\colon{}_D\qMod(\qK)_A\to
  {}_1\qMod(\qqCat)_{\Fun_{\eK}(C,A)} \mid C\in\eK, f\in\Fun_{\eK}(C,D)_0\}$$ is
  \emph{jointly conservative}. Moreover, by Remark \ref{rmk:cosmos-pres-qcat-limits}, each cosmological functor $\Fun_{\eK}(C,-)_f$
  preserves all small limits, so it follows that the family of them
  \emph{jointly reflects\/} all small limits.

  At this point we may summarise what we know about the essentially commutative square depicted in~\eqref{eq:gen-Yoneda-pres-limits} as follows:
  \begin{itemize}

  \item For each $f \colon C \to D$ in $\eK$, the the members of the family of lower horizontal Yoneda embeddings preserve any small limits that exist, by Lemma~\ref{lem:qcat-Yoneda-pres-limits}.
  \item For each $f \colon C \to D$ in $\eK$, the  left-hand vertical functor preserves any family of limits in $\Fun_{\eK}(D,A)$ that is stable under precomposition in $\eK$, directly from  Definition~\ref{defn:limit-stable-under-precomp}.
    \item As $f \colon C \to D$ is allowed to range over all functors in $\eK$ with codomain $D$, the family of right-hand vertical functors $\Fun_{\eK}(C,-)_f$ jointly reflects small limits in  ${}_D\qMod(\qK)_A$.
  \end{itemize}
  It follows that the composite of the  left-hand vertical and lower horizontal functors preserve any family of small limits in $\Fun_{\eK}(D,A)$ that are stable under precomposition in $\eK$. So the same is true for the isomorphic composites of the upper horizontal map $\yoneda\colon\Fun_{\eK}(D,A)\to{}_D\qMod(\qK)_A$ and the right-hand vertical maps. But as $f \colon C \to D$ varies over all functors with codomain $D$, that latter family of maps jointly reflects all small limits, so by cancellation we  infer that the generalised Yoneda embedding $\yoneda\colon\Fun_{\eK}(D,A)\to{}_D\qMod(\qK)_A$ preserves any such family of small limits that is stable under precomposition in $\eK$ as required.
\end{proof}

\subsection{Colimits of diagrams}\label{ssec:colimit-construction} 

It is commonplace in  category theory to study \emph{generating classes\/} for important closed classes of (co)limits. The canonical result in this regard is the construction of all limits from products and equalisers; Lurie establishes the quasi-categorical analogue in \cite[4.4.2.6]{Lurie:2009fk}. In this subsection, we establish an analogous construction for limits and colimits  in an $\infty$-category $A$ of an $\infty$-cosmos $\eK$.

\begin{rmk}\label{rmk:colim-of-diags}
  Suppose that we are given a diagram $J\colon\qJ\op\to\SSet$ which is a coproduct, pushout, or countable composite diagram in which certain connecting maps are expected to be inclusions of simplicial sets, in the sense dual to the cosmological limit types of axiom \ref{defn:cosmos}\ref{defn:cosmos:a}. For instance, the diagram \eqref{eq:skeletal-decomposition} is built out of iterating diagrams of this type. 
  
  We may take the colimit $J_\top$ of this diagram, with colimit cocone $\pi\colon J\To\Delta {J_\top}$, and observe that certain components of that cocone are inclusions as specified (dually) in \ref{defn:cosmos}\ref{defn:cosmos:b}. Then, given an $\infty$-category $A$, this then implies that the functors $\pi_c \colon A^{J_\top} \to A^{Jc}$ defined by restricting along the legs of the cone are isofibrations.
  \end{rmk}
  
  \begin{lem}\label{lem:cone-res-diagrams} Let $J_\top$ be a simplicial set defined as a coproduct, pushout of a monomorphism, or countable composite of monomorphisms of simplicial sets, presented as the colimit of a diagram $J \colon \qJ\op \to \SSet$ with colimit cone $\pi\colon J\To\Delta {J_\top}$. Suppose further that we are given an $\infty$-category $A\in\eK$ along with a fixed  diagram $d\colon D\to A^{J_\top}$ and consider the restricted diagrams  $\xymatrix@1{d_c\colon {D}\ar[r]^-{d} & A^{J_\top}\ar[r]^{A^{\pi_c}} & {A^{Jc}}}$ for each $c \in \qJ$.
  \begin{enumerate}[label=(\roman*)]
  \item\label{itm:cone-res-diag} The $\infty$-categories of cones $\Delta_{Jc}\comma d_c$ over the restricted diagrams assemble into a canonical diagram $\Delta_{J_*}\comma d_*\colon \qJ \to {}_D\qMod(\eK)_A$ of one of the cosmological limit types.
  \item\label{itm:cone-res-diag-cone} Moreover there is a canonical cosmological limit type cone over this diagram in ${}_D\qMod(\eK)_A$ whose summit is the $\infty$-category $\Delta_{J_\top}\comma d$ of cones over $d$.
  \end{enumerate}
  \end{lem}
\begin{proof}
  Given an arrow $f\colon c\to c'$ in $\qJ$ we may construct a transformation of diagrams
  \begin{equation*}
    \xymatrix@R=2em@C=3em{
      {D}\ar@{=}[d]\ar[r]^{d_c} & {A^{Jc}}\ar[d]^-{A^{J(f)}} &
      {A}\ar[l]_-{\Delta_{Jc}}\ar@{=}[d] \\
      {D}\ar[r]_{d_{c'}} & {A^{Jc'}} & 
      {A}\ar[l]^-{\Delta_{Jc'}}
    }
  \end{equation*}
  which induces a unique map $\comma(A,A^{J(f)},D)\colon \Delta_{Jc}\comma d_{c}\to \Delta_{Jc'}\comma d_{c'}$ as discussed in Proposition~\ref{prop:trans-comma}. This construction is clearly functorial, and so provides us with a diagram $\Delta_{J_*}\comma d_{*}\colon\qJ\to {}_{D}\qMod(\eK)_{A}$ as stipulated in \ref{itm:cone-res-diag}.  What is more if $J(f)\colon Jc'\inc Jc$ is an inclusion of simplicial sets then $A^{J(f)}\colon A^{Jc}\tfib A^{Jc'}$ is an isofibration and it follows, by Proposition~\ref{prop:trans-comma}, that $\comma(A,A^{J(f)},D)\colon \Delta_{Jc}\comma d_{c}\tfib \Delta_{Jc'}\comma d_{c'}$ is also an isofibration. So this diagram satisfies the isofibration condition required of the cosmological limit type diagrams of \ref{defn:cosmos}\ref{defn:cosmos:a}.
  
  Equally each transformation
  \begin{equation*}
    \xymatrix@R=2em@C=3em{
      {D}\ar@{=}[d]\ar[r]^{d} & {A^{J_\top}}\ar[d]^-{A^{\pi_c}} &
      {A}\ar[l]_-{\Delta_{J_\top}}\ar@{=}[d] \\
      {D}\ar[r]_{d_{c}} & {A^{Jc}} & 
      {A}\ar[l]^-{\Delta_{Jc}}
    }
  \end{equation*}
  induces a projection $\comma(A,A^{\pi_c},D)\colon \Delta_{J_\top}\comma d\to \Delta_{Jc}\comma d_c$, and this family provides us with a cone $\Delta_{J_\top}\comma d \To \Delta_{J_*}\comma d_{*}$.  For the same reason, the monomorphic legs of the colimit cone $\pi$ convert to isofibration legs in the cone in ${}_D\qMod(\eK)_{A}$, and hence this cone is also a cosmological limit type cone.
\end{proof}

As strongly suggested by the set up of Lemma \ref{lem:cone-res-diagrams}, we can in fact show that the cones just constructed are limit cones. 

\begin{lem}\label{lem:colim-of-diags}
  The cone of Lemma \ref{lem:cone-res-diagrams} displays $\Delta_{J_\top}\comma d$ as a limit of the diagram $\Delta_{J_*}\comma d_{*}$ in ${}_{D}\qMod(\eK)_{A}$.
\end{lem}

\begin{proof}
  The functor $(-)^\cattwo\colon\eK\to\eK^\cattwo$, which carries each object $A$ to the associated isofibration $A^\cattwo\tfib A\times A$, is a cosmological functor. In particular, it preserves the cosmological limit types, a fact which follows easily from the observations that these limit types are jointly created by $\dom,\cod\colon\eK^\cattwo\to\eK$ and that they commute with the limits $A\times A$ and $A^\cattwo$ in $\eK$. Furthermore the functor $A^{(-)}\colon\SSet\op\to\eK$ carries coproducts, pushouts of inclusions, and countable composites of inclusions to the corresponding limit types in $\eK$. It follows, therefore, that these functors carry the diagram $J\colon\qJ\op\to\SSet$ and its colimiting cocone $\pi\colon J\To \Delta_{J_\top}$ to a diagram $(A^{J_*})^\cattwo\colon \qJ\to \eK^\cattwo$ and a limit cone with apex $(A^{J_\top})^\cattwo\tfib A^{J_\top}\times A^{J_\top}$, and that these satisfy the conditions specified in \ref{defn:cosmos}\ref{defn:cosmos:b} in the $\infty$-cosmos $\eK^\cattwo$.

  We may apply the construction of Proposition~\refVII{prop:K^2-cosmos-limits} to the limit derived in the last paragraph---pulling back the codomains of the arrows in the image of the diagram in $\eK^\cattwo$ to the codomain of the limit object---to give a diagram $A^{\pi_*}\comma A^{\pi_*}\colon \qJ\to \eK_{/A^{J_\top}\times A^{J_\top}}$ which maps each object $c\in\qJ$ to the isofibration $A^{\pi_c}\comma A^{\pi_c}\tfib A^{J_\top}\times A^{J_\top}$ and a cone which displays $(A^{J_\top})^\cattwo\tfib A^{J_\top}\times A^{J_\top}$ as its limit in the $\infty$-cosmos $\eK_{/A^{J_\top}\times A^{J_\top}}$ of one of the cosmological limit types. Finally we may pull his latter diagram and limit cone back along the arrow $d\times\Delta_{J_\top}\colon D\times A\to A^{J_\top}\times A^{J_\top}$ to give the diagram and cone in $\eK_{/D\times A}$. The required result follows on recalling that pullback along that arrow determines a cosmological functor, which therefore carries our limit to a limit in the $\infty$-cosmos $\eK_{/D\times A}$ and hence in its full subcategory ${}_D\qMod(\eK)_A$ as required.
\end{proof}

We shall show that the limits we have constructed in Lemma \ref{lem:colim-of-diags} give rise to corresponding limits in the quasi-category of modules ${}_D\qMod(\qK)_A$.

\begin{cor}\label{cor:colim-of-diags}
The limit derived in Lemma~\ref{lem:colim-of-diags} presents the module $\Delta_{J_\top}\comma d$ as a limit of the diagram $\Delta_{J*}\comma d_{*}$ of shape $\qJ$ in the quasi-category of modules ${}_D\qMod(\qK)_A$.
\end{cor}

\begin{proof}
By Theorem \ref{thm:nerve-completeness} applied in the context of Example \ref{ex:comp-mods}, limits in the quasi-category of modules ${}_D\qMod(\qK)_A$ are constructed as pseudo homotopy limits in the Kan-complex-enriched category ${}_D\qMod(\eK)_A$. In the case of the diagrams under consideration at present---products, pullbacks, and limits of towers of isofibrations---Proposition \ref{prop:strict-pseudo-cones} reveals that such limits are given by the strict pseudo cones of Definition \ref{defn:strict-pseudo-cones} formed from the corresponding 1-categorical limits cones. These coincide with the cones constructed in Lemma \ref{lem:cone-res-diagrams}.
\end{proof}

\begin{obs}\label{obs:internalized-limit-diag}
Continuing in the context of Lemma \ref{lem:cone-res-diagrams}, now assume for each $c \in \qJ$ that the restricted diagram  $d_c \colon \xymatrix@1{{D}\ar[r]^-{d} & A^{J_\top}\ar[r]^{A^{\pi_c}} & {A^{Jc}}}$  admits a limit $\ell_c\colon D\to A$ as displayed by the following absolute right lifting diagram:
  \begin{equation}\label{eq:colim-of-diags}
    \vcenter{\xymatrix@=2.5em{
        {}\ar@{}[dr]|(0.7){\Downarrow\lambda_c} &
        {A}\ar[d]^{\Delta_{Jc}} \\
        {D}\ar[r]_{d_c}\ar[ur]^{\ell_c} & {A^{Jc}}
      }}
    \mkern40mu\leftrightsquigarrow\mkern40mu
    \vcenter{\xymatrix@R=2em@C=1em{
        {A\comma \ell_c}\ar[rr]^-{\simeq}
        \ar@{->>}[dr] && {\Delta_{Jc}\comma d_c}\ar@{->>}[dl] \\
        & {D\times A} &
      }}
  \end{equation}
  Equivalently, by Proposition~\ref{prop:limit-as-module-equivalence}, we know that $\ell_c$ provides a fibred equivalence of modules $\Delta_{Jc}\comma d_c$ as depicted on the right of the display above. These are otherwise isomorphisms in the quasi-category ${}_D\qMod(\qK)_A$ of modules between the vertices of the diagram $\Delta_{J*}\comma d_{*}\colon\qJ\to {}_D\qMod(\qK)_A$ and the covariantly represented modules $A\comma\ell_c$. Transferring the arrows in the $\qJ$-shaped diagram $\Delta_{J*}\comma d_{*}$ constructed in Lemma \ref{lem:cone-res-diagrams} along these pointwise isomorphisms, we may extend them to a natural isomorphism whose domain is a functor $A\comma\ell_{*}\colon\qJ\to {}_D\qMod(\qK)_A$ extending our family of represented modules.

  A represented module $A\comma\ell_c$ is simply the image of the functor $\ell_c\colon D\to A$ under the generalised Yoneda embedding $\yoneda\colon \Fun_{\eK}(D,A)\to {}_D\qMod(\qK)_A$. Furthermore generalised Yoneda is fully-faithful, by Proposition~\ref{prop:yoneda.fully.faithful}, so it follows that the functor $A\comma\ell_{*}$ factors through it to endow the family of covariant representatives $\ell_c$ with the structure of a functor $\ell_{*}\colon \qJ\to \Fun_{\eK}(D,A)$. We may summarise these various functors and their relationships in the following diagram:
  \begin{equation*}
    \xymatrix@R=2.5em@C=3em{
      {\qJ}\ar@/^2ex/[rr]!L(0.75)^-{\Delta_{J*}\comma d_{*}}
      \ar@/_2ex/[rr]!L(0.75)_-{A\comma\ell_{*}}\ar[dr]_{\ell_{*}} &
      \rotatebox[origin=c]{90}{$\scriptstyle\cong$}
      \ar@{}[d]|(0.6){\rotatebox[origin=c]{90}{$\scriptstyle\cong$}} &
      {{}_D\qMod(\qK)_A} \\
      & {\Fun_{\eK}(D,A)}\ar[ur]_{\yoneda} &
    }
  \end{equation*}
  When expressed in terms lifting properties this functor $\ell_{*}\colon\qJ\to\Fun_{\eK}(D,A)$ carries an arrow $f\colon c\to c'$ in $\qJ$ to a representative of the unique 2-cell $\ell_f$ induced by the lifting property of the right hand triangle in the following diagram: 
  \begin{equation}\label{eq:comma-diagrams}
    \vcenter{\xymatrix@=2.5em{
        {}\ar@{}[dr]|(0.7){\Downarrow\lambda_c} &
        {A}\ar[d]^{\Delta_{Jc}}\ar@{=}[r] & {A}\ar[d]^{\Delta_{Jc'}} \\
        {D}\ar[r]_{d_c}\ar[ur]^{\ell_c} &
        {A^{Jc}}\ar[r]_{A^{J(f)}} & {A^{Jc'}}
      }}
    \mkern40mu=\mkern40mu
    \vcenter{\xymatrix@=2.5em{
        {} & {A}\ar@{=}[r]\ar@{}[dr]|(0.7){\Downarrow\lambda_{c'}}
        \ar@{}[d]|{\overset{\ell_f}\To}&
        {A}\ar[d]^{\Delta_{Jc'}} \\
        {D}\ar@{=}[r]\ar[ur]^{\ell_c} & {D}\ar[r]_{d_{c'}}
        \ar[ur]^(0.4)*!<6pt,0pt>{\scriptstyle\ell_{c'}} & {A^{Jc'}}
      }}
  \end{equation}

\end{obs}
Note that the functor of quasi-categories $\ell_{*}\colon\qJ\to\Fun_{\eK}(D,A)$ may dually be regarded as being a functor $\ell_{*}\colon D\to A^{\qJ}$ inside the $\infty$-cosmos $\eK$. This all leads us to the following result:

\begin{prop}\label{prop:colim-of-diags}
Let $J_\top$ be a simplicial set defined as a coproduct, pushout of a monomorphism, or countable composite of monomorphisms of simplicial sets, presented a colimit of a diagram $J \colon \qJ\op \to \SSet$. Consider a fixed diagram $d\colon D\to A^{J_\top}$ in an $\infty$-category $A$ with the property that the restricted diagrams  $\xymatrix@1{d_c\colon {D}\ar[r]^-{d} & A^{J_\top}\ar[r]^{A^{\pi_c}} & {A^{Jc}}}$ for each $c \in \qJ$ have limits $\ell_c \colon D \to A$.

Then a functor $\ell\colon D\to A$ is a limit of the diagram $d\colon D\to A^{J_\top}$ if and only if it is a limit of the diagram $\ell_{*}\colon D\to A^{\qJ}$ formed from the limit functors in Observation \ref{obs:internalized-limit-diag}.
\end{prop}

\begin{proof}
Suppose first that $\ell \colon D \to A$ is a limit of the diagram $\ell_*\colon D \to A^{\qJ}$. By Lemma \ref{lem:limits-are-limits}, this is the case if and only if $\ell \colon 1 \to \Fun_{\eK}(D,A)$ is the limit of the dual diagram $\ell_* \colon 1 \to \Fun_{\eK}(D,A)^{\qJ}$ of quasi-categories and this limit is stable under precomposition. Proposition~\ref{prop:gen.yoneda.pres.lim} tells us that this limit is preserved by the generalised Yoneda embedding $\yoneda\colon\Fun_{\eK}(D,A)\to{}_D\qMod(\qK)_{A}$. By construction, however, the diagram $\yoneda\ell_{*}\colon \qJ\to{}_D\qMod(\qK)_A$  is isomorphic to the diagram $A\comma\ell_*\colon \qJ\to{}_D\qMod(\qK)_A$ which is in turn isomorphic to the diagram $\Delta_{J*}\comma d_{*}\colon \qJ\to{}_D\qMod(\qK)_A$, as depicted in~\eqref{eq:comma-diagrams}; it follows that the limits of the diagrams $\yoneda\ell_{*}$ and $\Delta_{J*}\comma d_{*}$ are isomorphic when they exist. In this case the limit of the former diagram is $\yoneda\ell\cong A\comma \ell$, since generalised Yoneda preserves the limit $\ell$, and the limit of the latter is $\Delta_{J_\top}\comma d$, as demonstrated in Corollary~\ref{cor:colim-of-diags}, so the resulting isomorphism $A\comma \ell \cong \Delta_{J_\top}\comma d$ in the quasi-category ${}_D\qMod(\qK)_A$ provides an equivalence $A\comma \ell \simeq \Delta_{J_\top}\comma d$ over $D\times A$ which presents $\ell$ as a limit of $d\colon D\to A^{J_\top}$.

Conversely, suppose $\ell \colon D \to A$ is a limit of $d \colon D \to A^{J_\top}$ and  recall again that the diagrams $\yoneda\ell_*$ and $\Delta_{J*}\comma d_{*}$ are isomorphic and so the limit $\Delta_{J_\top}\comma d$ of the second of these diagrams in ${}_D\qMod(\qK)_A$, as supplied by Corollary~\ref{cor:colim-of-diags}, is also a limit of the first diagram. However our assumption that $\ell$ is a limit of the family of diagrams $d\colon D\to A^{J_\bot}$ may otherwise be read as saying that this limit is covariantly represented by the functor $\ell$; in other words, we have shown that the diagram $\yoneda\ell_{*}\colon \qJ\to{}_D\qMod(\qK)_A$ has the module $\yoneda\ell$ as its limit in the quasi-category of modules ${}_D\qMod(\qK)_A$. Since the generalised Yoneda embedding $\yoneda\colon\Fun_{\eK}(D,A)\to{}_D\qMod(\qK)_A$ is fully faithful by Proposition \ref{prop:yoneda.fully.faithful},  the cone presenting this limit factors through the generalised Yoneda embedding  to give a cone in $\Fun_{\eK}(D,A)$ that displays $\ell$ as a limit of $\ell_{*}\colon \qJ\to\Fun_{\eK}(D,A)$. Now Lemma~\ref{lem:limits-are-limits} tells us that we are done if we can show that this limit is stable under precomposition, or in other words that it is preserved by the precomposition functor $\Fun_{\eK}(f,A)\colon \Fun_{\eK}(D,A)\to\Fun_{\eK}(C,A)$ associated with any functor $f\colon C\to D$ in $\eK$. To see this, we use the essentially commutative diagram
  \begin{equation*}
    \xymatrix@R=2em@C=3em{
      {\Fun_{\eK}(D,A)}\ar[r]^-{\yoneda}\ar[d]_{\Fun_{\eK}(f,A)}
      \ar@{}[dr]|{\cong} & {{}_D\qMod(\qK)_A}\ar[d]^{(f\times A)^*} \\
      {\Fun_{\eK}(C,A)}\ar[r]_-{\yoneda} & {{}_C\qMod(\qK)_A}
    }
  \end{equation*}
  of Lemma~\ref{lem:gen-Yoneda-precomp}. The limit under consideration is preserved by the top Yoneda embedding (since it was reflected from ${}_D\qMod(\qK)_A$), and preserved by the right-hand vertical functor by Remark \ref{rmk:cosmos-pres-qcat-limits}. It is then reflected by the bottom Yoneda embedding by Proposition \ref{prop:fully-faithful-reflects}, since this functor is fully faithful, and thus it must be preserved by the left-hand vertical as required.
\end{proof}

\begin{thm}[limit constructions]\label{thm:limit-construction}
  Suppose that $\kappa$ is a regular cardinal and that $A$ is an $\infty$-category in an $\infty$-cosmos $\eK$ that admits  products of cardinality $<\kappa$ and pullbacks. If $X$ is a $\kappa$-presentable simplicial set then $A$ admits all limits of diagrams of shape $X$.
\end{thm}

\begin{rmk}
  Here the phrase {\em admits all limits of diagrams of shape $X$\/} should be taken to mean that the diagonal functor $\Delta_X\colon A\to A^X$ admits a right adjoint as in Definition \ref{defn:all-limits}, which is equivalent to  postulating that every family of diagrams $d\colon D\to A^X$ admits a limit. In $\qCat$, Corollary \refI{cor:pointwise} tells us that this is equivalent to postulating that every diagram $d \colon 1 \to A^X$ has a limit, because the quasi-category 1 acts as a ``generator'' for the $\infty$-cosmos $\qCat$ in a suitable sense, but this reduction to the case $D=1$ is not possible in all $\infty$-cosmoi.
\end{rmk}

\begin{proof}
  Our proof proceed by induction on the skeleta of the simplicial set $X$. We note first that a simplicial set is $\kappa$-presentable if and only if it has a set of non-degenerate simplices of cardinality $<\kappa$.

  When $X$ is $0$-skeletal it comprises a disjoint set of vertices of cardinality $<\kappa$ and so it follows that the limit of any family of diagrams $d\colon D\to A^X$ exists by our assumption that $A$ admits all products of cardinality $<\kappa$.

  So fix a natural number $n$ and adopt the inductive hypothesis that the result of the statement holds for all $(n-1)$-skeletal $\kappa$-presentable simplicial sets. Suppose that $X$ is a $n$-skeletal $\kappa$-presentable simplicial set, then we may express it as a pushout
  \begin{equation}\label{eq:po-of-skeleta}
    \xymatrix@R=2.5em@C=2em{
      {\coprod_{L_nX} \boundary \Delta^n}\ar@{^(->}[r]\ar[d] &
      {\coprod_{L_nX}  \Delta^n}\ar[d] \\
      {\sk_{n-1}X}\ar@{^(->}[r] & {X}\poexcursion 
    }
  \end{equation}
  in which $L_{n}X$ denotes the set of non-degenerate $n$-simplices of $X$, which has cardinality $<\kappa$.

  The diagonal $A\to A^{\Del^n}$ always admits a right adjoint given by precomposition with the map $\Del^{\fbv{0}}\inc\Del^{n}$. In other words $A$ admits all limits of diagrams of shape $\Del^n$ and these are given by evaluating at $0$; see Propositions \refV{prop:initial-limits} and Lemma \refV{lem:adj-initial}. It follows immediately that the functor $A^{L_{n}X}\to A^{\coprod_{L_nX}  \Delta^n}$ determined by precomposition with the projection $\pi\colon \coprod_{L_nX}  \Delta^n\cong L_{n}X\times\Del^{n}\to L_{n}X$ also admits a right adjoint. What is more the diagonal $A\to A^{L_{n}X}$  admits a right adjoint, by our hypothesis that $A$ possesses all products of cardinality $<\kappa$, and composing these adjunctions 
  \[
\xymatrix@C=3em{ A \ar@/^/[r]^-\Delta \ar@{}[r]|-\bot& A^{L_{n}X} \ar@/^/[r]^-\Delta \ar@{-->}@/^/[l] \ar@{}[r]|-\bot& A^{\coprod_{L_nX}  \Delta^n} \ar@{-->}@/^/[l]}
\]
  we see that $A$ admits all limits of shape $\coprod_{L_nX}  \Delta^n$. Furthermore the simplicial sets $\sk_{n-1}X$ and $\coprod_{L_nX}  \boundary\Delta^n$ are both $(n-1)$-skeletal and $\kappa$-presentable so the inductive hypothesis suffices to show that $A$ admits all limits of diagrams of those shapes. In this way we have established the hypothesis required to apply Proposition~\ref{prop:colim-of-diags} to the diagram whose pushout is depicted in~\eqref{eq:po-of-skeleta} to infer that $A$ admits all limits of diagrams of shape $X$ as required.

  It remains only to prove that this result also holds when $X$ is a $\kappa$-presentable simplicial set which is not $n$-skeletal for any $n$. In that case $\kappa > \omega$, because a finitely presentable simplicial set is always $n$-skeletal for some $n$, and it follows that $A$ admits all limits of countable sequences since their diagram shape is a $2$-skeletal and $\kappa$-presentable simplicial set. Now the simplicial set $X$ may be expressed as the countable composite of its skeleta inclusions
  \begin{equation}\label{eq:seq-of-skeleta}
    \xymatrix@R=0em@C=2.5em{
      {\sk_0X}\ar@{^(->}[r] & {\sk_1X} \ar@{^(->}[r] &
      {\sk_1X} \ar@{^(->}[r] & \cdots
      \ar@{^(->}[r] &
      {\sk_{n}X}\ar@{^(->}[r] & {\sk_{n+1}X}\ar@{^(->}[r] &\cdots
    }
  \end{equation}
  and each of these skeleta is $\kappa$-presentable; furthermore, the inductive argument above also applies to show that $A$ admits all limits of diagrams of shape $\sk_{n}X$. In this way we have established the hypothesis required to apply Proposition~\ref{prop:colim-of-diags} to the diagram depicted in~\eqref{eq:seq-of-skeleta} to infer that $A$ admits all limits of diagrams of shape $X$ as required.
\end{proof}

For the reader's convenience, we explicitly derive the dual: 
\begin{thm}[colimit constructions]\label{thm:colimit-construction}
  Suppose that $\kappa$ is a regular cardinal and that $A$ is an $\infty$-category in an $\infty$-cosmos $\eK$ that admits coproducts of cardinality $<\kappa$ and pushouts. If $X$ is a $\kappa$-presentable simplicial set then $A$ admits all colimits of diagrams of shape $X$.
 \end{thm}
 \begin{proof}
 Colimits of $X$-indexed diagrams valued in an $\infty$-category $A$  in an $\infty$-cosmos $\eK$ coincide with limits of $X\op$-indexed diagrams in $A$ in the $\infty$-cosmos $\eK\co$ of Definition \ref{defn:dual-cosmoi}. Thus Theorem \ref{thm:limit-construction} applies.
 \end{proof}


  \bibliographystyle{special}
  \bibliography{../../common/index}


\end{document}